\documentclass[reqno]{amsart}
\usepackage{graphicx, amsmath,amsfonts,amsthm,amssymb, paralist, fullpage}
\usepackage{hyperref}

\newtheorem{theorem}{Theorem}[section]
\newtheorem{corollary}[theorem]{Corollary}
\newtheorem{lemma}[theorem]{Lemma}

\theoremstyle{definition}

\theoremstyle{remark}
\newtheorem{remark}[theorem]{Remark}
\numberwithin{equation}{section}

\newcommand{\g}{\geqslant}

\newcommand{\RR}{\mathbb{R}}\newcommand{\R}{\mathbb{R}}
\newcommand{\ZZ}{\mathbb{Z}}
\newcommand{\CC}{\mathbb{C}}

\newcommand{\NN}{\mathbb{N}}
\newcommand{\p}{\partial}
\newcommand{\q}{\varphi}
\newcommand{\les}{\leqslant}
\newcommand{\lesa}{\lesssim}

\newcommand{\mc}[1]{\mathcal{#1}}

\newcommand{\eref}[1]{(\ref{#1})}
\newcommand{\lr}[1]{ \langle #1 \rangle}
\newcommand{\ind}{\mathbbold{1}}

\DeclareSymbolFont{bbold}{U}{bbold}{m}{n}
\DeclareSymbolFontAlphabet{\mathbbold}{bbold}

\DeclareMathOperator*{\supp}{supp}

\DeclareMathOperator*{\dist}{dist}

\newcommand{\C}{\mathbb{C}}
\newcommand{\N}{\mathbb{N}}

\newcommand{\D}{\mathcal{D}}

\newcommand{\Z}{\mathbb{Z}}

\newcommand{\EQ}[1]{\begin{equation}\begin{split} #1 \end{split}\end{equation}}
\newcommand{\De}{\Delta}
\newcommand{\pq}{\quad}

\newcommand{\BR}[1]{\left[#1\right]}
\newcommand{\pt}{&}
\newcommand{\pr}{\\ &}

\newcommand{\I}{\infty}
\newcommand{\CAS}[1]{\begin{cases} #1 \end{cases}}
\newcommand{\fy}{\varphi}
\newcommand{\lec}{\lesssim}

\newcommand{\fg}{\mathfrak{g}}
\newcommand{\e}{\varepsilon}

\newcommand{\Ga}{\Gamma}
\newcommand{\de}{\delta}

\newcommand{\ta}{\tau}

\newcommand{\si}{\sigma}
\newcommand{\x}{\xi}
\newcommand{\etc}{,\ldots,}
\newcommand{\la}{\lambda}
\newcommand{\ti}{\widetilde}
\newcommand{\be}{\beta}
\newcommand{\cS}{\mathcal{S}}
\newcommand{\PN}[1]{P^N_{#1}}
\newcommand{\PF}[1]{P^F_{#1}}
\newcommand{\na}{\nabla}
\newcommand{\pn}{}

\newcommand{\LR}[1]{{\langle #1 \rangle}}
\newcommand{\cI}{\mathcal{I}}

\setdefaultenum{(i)}{(a)}{1}{A}

\begin{document}

\title[GWP for the energy-critical Zakharov system]{Global
  wellposedness for the energy-critical Zakharov system below the ground
  state}

\author[T.~Candy]{Timothy Candy}
\address[T.~Candy]{Department of Mathematics and Statistics, University of Otago, PO Box 56, Dunedin 9054, New Zealand}
\email{tcandy@maths.otago.ac.nz}

\author[S.~Herr]{Sebastian Herr}
\address[S.~Herr]{Fakult\"at f\"ur
  Mathematik, Universit\"at Bielefeld, Postfach 10 01 31, 33501
  Bielefeld, Germany}
\email{herr@math.uni-bielefeld.de}

\author[K.~Nakanishi]{Kenji Nakanishi}
\address[K.~Nakanishi]{Research Institute for Mathematical Sciences, Kyoto University, Kyoto 606-8502, Japan}
\email{kenji@kurims.kyoto-u.ac.jp}

\begin{abstract}
The Cauchy problem for the Zakharov system in the energy-critical dimension $d=4$ is
considered. We prove that global well-posedness holds in the full (non-radial) energy space for any initial data with energy and wave mass below
the ground state threshold. The result is based on a Strichartz estimate for the Schr\"odinger equation with a potential. More precisely, a Strichartz estimate is proved to hold uniformly for any potential solving the free wave equation with mass below the ground state constraint. The key new ingredient is a bilinear (adjoint) Fourier restriction estimate for solutions of the inhomogeneous Schr\"odinger equation with forcing in dual endpoint Strichartz spaces.
\end{abstract}

\keywords{Zakharov system, global well-posedness, ground state,
  Fourier restriction}
\subjclass[2010]{Primary: 35Q55. Secondary: 42B37, 35L70}

\maketitle


\section{Introduction}

In 1972 Zakharov introduced a set of equations modeling the dynamics of Langmuir waves in plasma physics \cite{zakharov_collapse_1972,Sulem1999,Colin2004,Texier2007}. These are rapid oscillations of an electric field in a conducting plasma. A simplified scalar version of this model, known as the Zakharov system, is given by the coupled system of Schr\"odinger-wave equations
    \begin{equation}\label{eq:Zakharov}
        \begin{split}
           i \p_t u + \Delta u &= v u, \\
          \frac{1}{\alpha} \p_t^2 v - \Delta v &= \Delta |u|^2,
        \end{split}
    \end{equation}
where $u(t,x): \mathbb{R} \times \mathbb{R}^d \to \mathbb{C}$ is the complex envelope of the electric field, $v(t,x): \mathbb{R} \times \mathbb{R}^d \to \mathbb{R}$ is the ion density fluctuation, and the ion sound speed $\alpha>0$ is a fixed constant. The well-posedness of the Zakharov system has attracted considerable attention in the mathematical literature over the past 30 years, see for instance  \cite{Schochet-Weinstein,added1988,ozawa_nonlinear_1992,Kenig1995,Mas08, Guo2018} and the reference therein, in part due to its close relationship with the cubic nonlinear Schr\"odinger equation which formally appears by taking the limit $\alpha \to \infty$.

For the purposes of this paper, it is slightly more convenient to replace \eqref{eq:Zakharov} with the corresponding first order system. More precisely, writing $V := v-i|\na|^{-1}\p_t v$, we see that \eqref{eq:Zakharov} (with $\alpha=1$) is equivalent to
\EQ{ \label{eqn:Zakharov first order}
 \pt i\p_t u + \De u = \Re(V)u,
 \pr i\p_t V + |\na| V = -|\na||u|^2}
and our goal is to study the Cauchy problem with initial data
    \begin{equation}\label{eq:id}
      u(0) = f \in H^s(\RR^d;\C), \qquad \qquad V(0) = g \in H^\ell(\RR^d;\R).
    \end{equation}

The Zakharov system \eqref{eqn:Zakharov first order} is a Hamiltonian system, and in particular, sufficiently regular solutions preserve the (Schr\"odinger) mass and energy
 \EQ{\label{eq:mass-energy}
 M(u):=\int_{\R^d}|u|^2dx, \pq E_Z(u,V):= \int_{\R^d} \frac{1}{2} |\na
 u|^2+\frac{1}{4}|V|^2+\frac{1}{2} \Re(V)|u|^2dx.}
Unfortunately, the energy is sign indefinite, and thus in general does not give any a priori control over the $\dot{H}^1$ norm of $u$. On the other hand, in dimension $d=4$, the indefinite term in the energy can be
controlled by the critical Sobolev embedding
\[
\Big| \int_{\R^4} \Re(V)|u|^2dx\Big| \les  \|V(t)\|_{L^2(\R^4)}\|u(t)\|_{L^4(\R^4)}^2\les
C
\|V(t)\|_{L^2(\R^4)}\|\nabla u(t)\|_{L^2(\R^4)}^2,
\]
see \eqref{eq:cs} below for the sharp constant $C$. In particular, when $d=4$, the kinetic and potential energy have the same scaling, and thus the Zakharov system is energy critical in dimension $d=4$ with energy space corresponding to the regularity $(s,\ell)=(1,0)$.

The energy critical Zakharov system is closely connected to the focusing cubic nonlinear Schr\"odinger equation (NLS)
\begin{equation}\label{eq:nls}
i\partial_t u +\Delta u =-|u|^2u,
\end{equation}
which is also energy-critical in dimension $d=4$ in the sense of invariant scaling. For instance, the cubic NLS \eqref{eq:nls} arises as the subsonic limit $\alpha \to \infty$ of the Zakharov system
\eqref{eq:Zakharov} as rigorously proved in \cite{Schochet-Weinstein,added1988,ozawa_nonlinear_1992,Kenig1995,Mas08}. Furthermore, the ground state for the NLS plays a key role in the global dynamics of the Zakharov system. More precisely, recall the Aubin-Talenti function
    \EQ{W(x)=\big(1+|x|^2/8\big)^{-1},}
and let $W_\la:=\la W(\la x)$ for $\la>0$. The function $W$ is the (non-square integrable) ground state of the focusing cubic NLS \eqref{eq:nls} on $\R^4$, and in particular satisfies
    \EQ{-\De W_\la=W_\la^3 \quad \text{ in }\R^4.}
Thus taking $u = W_\la$ gives a family of static solutions to the cubic NLS \eqref{eq:nls}. It is also an extremiser of the energy-critical Sobolev inequality: For any $\varphi \in \dot H^1(\R^4)$, we have
    \EQ{\label{eq:cs} \|\fy\|_{L^4(\R^4)} \le \frac{\|W\|_{L^4(\R^4)}}{\|\na W\|_{L^2(\R^4)}}\|\na\fy\|_{L^2(\R^4)}.}
In seminal work \cite{Kenig2006} Kenig and Merle proved that the ground state $W$ defines the sharp threshold for global well-posedness, scattering and blow-up in the case of the focusing
cubic NLS on $\R^4$ in the radial setting. This has recently been extended by Dodson \cite{Dodson2019} to the general case.

Returning to the Zakharov system, we observe that the family of static solutions $W_\la$ to the NLS, also gives rise to a family of static solutions for the Zakharov system \eqref{eqn:Zakharov first order} on $\RR^4$ by simply taking $(u,V)=(W_\la,-W_\la^2)$. In particular this shows that there certainly exists non-dispersing solutions to \eqref{eqn:Zakharov first order}. Moreover, we have
 \[E_Z(W,-W^2)=E_S(W)=\frac14\|W^2\|_{L^2}^2, \text{ where }E_S(u):=\int_{\R^d}\frac12 |\nabla u|^2-\frac14 |u|^4dx\] is the energy of NLS.
We refer to the pairs $(W_\la,-W_\la^2)$ as the ground states of the Zakharov system, as similar to the focussing cubic NLS, they play a crucial role in the global dynamics of
\eqref{eq:Zakharov}. Indeed, recently \cite{Guo2018} proved the following,  which is analogous to \cite{Kenig2006}: Under the energy constraint
\EQ{
 E_Z(u,V) <  E_Z(W,-W^2),}
the energy space $(u,V)\in H^1(\R^4)\times L^2(\R^4)$ is topologically split into the two domains
\[
 \{\|V\|_{L^2(\R^4)} < \|W^2\|_{L^2(\R^4)}\} \text{ and }  \{\|V\|_{L^2(\R^4)}>\|W^2\|_{L^2(\R^4)}\}, \]
although the wave $L^2$ norm is not conserved by the Zakharov flow. Further more, all radial solutions in the former domain are global and scattering, while those in the latter can not be global and bounded in the energy space. We refer the reader to the introduction  of \cite{Guo2018} for more details on the connection between the NLS and the Zakharov equation \eqref{eqn:Zakharov first order}.

Regarding the small data theory, in our recent paper \cite{Candy2023} we identified the optimal regularity range for $(s, \ell)$ such that the Cauchy problem for the Zakharov system is well-posed for $d\g4$.
More precisely, for any $(s,\ell)\in \RR^2$  satisfying
   \begin{equation}\label{eqn:cond on s l} \ell \g \frac{d}{2}-2,
   \qquad \max\Big\{\ell-1, \frac{\ell}{2} + \frac{d-2}{4}\Big\}
   \les s \les \ell + 2, \quad (s,\ell) \not = \Big(\frac{d}{2},
   \frac{d}{2}-2\Big),
   \Big(\frac{d}{2},\frac{d}{2}+1\Big)\end{equation}
the Cauchy problem  \eqref{eq:Zakharov}--\eqref{eq:id} for $d\g 4$ is well-posed. In addition, we proved global well-posedness and scattering provided
 that the Schr\"odinger initial datum  $\|f\|_{H^{\frac{d-3}{2}}(\RR^d)}$ is sufficiently small. Notice that in dimension $d=4$, the energy space regularity $(s,\ell)=(1,0)$ is admissible and lies on the boundary of this region. This reflects the fact that the Zakharov system is energy critical in $d=4$. We refer the reader to the introduction of \cite{Candy2023} for a thorough account on the well-posedness problem and references to earlier work. \\

The aim of this paper is to prove global well-posedness in the former domain for $s\ge 1$, which includes the energy space without the radial symmetry assumption. Our main result is the following

 \begin{theorem}[GWP below the ground state]\label{thm:gwp below W}
Let $d=4$ and $(s,\ell)\in\R^2$ satisfy \eqref{eqn:cond on s l} and $s\ge 1$.
Then for any initial data $(u(0),V(0))\in H^s(\R^4)\times H^\ell(\R^4)$ satisfying
\EQ{
 E_Z(u(0),V(0)) <  E_Z(W,-W^2), \pq \|V(0)\|_{L^2(\R^4)} \les \|W^2\|_{L^2(\R^4)},}
the Zakharov system \eqref{eqn:Zakharov first order} has a unique global solution $(u,V)\in
C(\R;H^s(\R^4)\times H^\ell(\R^4))$ satisfying $ u \in L^2_{t,loc}
W^{\frac12,4}_x $. Moreover, the flow map is locally Lipschitz continuous, the mass $M(u)$ and the energy $E_Z(u,V)$ are conserved, and we have the
global bounds
\begin{align*}
  \sup_{t \in \R}\|\nabla u(t)\|_{L^2}^2& \les
  \frac{2\|W^2\|^2_{L^2}}{ E_Z(W,-W^2)-E_Z(u(0),V(0))}E_Z(u(0),V(0)),\\ \sup_{t\in
  \R}\|V(t)\|_{L^2}^2 &\les 4E_Z(u(0),V(0)).
\end{align*}
\end{theorem}

Although no blow-up is known for the Zakharov system in $\R^4$, in view of the existence of blow-up for NLS near the threshold, one may naturally expect that the above condition $\|V(0)\|_{L^2(\R^4)} \les \|W^2\|_{L^2(\R^4)}$ is optimal (see also \cite{Candy2023,Guo2018} for further discussion).

Theorem \ref{thm:gwp below W} follows from the variational properties
of the ground state $(W,-W^2)$, see \cite{Guo2018}, and the following
local well-posedness result with a lower bound on the existence time:
\begin{theorem} \label{thm:lwp below W}
Let $d=4$ , $0<B<\|W^2\|_{L^2(\R^4)}$, $0<M<\I$, and
$\frac{1}{2}<s<2$. Then there exists $T=T(B,M,s)>0$, such that
for any  initial data $(u(0),V(0))\in H^s(\R^4)\times L^2(\R^4)$ satisfying
\EQ{
 \|u(0)\|_{H^{s}}\le M, \pq \|V(0)\|_{L^2}\le B,}
the Zakharov system \eqref{eqn:Zakharov first order} has a unique solution $(u,V)\in C([0,T];H^s(\R^4)\times L^2(\R^4))$
 with the poperty that $ u \in L^2_{t}([0,T];
W^{\frac12,4}_x(\R^4)) $ and the flow map is locally Lipschitz continuous. \end{theorem}

Although the endpoint $s=\frac{1}{2}$ is excluded from Theorem \ref{thm:lwp below W}, at a cost of complicating the statement, an endpoint version is possible. In particular, the condition on the Schr\"odinger data can be sharpened to instead requiring the weaker condition $\| e^{it\Delta} u(0) \|_{L^2_t W^{\frac{1}{2}, 4}_x([0, T]\times \RR^4)} \ll 1$, see Theorem \ref{thm:lwp sharp}.

Theorem \ref{thm:lwp below W} implies that provided the wave $L^2_x$ norm is below the ground state, the time of existence in fact only depends on the size of the norm $\|V\|_{L^\infty_t L^2_x}$.
This is in a sharp constrast to NLS, which is scaling-critical in the energy space.

The main difficulty in the proof of Theorem \ref{thm:lwp below W} comes from the interaction of the Schr\"odinger component with the large {\it free} wave component, which is the same as in the radial case \cite{Guo2018}, and also the major difference from the NLS case \cite{Kenig2006}.
In particular, the key step is to prove a uniform Strichartz estimate for the model problem
\[
 (i\p_t + \Delta - \Re(V)) u = F, \]
where the potential $V=e^{it|\nabla|} g \in L^\infty_t L^2_x$ is a free wave, see Theorem \ref{thm:unif Stz}.
It is a refined extension from the radial version in \cite[Theorem 1.5]{Guo2018}.
By considering the supersonic limit $\alpha\to 0$ in $V=e^{i\alpha t|\nabla|}g$, it is also linked to the static potential case. See \cite{Guo2018} and the references therein for more background concerning the Strichartz estimate with potentials.

It is crucial for our purpose that the Strichartz estimate is {\it uniform}: the best constant is bounded in terms of $\|g\|_{L^2}$.
Otherwise, the estimate depending on the profile $g$ follows from a perturbative argument and without any size constraint, using the dispersive property of the free wave $V$.
To prove the uniformity, we follow the same strategy as in \cite{Guo2018}: the concentration compactness argument, or the profile decomposition for the free wave potential $V$.
The main difficulty comes from the concentration of $V$, which is described by scaling limit in the profile decomposition, and
corresponding to the supersonic limit in the Schr\"odinger (parabolic) rescaling. The crucial observation is that the interaction with the Schr\"odinger component is essentially localized to the same frequency as the free wave (which is going to $\infty$ by concentration).
In the radial case \cite{Guo2018}, the frequency localization is achieved by using the normal form transformation for lower frequencies of the Schr\"odinger component, and the Strichartz estimate for radial solutions with derivative gain for higher frequencies.
The improved Strichartz is exploited also to treat remainder terms in the profile decomposition and in the supersonic approximation, which are decaying only in negative Sobolev (or Besov) norms.
The main idea of the present paper is to replace the radial Strichartz estimate with bilinear Strichartz (or Fourier restriction) estimates, exploiting the transversal interaction property of free Schr\"odinger solutions in high frequency. Thus, our analysis relies on the non-resonance property of the Schr\"odinger-wave interactions. To be more precise, we prove new bilinear Fourier restriction estimates for solutions of inhomogeneous wave and Schr\"odinger equations with forcing terms in dual endpoint Strichartz norms. This is one of the main novelties of this paper, see Section \ref{sec:fre} for more details.

Extending the remaining part of \cite{Guo2018}, namely the scattering and the weak blow-up, to the non-radial setting will require to solve another difficulty: a priori and global control of propagation.
Similarly to NLS, translation invariance of the Hamiltonian implies conservation of the momentum, and the scaling covariance implies the virial identity, but the Schr\"odinger-wave interaction destroys the Galilei invariance (or the Lorentz invariance for the wave equation), which prevents us from extracting a traveling effect from radiation.

\subsection{Organisation of the paper}\label{subsec:orga}
In Section \ref{sec:prel} we introduce notation, in particular Fourier multipliers and function spaces, and provide some preliminary estimates. In Section \ref{sec:bilest} we prove bilinear estimates for the right hand side of the Schr\"odinger equation. In Section \ref{sec:fre} we discuss bilinear Fourier restriction estimates for products of solutions to inhomogeneous wave and Schr\"odinger equations. In Section \ref{sec:proof of res gain} we apply this to prove refined bilinear estimates which involve a weaker Besov norm. Section \ref{sec:uniform-str} contains one of the main contributions of this paper, namely a uniform Strichartz estimate for the Schr\"odinger equation with a solution to the wave equation as a potential term satisfying the ground state constraint. Finally, in Section \ref{sec:gwp-proof} we first prove (a refined version of) Theorem  \ref{thm:lwp below W} and then derive Theorem \ref{thm:gwp below W}.

\section{Preliminaries}\label{sec:prel}

\subsection{Fourier multipliers}\label{subsec:fm}
In this subsection we fix notation, where we follow \cite[Section 2]{Candy2023} as far as possible.
We require both homogeneous and inhomogeneous Littlewood-Paley decompositions.
Let $\psi(r)\in C^\I_0(\RR)$ such that $r\psi'(r)\le 0\le\psi(r)\le
1$, $\psi(r)=1$ on $|r|\le 1$ and $\psi(r)=0$ on $|r|\ge 2$. We set
$\fy(r):=\psi(r)-\psi(r/2)$ and $\fy_\la(r):=\fy(r/\lambda)$ for $\lambda>0$. Then, we have
\[
 1= \psi(r) + \sum_{1< \la \in 2^\N} \fy_\la(r)
 \text{ for all }r\in\R, \text{ and } 1 = \sum_{ \lambda \in 2^\ZZ}
 \fy_\la \text{ for all } r\not=0.\]
For $\lambda\in 2^\NN$ with $\lambda>1$, define the Fourier multipliers
$$ P_\lambda = \varphi_\la(|\nabla|), \qquad P^{(t)}_\lambda = \varphi_\la( |\p_t|), \qquad C_\lambda = \varphi_\la(i\p_t + \Delta)$$
and for $\lambda =1$ we take
    $$ P_1 = \psi ( |\nabla|), \qquad P^{(t)}_1 = \psi ( |\p_t|), \qquad C_1 = \psi(i\p_t + \Delta)$$
Thus $P_\lambda$ is a (inhomogeneous) Fourier multiplier localising  the spatial Fourier support to the set $\{ \frac{\lambda}{2} < |\xi| < 2 \lambda \}$, $P^{(t)}_\lambda$ localises the temporal Fourier support to the set $\{ \frac{\lambda}{2} \les |\tau| \les 2 \lambda\}$, and $C_\lambda$ localises the space-time Fourier support to distances $\approx \lambda$ from the paraboloid (for $\lambda>1$). We also require homogeneous Fourier decomposition. To this end, we define for $\lambda \in 2^\ZZ$ the (homogeneous) multipliers
    $$ \dot{P}_\lambda = \varphi_\la(|\nabla|).$$
To restrict the Fourier support to larger sets, we use the notation
	\[P_{\les \lambda}=\psi\Big( \frac{|\nabla|}{\la}\Big), \pq P^{(t)}_{\les \lambda}=\psi\Big( \frac{|\p_t|}{\la}\Big), \pq C_{\les \lambda}=\psi\Big( \frac{i\p_t + \Delta}{\la}\Big),\]
and define $P^{(t)}_{>\la}=I-P^{(t)}_{\les \la}$ and $C_{>\la} = I - C_{\les \la}$. For ease of notation, for $\lambda \in 2^\NN\ge 1$ we often use the shorthand $P_\lambda f = f_\lambda$. In particular, note that $u_1 = P_1 u$ has Fourier support in $\{|\xi|\les 2\}$, and we have the identities
        $$ f = \sum_{\lambda \in 2^\NN} f_\lambda, \quad  f =
        \sum_{\lambda \in 2^\ZZ} \dot{P}_\lambda f, \quad \text{ for
          any }f \in L^2(\RR^d).$$
We reserve the notation $f_\lambda$ to denote the inhomogeneous Littlewood-Paley decomposition.

For brevity, let us denote the frequently used decompositions in modulation by
        $$ \PN{\lambda} u := C_{\les (\frac{\lambda}{2^8})^2} P_\lambda u, \qquad  \PF{\lambda} u := C_{> (\frac{\lambda}{2^8})^2} P_\lambda u,$$
and take
    \EQ{
 \PN{} u := \sum_{\la\in 2^\N} \PN{\lambda} u , \pq \PF{} u:= \sum_{\la\in 2^\N} \PF{\lambda} u.}
Thus $\PN{}$ localises to frequencies near to the paraboloid, $\PF{}$ localises to frequencies far from the paraboloid, and we have the identities  $u_\lambda =\PN{\lambda}u+\PF{\lambda}u$ and $u=\PN{}u+\PF{}u$. Note that the multipliers $\PN{}$ and $\PF{}$ have the parabolic scaling
\EQ{
 (\PN{\la} u)(t/\la^2,x/\la)=\PN{2}(u(4t/\la^2,2x/\la)),}
where $\PN{2}$ is a space-time convolution with a Schwartz function, so that we can easily deduce that $\PN{\la}$ and $\PF{\la}$ are bounded on any $L^p_tL^q_x$ uniformly in $\la\in 2^\N$, and that $\PN{}$ and $\PF{}$ are bounded on any $L^2_tB^s_{q,2}$.

\subsection{Function spaces}
In this subsection we introduce the function spaces which are used throughout this paper. Some are (special cases of) \cite[Section 2]{Candy2023}, although in addition we require certain refinements in view of the bilinear restriction estimates in Section \ref{sec:fre}. We define the homogeneous Besov spaces $\dot{B}^s_{q, r}$ using the norm
        $$ \| f \|_{\dot{B}^s_{q,r}} = \Big( \sum_{\lambda \in 2^\ZZ} \lambda^{sr} \| \dot{P}_\lambda f \|_{L^q}^r \Big)^\frac{1}{r}.$$
We use the notation $2^*= \frac{2d}{d-2}$ and $2_* = (2^*)' = \frac{2d}{d+2}$ to denote the endpoint Strichartz exponents for the Schr\"odinger equation. Thus, for $d \g 3$ we have
        \begin{equation}\label{eqn:stand stri} \| e^{it\Delta} f \|_{L^\infty_t L^2_x \cap L^2_t L^{2^*}_x} + \Big\| \int_0^t e^{i(t-s)\Delta} F(s) ds \Big\|_{L^\infty_t L^2_x \cap L^2_t L^{2^*}_x} \lesa \| f\|_{L^2_x} + \| F \|_{L^2_t L^{2_*}_x},\end{equation}
        see \cite{Keel1998}.
To control frequency localised nonlinear terms on the righthand side of the Schr\"odinger equation, for $\lambda \in 2^\NN$ we define
        $$ \| F \|_{N^{s}_\lambda} = \lambda^s  \| C_{\les (\frac{\lambda}{2^8})^2} F \|_{L^2_t L^{2_*}_x} + \lambda^{s-1} \|  F \|_{L^2_{t,x}}.$$
For later use, we observe that an application of Bernstein's inequality gives
    \begin{equation}\label{eqn:thm main schro est nonres:N embedding}
        \| F_{\lambda} \|_{N^s_{\lambda} } \lesa \lambda^s \| F_{\lambda} \|_{L^2_t L^{2_*}_x}.
    \end{equation}
To estimate the frequency localised solution, we define
        $$ \| u \|_{S^{s}_\lambda} = \lambda^s \| u \|_{L^\infty_t L^2_x} + \lambda^s \| u \|_{L^2_t L^{2^*}_x} +  \lambda^{s -1} \| (i \p_t + \Delta) u \|_{L^2_{t,x}}.$$
We require the stronger norm
        $$ \| u \|_{ \underline{S}^s_\lambda} = \lambda^s \| u
        \|_{L^\infty_t L^2_x} + \| (i\p_t + \Delta) u
        \|_{N^s_\lambda}$$
        in order to obtain the refined estimates involving Besov norms.
A computation shows that the norms $\| \cdot \|_{S^s_\lambda}$ and $\| \cdot \|_{\underline{S}^s_\lambda}$ only differ in the low modulation regime. Moreover,  a convexity argument gives control over the  Schr\"odinger admissible Strichartz spaces $L^q_t L^r_x$.

\begin{lemma}\label{lem:strichartz control}
Let $q,r \g 2$, $\frac{1}{q} + \frac{d}{2r} = \frac{d}{4}$, and $d\g 3$. Then for any $\lambda \in 2^\NN$ we have
        $$ \lambda^s \| u_\lambda \|_{L^q_t L^r_x} \lesa \| u_\lambda \|_{S^s_\lambda} \lesa \| u_\lambda \|_{\underline{S}^s_\lambda} $$
and the characterisation
        \begin{equation}\label{eqn:strong charac}
        \| u_\lambda \|_{\underline{S}^s_\lambda} \approx \| u_{\lambda} \|_{S^s_\lambda} + \| (i\p_t + \Delta) P^N_\lambda u \|_{L^2_t L^{2_*}_x}.
       \end{equation}
\end{lemma}
\begin{proof}
To prove the first inequality, we simply observe that by convexity we have $0\les \theta \les 1$ such that
        $$ \| u_\lambda \|_{L^q_t L^r_x} \lesa \| u_\lambda \|_{L^\infty_t L^2_x}^\theta \| u_\lambda \|_{L^2_t L^{2^*}_x}^{1-\theta}$$
and hence the bound follows from the definition of the norm $\| \cdot \|_{S^s_\lambda}$. Thus it remains to prove \eqref{eqn:strong charac}. But this follows by unpacking the definitions of the norms $\|\cdot \|_{S^s_\lambda}$ and $\| \cdot \|_{\underline{S}^s_\lambda}$, applying the Strichartz estimate \eqref{eqn:stand stri}, and observing that the Fourier localisation of $P^F_\lambda u$ gives
        \begin{align*}
            \| P^F_\lambda u \|_{L^2_t L^{2^*}_x} \lesa \lambda \| P^F_\lambda u \|_{L^2_{t,x}} = \lambda \| (i\p_t + \Delta)^{-1} (i\p_t + \Delta) P^F_\lambda u \|_{L^2_{t,x}}\lesa \lambda^{-1} \| (i\p_t + \Delta) u_\lambda \|_{L^2_{t,x}}.
        \end{align*}
\end{proof}

We sum the dyadic terms in $\ell^2$, and control the full solution using the norms
        $$ \| u \|_{S^{s}} = \Big( \sum_{\lambda \in 2^\NN} \| u_\lambda \|_{S^{s}_\lambda}^2 \Big)^\frac{1}{2},\qquad \| u \|_{\underline{S}^{s}} = \Big( \sum_{\lambda \in 2^\NN} \| u_\lambda \|_{\underline{S}^{s}_\lambda}^2 \Big)^\frac{1}{2} \qquad \| F \|_{N^{s}} = \Big( \sum_{\lambda \in 2^\NN} \| F_\lambda \|_{N^{s}_\lambda}^2 \Big)^\frac{1}{2}$$
where we recall that $u_\lambda = P_\lambda u $, $F_\lambda = P_\lambda F$ denotes the inhomogeneous Littlewood-Paley decomposition.

For technical reasons related to the duality and the energy inequality, we introduce the notation
        $$ \| u \|_Z = \| u \|_{L^\infty_t L^2_x} +  \| ( i\p_t + \Delta) u \|_{L^1_t L^2_x + L^2_t L^{2_*}_x}. $$

To localise any of the normed spaces $F=\underline{S}^s, S^s,\ldots$
defined above  to time intervals $I \subset \RR$, we define the usual temporal restriction norm as
$$ \| u \|_{F(I)} = \inf_{ u'\in F\text{ and } u'|_I = u} \| u'\|_{F},$$
provided that such an extension $u'\in F$ exists.

\subsection{The Duhamel formula}\label{subsec:duhamel}
The free Schr\"odinger group is denoted by $ \mc{U}_0(t)=e^{it\Delta}$.
Fix any $t_0\in\R$ and any interval $I\ni t_0$. Then, the Duhamel integral from $t_0$ is denoted by
\EQ{
 \mc{I}_0[F](t) = -i \int_{t_0}^t e^{i(t-s) \Delta} F(s) ds,}
which is the unique solution $u$ on $I$ of
\EQ{
 i\p_t u + \De u = F, \pq u(t_0)=0.}

Later we will use the more general notation with a time-dependent potential $V(t,x)$ on $I\times \RR^d$, where $\mc{I}_VF$ denotes the unique solution $u$ of
\EQ{
 i\p_t u + \De u - V u = F, \pq u(t_0)=0}
and $\mc{U}_V(t;s)$ to denote the homogeneous solution operator, thus for $t,s\in I$, $\mc{U}_V(t;s) f$ solves
    $$ i\p_t u + \Delta u - V u = 0, \qquad u(s) = f. $$
See Section \ref{ss:Duh exp} for the precise definition and its usage.

The energy inequality we will use for the Schr\"odinger equation is the following:

\begin{lemma}\label{lem:energy ineq}
Let $t_0 \in \RR$. For $u(t)=e^{i(t-t_0)\Delta}f+\mc{I}_0[F](t)$ and
all $\lambda \in 2^\NN$ we have
        \begin{equation}\label{eqn:lem energy ineq:energy}
            \| u_\lambda \|_{\underline{S}^{s}_\lambda} \lesa \lambda^s \| f_\lambda \|_{L^2_x} +  \| F_\lambda \|_{N^{s}_\lambda}
        \end{equation}
and
         \begin{equation}\label{eqn:lem energy ineq:sharp}
            \| u \|_{L^\infty_t L^2_x} + \| u \|_{L^2_t L^{2^*}_x} \lesa \| f \|_{L^2_x} + \sup_{ \| w \|_{Z} \les 1} \Big| \int_{\RR^{1+d}} \overline{w} F \,dt dx \Big|,
         \end{equation}
         and
        \begin{equation} \label{eqn:lem energy ineq:sharp-highmod}
        \sup_{ \| w \|_{Z} \les 1} \Big| \int_{\RR^{1+d}} \overline{w} C_{\gtrsim \lambda^2 } F_\lambda \,dt dx \Big|\lesa \lambda^{-1} \| F_\lambda\|_{L^2}.
	        \end{equation}
\end{lemma}
\begin{proof}
To prove the energy inequality \eqref{eqn:lem energy ineq:energy}, by the definition of the norm $\| \cdot \|_{\underline{S}^s_\lambda}$ and the Strichartz estimate \eqref{eqn:stand stri}, it suffices to show that
        $$ \| \mc{I}_0[P^F_\lambda F] \|_{L^\infty_t L^2_x} \lesa \lambda^{-1} \| F_\lambda \|_{L^2_{t,x}}. $$
Define $G(t) = e^{-it\Delta} F$ and write $P^F_\lambda F (t) = e^{it\Delta} P^{(t)}_{ > (\frac{\lambda}{2^8})^2} G(t)$ where $P^{(t)}_{\mu}$ restricts the temporal Fourier support to the region $\mu/2 < \tau < 2 \mu$. Then an application of Berstein's inequality gives
       \begin{align*}
            \| \mc{I}_0[P^F_\lambda F] \|_{L^\infty_t L^2_x} &= \Big\| \int_{t_0}^t \p_s \big( P^{(t)}_{> (\frac{\lambda}{2^8})^2} \p_t^{-1} G\big)(s) ds \Big\|_{L^\infty_t L^2_x} \\
            &\lesa \| P^{(t)}_{> (\frac{\lambda}{2^8})^2} \p_t^{-1} G \|_{L^\infty_t L^2_x} \lesa \lambda^{-1} \| G_\lambda \|_{L^2_{t,x}} \approx \lambda^{-1} \| F_\lambda \|_{L^2_{t,x}}
       \end{align*}
as required.

 To prove \eqref{eqn:lem energy ineq:sharp}, we apply duality and obtain
\[  \| u \|_{L^\infty_t L^2_x} + \| u \|_{L^2_t L^{2^*}_x} \lesa  \| f \|_{L^2_x} + \Big| \int_{\RR^{1+d}} \overline{v_1}\mc{I}_0Fdtdx \Big| +\Big| \int_{\RR^{1+d}} \overline{v_2}\mc{I}_0F dtdx \Big|,
\]
where $ \|v_1\|_{L^1_t L^{2}_x}\lesa 1$, $ \|v_2\|_{L^2_t L^{2_*}_x}\lesa 1$. We compute
\[\int_{\RR^{1+d}} \overline{v_j}\mc{I}_0F dtdx= \int_{\RR^{1+d}} \overline{w}F dtdx,\]
for \[
w_j(s)=\begin{cases} i\int^{\infty}_s e^{i(s-t)\Delta}v_j(t)dt \quad (s\g 0),\\ -i\int_{-\infty}^s e^{i(s-t)\Delta}v_j(t)dt \quad (s< 0).\end{cases}
\]
The function $w=w_1+w_2$ satisfies $ \| w \|_{Z} \les 1$ and the proof of \eqref{eqn:lem energy ineq:sharp} is complete.

Finally, by Cauchy-Schwarz, it remains to prove that
\[\lambda   \| C_{\gtrsim \lambda^2 }P_{\lambda} w \|_{L^2} \lesa \| w \|_{Z}, \]
which follows from the following two inequalities: Firstly, using Bernstein's inequality,
\begin{align*}
  \| C_{\gtrsim \lambda^2 } P_{\lambda} w_1 \|_{L^2} \lesa{}& \sum_{\mu \gtrsim \lambda^2}   \| C_{\mu } w_1 \|_{L^2} \lesa \sum_{\mu \gtrsim \lambda^2} \mu^{\frac{1}{2}} \| P^{(t)}_{ \mu}(e^{-it\Delta} w_1) \|_{L^1_tL^2_x} \lesa \sum_{\mu \gtrsim \lambda^2} \mu^{-\frac{1}{2}} \| \partial_t P^{(t)}_{ \mu}(e^{-it\Delta} w_1) \|_{L^1_tL^2_x} \\\lesa{}& \lambda^{-1}\| (i\partial_t+\Delta) w_1 \|_{L^1_tL^2_x}.
\end{align*}
Secondly, by a similar argument,
\begin{align*}
  \| C_{\gtrsim \lambda^2 } P_{\lambda} w_2 \|_{L^2}  \lesa{}& \lambda^{-2}\| (i\partial_t+\Delta) P_{\lambda} w_2 \|_{L^2}\lesa \lambda^{-1}\| (i\partial_t+\Delta) w_2 \|_{L^2_tL^{2_*}_x},
\end{align*}
which completes the proof of \eqref{eqn:lem energy ineq:sharp-highmod}.
\end{proof}
Clearly, in view of Lemma \ref{lem:strichartz control}, the inequality \eqref{eqn:lem energy ineq:energy} can also be used to bound the weaker norm $S^{s}$ and in fact this suffices for the small data theory. However for the large data theory, we need a refinement to Besov norms in the resonant interactions (when $u$ has Fourier support close to the parabola) and this requires the sharper bound \eqref{eqn:lem energy ineq:sharp}.

\subsection{An elementary product estimate}

\begin{lemma}\label{lem:elem prod est}
There exists $C>0$ such that for any $s, \epsilon \g 0$ and any sequence of functions $(g^{(\mu)})_{\mu \in 2^\NN}$ with $\supp \widehat{g}^{(\mu)} \subset \{ |\xi| \approx \mu\}$ we have
        $$ \Big( \sum_{\lambda \in 2^\NN} \lambda^{2(s-1)} \Big\| \sum_{\mu \gg \lambda} P_\lambda( f g^{(\mu)}) \Big\|_{L^2_x}^2 \Big)^\frac{1}{2} \les C \| \lr{\nabla}^{-\epsilon} f \|_{L^\frac{d}{2}_x} \Big\| \Big(\sum_{\mu \in 2^\NN} \mu^{2(s+\epsilon)}  |g^{(\mu)}|^2 \Big)^\frac{1}{2} \Big\|_{L^{2^*}_x}. $$
\end{lemma}
\begin{proof}
An application of the continuous embedding $F^0_{2_*, \infty} \subset H^{-1}$ implies that
         $$ \Big( \sum_{\lambda \in 2^\NN} \lambda^{2(s-1)} \Big\| \sum_{\mu \gg \lambda} P_\lambda( f g^{(\mu)}) \Big\|_{L^2_x}^2 \Big)^\frac{1}{2} \lesa \Big\| \sup_{\lambda \in 2^\NN} \lambda^s \Big| \sum_{\mu \gg \lambda} P_\lambda( f g^{(\mu)})\Big| \Big\|_{L^{2_*}_x}. $$
Let $\rho(\lambda y) \lambda^d$ denote the kernel of $P_\lambda$. Then
for any $\lambda \in 2^\NN$ we have the pointwise a.e.\ estimate
        \begin{align*}
            \lambda^s \Big| \sum_{\mu \gg \lambda} P_\lambda( f g^{(\mu)})(x)\Big|
                        &\lesa \lambda^s \int_{\RR} |\rho(\lambda y)| \lambda^d \sum_{\mu \gg \lambda} |f_{\approx \mu}(x-y)| |g^{(\mu)}(x-y)| dy \\
                        &\lesa  \int_{\RR} |\rho(\lambda y)| \lambda^d \Big(\sum_{\mu} \mu^{ - 2 \epsilon} |f_{\mu}(x-y)|^2\Big)^\frac{1}{2} \Big( \sum_{\mu} \mu^{2(s+\epsilon)} |g^{(\mu)}(x-y)|^2\Big)^\frac{1}{2} dy \\
                        &\lesa \mc{M}\Big[ \Big(\sum_{\mu} \mu^{ - 2 \epsilon} |f_{\mu}|^2\Big)^\frac{1}{2} \Big( \sum_{\mu} \mu^{2(s-\epsilon)} |g^{(\mu)}|^2\Big)^\frac{1}{2} \Big](x)
        \end{align*}
where $\mc{M}$ denotes the Hardy-Littlewood maximal function. Therefore, by the $L^{2_*}_x$ boundedness of the maximal function, the standard square function estimate, and H\"older's inequality  we obtain
        \begin{align*}
         \Big( \sum_{\lambda \in 2^\NN} \lambda^{2(s-1)} \Big\| \sum_{\mu \gg \lambda} P_\lambda( f g^{(\mu)}) \Big\|_{L^2_x}^2 \Big)^\frac{1}{2}
                    &\lesa \Big\| \mc{M}\Big[ \Big(\sum_{\mu} \mu^{ - 2 \epsilon} |f_{\mu}|^2\Big)^\frac{1}{2} \Big( \sum_{\mu} \mu^{2(s+\epsilon)} |g^{(\mu)}|^2\Big)^\frac{1}{2} \Big] \Big\|_{L^{2_*}_x} \\
                    &\lesa \Big\| \Big(\sum_{\mu} \mu^{ - 2 \epsilon} |f_{\mu}|^2\Big)^\frac{1}{2} \Big\|_{L^{\frac{d}{2}}_x} \Big\|\Big( \sum_{\mu} \mu^{2(s+\epsilon)} |g^{(\mu)}|^2\Big)^\frac{1}{2} \Big\|_{L^{2^*}_x} \\&\approx \|\lr{\nabla}^{-\epsilon} f \|_{L^{\frac{d}{2}}_x} \Big\| \Big(\sum_{\mu \in 2^\NN} \mu^{2(s+\epsilon)}  |g^{(\mu)}|^2 \Big)^\frac{1}{2} \Big\|_{L^{2^*}_x}.
        \end{align*}
\end{proof}

\section{Bilinear estimates}\label{sec:bilest}
In this section we give the basic bilinear estimate that we require to control the Schr\"odinger nonlinearity. Although this bound is not strong enough to obtain the uniform Strichartz estimate that is needed in the proof of Theorem \ref{thm:lwp below W}, it has the advantage that it can be used, together with the energy inequality Lemma \ref{lem:energy ineq}, to easily upgrade a solution $u \in S^s$, to $u \in \underline{S}^s$. The stronger control provided by the space $\underline{S}^s$ is crucial in obtaining the Besov refinement that we require in later sections.

\begin{theorem}\label{thm:main schro est nonres}
Let $d\g 4$, $0\les s < 1$, and $0\les \epsilon < 1-s$. Then
    \EQ{ \label{V Sw2N}
  \|  v  u \|_{N^{s}} \lesa \Big(\| \lr{\nabla}^{-\epsilon} v \|_{L^\infty_t L^{\frac{d}{2}}_x} + \| (i\p_t \pm |\nabla|) v \|_{L^2_t H^{\frac{d}{2}-3-\epsilon}_x}\Big) \|u \|_{S^{s+\epsilon}}. }
\end{theorem}

Theorem \ref{thm:main schro est nonres} suffices to prove small data
global well-posedness in $H^s$ and (at least) the case $\epsilon=0$ can
be seen as a special case of \cite[Theorem 3.1]{Candy2023}. On the other hand, to deal with general data below the ground state, we need a version of Theorem \ref{thm:main schro est nonres} with a refinement to Besov norms. This is a much more delicate problem, which requires the use of bilinear restriction estimates to exploit the transversality occurring in the (time) resonant interactions.

\begin{proof}[Proof of Theorem \ref{thm:main schro est nonres}] Let $d\g 4$, $0\les s < 1$, and $0\les \epsilon < 1-s$. We start by proving that
        \begin{equation}\label{eqn:thm main schro est nonres:v large temp freq}
            \| v u \|_{N^s} = \Big( \sum_{\lambda_0 \in 2^\NN} \lambda_0^{2s} \| P_{\lambda_0}(vu) \|_{N^s_{\lambda_0}}^2 \Big)^\frac{1}{2} \lesa \| v \|_{L^2_t H^{\frac{d}{2}-1-\epsilon}_x} \| u \|_{S^{s+\epsilon}}.
        \end{equation}
Decompose the product into the interactions
        \begin{equation}\label{eqn:thm main schro est nonres:high-low decomp}
            P_{\lambda_0}(vu) = P_{\lambda_0}(v u_{\ll \lambda_0})+P_{\lambda_0}(v u_{\gg \lambda_0}) + P_{\lambda_0}(v u_{\approx \lambda_0}).
        \end{equation}

For the first term in \eqref{eqn:thm main schro est nonres:high-low decomp}, we have for any $\lambda_1 \ll \lambda_0$
            \begin{align*}
              \| P_{\lambda_0}( v u_{\lambda_1} ) \|_{N^s_{\lambda_0}} \lesa \lambda_0^s \| v_{\approx \lambda_0} u_{\lambda_1} \|_{L^2_t L^{2_*}_x} \lesa \Big( \frac{\lambda_1}{\lambda_0}\Big)^{ \frac{d}{2} -1 - s - \epsilon} \| v \|_{L^2_t H^{\frac{d}{2}-1 - \epsilon}_x} \lambda_1^{s+\epsilon} \| u_{\lambda_1} \|_{L^\infty_t L^2_x}
            \end{align*}
        and so summing up over $\lambda_0 \gg \lambda_1$ gives \eqref{eqn:thm main schro est nonres:v large temp freq}.
For the second term, we observe that an application of Bernstein's inequality and \eqref{eqn:thm main schro est nonres:N embedding} gives for any $\lambda_1 \gg \lambda_0$
    \begin{align*}
       \| P_{\lambda_0}(vu_{\lambda_1}) \|_{N^s_{\lambda_0}} \lesa \lambda_0^s  \| P_{\lambda_0}( v_{\approx \lambda_1} u_{\lambda_1} ) \|_{L^2_t L^{2_*}_x} &\lesa \lambda_0^{s+\frac{d}{2}-1} \| v_{\approx \lambda_1} \|_{L^2_{t,x}} \| u_{\lambda_1} \|_{L^\infty_t L^2_x} \\
       &\lesa \Big( \frac{\lambda_0}{\lambda_1} \Big)^{s+\frac{d}{2} - 1} \| v \|_{L^2_t H^{\frac{d}{2}-1-\epsilon}_x} \| u_{\lambda_1} \|_{S^{s+\epsilon}_{\lambda_1}}.
    \end{align*}
Clearly, provided that $s + \frac{d}{2}-1 > 0$, this can be summed up over frequencies to give \eqref{eqn:thm main schro est nonres:v large temp freq} for the second term in \eqref{eqn:thm main schro est nonres:high-low decomp}.
To bound the third term in \eqref{eqn:thm main schro est nonres:high-low decomp}, we note that
    \begin{align*}
      \| P_{\lambda_0}(v u_{\approx \lambda_0} ) \|_{N^s_{\lambda_0}} \lesa \lambda_0^s \| v_{\lesa \lambda_0}  u_{\approx \lambda_0} \|_{L^2_t L^{2_*}_x} \lesa \| v_{\lesa \lambda_0} \|_{L^2_t L^d_x} \lambda_0^s \| u_{\approx \lambda_0} \|_{L^\infty_t L^2_x} \lesa \| v \|_{L^2_t H^{\frac{d}{2}-1-\epsilon}_x} \lambda_0^{s+\epsilon} \| u_{\approx \lambda_0} \|_{L^\infty_t L^2_x}.
    \end{align*}
Again, this can be summed up to give \eqref{eqn:thm main schro est nonres:v large temp freq}.  Therefore, for any fixed $0<\de\ll 1$, if $\supp \widetilde{v} \subset \{ |\tau| \ge \de|\xi|^2+4\}$, the required bound \eqref{V Sw2N} follows from the inequality
        $$ \| P^{(t)}_{\ge \de\lambda_0^2/2+2} v_{\lambda_0} \|_{L^2_t H^{\frac{d}{2}-1-\epsilon}_x} \lesa \| (i\p_t \pm |\nabla|) v_{\lambda_0} \|_{L^2_t H^{\frac{d}{2}-3-\epsilon}_x}. $$
It only remains to prove that, provided $v \in L^\infty_t L^{\frac{d}{2}}_x$ with $\supp \widetilde{v} \subset \{ |\tau| \le 2\de|\xi|^2+8\}$, we have
    \begin{equation}\label{eqn:thm main schro est nonres:goal}
          \Big( \sum_{\lambda_0 \in 2^\NN}  \big\| P_{\lambda_0}\big(v u\big) \big\|_{N^s_{\lambda_0}}^2 \Big)^\frac{1}{2} \lesa \| \lr{\nabla}^{-\epsilon} v \|_{L^\infty_t L^\frac{d}{2}_x} \| u \|_{S^{s+\epsilon}}.
         \end{equation}
Again we decompose the product
     $$ P_{\lambda_0}(uv) = P_{\lambda_0}( v u_{\ll \lambda_0}) + P_{\lambda_0} ( v u_{\gg \lambda_0}) + P_{\lambda_0}( v u_{\approx \lambda_0})=:I_1+I_2+I_3$$
and consider each interaction separately.

\textbf{Contribution of $I_1$.} We in fact show that for any $ 2\les r \les d $ and $\lambda_0 \gg \lambda_1$ we have the stronger estimate
    \begin{equation}\label{eqn:thm multilinear est gen:main high-low}
        \big\| P_{\lambda_0} \big( v u_{\lambda_1} \big) \big\|_{N^{s}_{\lambda_0}} \lesa
            \Big( \frac{\lambda_1}{\lambda_0}\Big)^{\frac{d}{r}-s-1-\epsilon} \| v \|_{L^\infty_t B^{\frac{d}{r}-2-\epsilon}_{r, \infty}} \| u_{\lambda_1} \|_{S^{s+\epsilon}_{\lambda_1}}.
    \end{equation}
This implies \eqref{eqn:thm main schro est nonres:goal} in the case $\lambda_0 \gg \lambda_1$ in view of our assumptions on $s$ and $\epsilon$. We start by dealing with the low modulation output. The Fourier support assumption on $v$ implies that
        $$  \PN{\lambda_0}( v u_{\lambda_1}) = \PN{\lambda_0}( v_{\approx \lambda_0} P^{(t)}_{\approx \lambda_0^2} \PF{\la_1} u).$$
Hence for any $2 \les r \les d$, the disposability of the multiplier $\PN{\lambda_0}$, and Bernstein's inequality, gives
        \begin{align*}
           \lambda_0^s \big\| \PN{\lambda_0}( v u_{\lambda_1}) \big\|_{L^2_t L^{2_*}_x}
                        &\lesa \lambda_0^s \lambda_1^{\frac{d}{r}-1}  \| v_{\approx \lambda_0} \|_{L^\infty_t L^r_x} \| P^{(t)}_{\approx \lambda_0^2} \PF{\la_1} u \|_{L^2_{t,x}} \\
                        &\lesa  \lambda_0^{s-2} \lambda_1^{\frac{d}{r}-1}  \| v_{\approx \lambda_0} \|_{L^\infty_t L^r_x} \| (i\p_t + \Delta) \PF{\la_1} u \|_{L^2_{t,x}},\\
                        &\lesa \Big( \frac{\lambda_1}{\lambda_0}\Big)^{\frac{d}{r}-1-s-\epsilon} \|  v \|_{L^\infty_t B^{\frac{d}{r}-2-\epsilon}_{r,\infty}}  \| u_{\lambda_1} \|_{S^{s+\epsilon}_{\lambda_1}}
        \end{align*}
thus \eqref{eqn:thm multilinear est gen:main high-low} holds. To bound the high modulation output, we observe that
    \begin{align*}
      \lambda_0^{s-1}  \| \PF{\lambda_0}( v u_{\lambda_1} ) \|_{L^2_{t,x}} &\lesa \lambda_0^{s-1} \| v_{\approx \lambda_0} \|_{L^\infty L^r_x}  \lambda_1^{\frac{d}{r} -1} \| u_{\lambda_1} \|_{L^2_t L^{2^*}_x}
      \lesa \Big( \frac{\lambda_1}{\lambda_0} \Big)^{\frac{d}{r} - 1 - s-  \epsilon} \| v \|_{L^\infty_t B^{\frac{d}{r}-2-\epsilon}_{r,\infty}} \| u_{\lambda_1} \|_{S^{s+\epsilon}_{\lambda_1}}
    \end{align*}
and so  \eqref{eqn:thm multilinear est gen:main high-low} follows.

\textbf{Contribution of $I_2$.} We start by observing that for the part of $u$ which is close to the paraboloid, the non-resonant identity
    $$P_{\lambda_0}( v P^N_{\gg \lambda_0} u ) = P^F_{\lambda_0}( v P^N_{\gg \lambda_0} u)$$
together with Lemma \ref{lem:elem prod est} implies that for any $s, \epsilon \g 0$ we have
    \begin{align*}
        \Big( \sum_{\lambda_0 \in 2^\NN} \lambda_0^{2s} \| P_{\lambda_0}( v P^N_{\gg \lambda_0} u ) \|_{N^s_{\lambda_0}}^2 \Big)^\frac{1}{2}
                    &\lesa \Big( \sum_{\lambda_0 \in 2^\NN} \lambda_0^{2(s-1)} \| P_{\lambda_0}( v P^N_{\gg \lambda_0} u ) \|_{L^2_{t,x}}^2 \Big)^\frac{1}{2}  \\
                    &\lesa \| \lr{\nabla}^{-\epsilon} v \|_{L^\infty_t L^{\frac{d}{2}}_x} \Big( \sum_{\lambda_1 \in 2^\NN} \lambda_1^{2(s+\epsilon)} \| P^N_{\lambda_1} u \|_{L^2_t L^{2^*}_x}^2 \Big)^{\frac{1}{2}}\\
                    &\lesa \| \lr{\nabla}^{-\epsilon} v \|_{L^\infty_t L^{\frac{d}{2}}_x}  \| u \|_{S^{s+\epsilon}}.
    \end{align*}
    On the other hand, if $u_{\lambda_1}$ is supported away from the paraboloid, then applying \eqref{eqn:thm main schro est nonres:N embedding} we see that for any $2\les r \les d$
	\begin{align}
	\| P_{\lambda_0}( v \PF{\lambda_1}u) \|_{N^s_{\lambda_0}}&\lesa	\lambda^s_0\big\| P_{\lambda_0} (v_{\approx \lambda_1} \PF{\lambda_1} u)\big\|_{L^2_t L^{2_*}_x} \notag \\
			&\lesa \lambda_0^{s+\frac{d}{r}-1} \| v_{\approx \lambda_1}\|_{L^\infty_t L^r_x}  \| \PF{\la_1} u \|_{L^2_{t,x}}
			\lesa \Big( \frac{\lambda_0}{\lambda_1} \Big)^{ s + \frac{d}{r} -1} \| v \|_{L^\infty_t B^{\frac{d}{r}-2-\epsilon}_{r, \infty}} \| u_{\lambda_1} \|_{S^{s+\epsilon}_{\lambda_1}}.
            \label{eqn:thm multilinear est gen:main low-high with gain}
	\end{align}
Since there is a low-high frequency gain, summing up over frequencies gives \eqref{eqn:thm main schro est nonres:goal}.

\textbf{Contribution of $I_3$.} We now prove, without the Fourier support assumption on $v$, that for any $\lambda_0 \approx \lambda_1$ we have the resonant bound
        \begin{equation}\label{eqn:thm multilinear est gen:main high-high res}
            \lambda_0^s \| \PN{\lambda_0}( v \PN{\la_1} u ) \|_{L^2_t L^{2_*}_x} \lesa \| v_{\lesa \lambda_1} \|_{L^\infty_t L^{d/2}_x}  \| u_{\lambda_1} \|_{S^{s}_{\lambda_1}}
        \end{equation}
and the non-resonant estimates with a weaker norm of $v$
        \begin{equation}\label{eqn:thm multilinear est gen:main high-high nonres}
            \begin{split}
            \lambda_0^s \big\| \PN{\lambda_0}\big( v \PF{\la_1} u \big) \big\|_{L^2_t L^{2_*}_x} + \lambda_0^{s-1} \| \PF{\lambda_0}( v u_{\lambda_1} ) \|_{L^2_{t,x}} \lesa \lambda_1^{-1} \| v_{\lesa \lambda_1} \|_{L^\infty_t L^d_x}  \| u_{\lambda_1} \|_{S^{s}_{\lambda_1}}.
            \end{split}
        \end{equation}
Applying the definition of the norm $N^{s}_{\lambda_0}$ together with Bernstein's inequality, we see that \eqref{eqn:thm multilinear est gen:main high-high res} and \eqref{eqn:thm multilinear est gen:main high-high nonres} implies \eqref{eqn:thm main schro est nonres:goal} when $\lambda_0 \approx \lambda_1$. To prove the resonant bound \eqref{eqn:thm multilinear est gen:main high-high res}, we simply apply H\"older's inequality
        \begin{align*}
            \lambda_0^s \| \PN{\lambda_0} ( v \PN{\la_1} u ) \|_{L^2_t L^{2_*}_x}
                    &\lesa \lambda_0^s \| v \|_{L^\infty_t L^{d/2}_x}  \| \PN{\la_1} u \|_{L^2_t L^{2^*}_x} \lesa \| v \|_{L^\infty_t L^{d/2}_x} \| u_{\lambda_1} \|_{S^{s}_{\lambda_1}}.
        \end{align*}
On the other hand, to prove \eqref{eqn:thm multilinear est gen:main high-high nonres}, if the output has small modulation, then we again apply H\"older's inequality and observe that
        \begin{align*}
          \lambda_0^s \big\| \PN{\lambda_0}\big( v \PF{\la_1} u \big) \big\|_{L^2_t L^{2_*}_x}
                            &\lesa \lambda_0^s \| v \|_{L^\infty_t L^d_x} \| \PF{\la_1} u \|_{L^2_{t,x}}
                            \lesa \lambda_0^{-1} \|v \|_{L^\infty_t L^d_x} \| u_{\lambda_1} \|_{S^s_{\lambda_1}}.
        \end{align*}
If the output has large modulation, then we instead use
        \begin{align*}
          \lambda_0^{s-1} \big\| \PF{\lambda_0}\big( v u_{\lambda_1} \big) \big\|_{L^2_{t,x}}
                            &\lesa \lambda_0^{s-1} \| v \|_{L^\infty_t L^d_x} \| u_{\lambda_1} \|_{L^2_t L^{2^*}_x}
                            \lesa \lambda_0^{-1} \|v \|_{L^\infty_t L^d_x} \| u_{\lambda_1} \|_{S^s_{\lambda_1}}.
        \end{align*}
Therefore \eqref{eqn:thm multilinear est gen:main high-high res} and \eqref{eqn:thm multilinear est gen:main high-high nonres} follow.
\end{proof}

\section{Fourier restriction estimates}\label{sec:fre}

 Our eventual goal is to improve the result of Theorem \ref{thm:main
   schro est nonres} and include a weaker Besov norm on the right hand
 side. A close inspection of the proof of Theorem \ref{thm:main schro
   est nonres}, shows that we already have a Besov gain for all
 interactions except the resonant case of $I_2$. To obtain a
 suitable gain in this region, we require stronger estimates which
 exploit the fact that resonant interactions can only occur for
 transverse interactions. In particular, the new input is the use of bilinear restriction estimates to exploit the transversality between free waves and free Schr\"odinger solutions.

There are two bilinear estimates that we require. The first is an inhomogeneous version of the bilinear restriction estimate for wave-Schr\"odinger interactions. The key point is that we prove that the bilinear restriction type estimate holds not just for free solutions to the Schr\"odinger equation, but also inhomogeneous solutions satisfying $(i\p_t + \Delta) u  \in L^2_t L^{\frac{4}{3}}_x$. This bilinear restriction estimate is used to obtain a high-low frequency gain which is required to place $v(t) \in \dot{B}^{\frac{4}{r}-2}_{r, \infty}$. Without a high-low gain we can only place $v(t) \in \dot{B}^{\frac{4}{r}-2}_{r, 1}$, i.e. we would require a much stronger $\ell^1$ summation over the dyadic frequencies of $v$. Recall that we have defined
        $$ \| u \|_Z = \|u \|_{L^\infty_t L^2_x} + \| ( i\p_t + \Delta) u \|_{L^1_t L^2_x + L^2_t L^{2_*}_x}.$$
Clearly we have the inequalities
        $$ \|u \|_Z \les  \|u \|_{L^\infty_t L^2_x} + \| ( i\p_t + \Delta) u \|_{L^2_t L^{2_*}_x}$$
and
    $$\| P^N_\la u \|_{S^s_\lambda} \lesa \lambda^s \| P^N_\la u \|_Z \lesa \| u_\lambda \|_{\underline{S}^s_\lambda}.$$

\begin{theorem}[Bilinear $L^2_{t,x}$ for inhomogeneous wave-Schr\"odinger interactions]\label{thm:bilinear wave-schrodinger}
Let $0\les \alpha < \frac{1}{2}$. For all $\lambda_0, \lambda_1 \in 2^\NN$ and $ \mu \in 2^\ZZ$ with $\min\{\mu, \lambda_0\}\lesa \lambda_1$, and any free wave $v = e^{\pm i t |\nabla|} g $ we have
    \begin{equation}\label{eqn:thm bi wave-schro:low-high}
        \| P_{\les \lambda_0}( \dot{P}_{\mu} v u_{\lambda_1}) \|_{L^2_{t,x}(\RR^{1+4})} \lesa   \Big( \frac{\min\{\mu, \lambda_0\}}{\lambda_1}\Big)^{\alpha} \min\{\mu, \lambda_0\} \| \dot{P}_\mu v \|_{L^\infty_t L^2_x} \| u_{\lambda_1} \|_{Z}.
    \end{equation}
\end{theorem}

The second bilinear estimate we exploit is an inhomogeneous version of
a bilinear restriction estimate for the paraboloid. The key point is
that bilinear restriction estimates give additional spatial
integrability, which eventually means that we can replace the
$L^\infty_t L^2_x$ norm of the free wave $v$, with the weaker Besov
norm $L^\infty_t B^{-1}_{4, \infty}$. This is a crucial ingredient for
bounding error terms which arise in the profile decomposition
arguments in Subsection \ref{ss:prof decomp}.

\begin{theorem}[Bilinear restriction for inhomogeneous Schr\"odinger]\label{thm:bilinear schrodinger}
Let $r>\frac{5}{3}$. For any $\mu \in 2^\ZZ$ we have
    \begin{equation}\label{eq:thm bi-schro:high-high}
    		\| \dot{P}_\mu (\overline{w} u) \|_{L^1_t L^r_x(\RR^{1+4})} \lesa \mu^{2-\frac{4}{r}}  \| w \|_{Z} \|u \|_{Z}.
    	\end{equation}
\end{theorem}

The range $r \g \frac{5}{3}$ is sharp, in the sense that \eqref{eq:thm bi-schro:high-high} fails for $r<\frac{5}{3}$. Note that by taking $w$ and $u$ to be free solutions to the Schr\"odinger equation, \eqref{eq:thm bi-schro:high-high} recovers the bilinear restriction estimates for the paraboloid in $L^1_t L^r_x$. Bilinear restriction estimates for the paraboloid were first obtained by Tao \cite{Tao2003a}, this was then extended to the mixed norm case $L^q_t L^r_x$ with $q>1$ by Lee-Vargas \cite{Lee2008}. The case $L^1_t L^r_x$, which corresponds to the homogeneous version of \eqref{eq:thm bi-schro:high-high}, can be found in \cite{Candy2017}. The key importance of Theorems \ref{thm:bilinear wave-schrodinger} and \ref{thm:bilinear schrodinger} is that they hold for \emph{inhomogeneous} solutions to the Schr\"odinger equation and wave equations, and thus are particularly well suited to iterative arguments arising in the study of nonlinear PDE.

In order to simplify the presentation, we do not attempt to state the
bilinear estimates in Theorem \ref{thm:bilinear wave-schrodinger} and
Theorem \ref{thm:bilinear schrodinger} in the greatest possible
generality. However, it is clear from the proof given below, that
similar statements hold in general dimensions and for general
frequency interactions (provided only that the corresponding estimate for free solutions holds).

The strategy to prove both Theorem \ref{thm:bilinear schrodinger} and Theorem \ref{thm:bilinear wave-schrodinger} is similar. For instance, to prove Theorem \ref{thm:bilinear wave-schrodinger}, we start by observing that the estimate is true for free solutions. More precisely, we claim that if $\lambda_0, \lambda_1 \in 2^\NN$ and $ \mu \in 2^\ZZ$ with $\min\{\mu, \lambda_0\}\lesa \lambda_1$ then
        \begin{equation}\label{eqn:bil wave-schro I free case}
            \|  P_{\les \lambda_0}( \dot{P}_\mu v e^{it\Delta} f_{\lambda_1} ) \|_{L^2_{t,x}(\RR^{1+4})} \lesa   \Big( \frac{\min\{\mu, \lambda_0\}}{\lambda_1}\Big)^{\frac{1}{2}} \min\{\mu, \lambda_0\} \| \dot{P}_{\mu} v \|_{L^\infty_t L^2_x} \| f_{\lambda_1} \|_{L^2_{x}},
        \end{equation}
where we recall that $v = e^{it|\nabla|} g$ is a solution to the free wave equation. If $ \lambda_1 \approx 1$, then $\mu \lesa \lambda_0 \approx 1$, and so \eqref{eqn:bil wave-schro I free case} follows from H\"older's inequality and the $L^2_t L^\infty_x$ Strichartz estimate for the free wave equation \cite{Keel1998}
        $$ \|  P_{\les \lambda_0}( \dot{P}_\mu v e^{it\Delta} f_{\lambda_1} ) \|_{L^2_{t,x}(\RR^{1+4})} \lesa \| \dot{P}_\mu v \|_{L^2_t L^\infty_x} \| f_{\lambda_1} \|_{L^2_x} \lesa \mu^{\frac{3}{2}} \| \dot{P}_{\mu} g \|_{L^2_x} \| f_{\lambda_1} \|_{L^2_x}. $$
On the other hand, if $\lambda_1 \gg 1$, then we decompose the Fourier support into cubes of size $\lambda_0$. More precisely, let $\mc{Q}_{\lambda_0}$ denote a decomposition of $\RR^4$ into cubes $q$ of diameter $\frac{{\lambda_0}}{100}$, and let $P_q$ be the corresponding Fourier localisation operators such that
    $$ f = \sum_{q \in \mc{Q}_\mu} P_q f, \qquad \supp \widehat{P_q f} \subset q.$$
Then decomposing $f$ and $g$ using the Fourier multipliers $P_q$, and noting that the Fourier support of the output is constrained to frequencies $\lesa \lambda_0$,
    \begin{align*}
      \| P_{\les \lambda_0}( \dot{P}_\mu v e^{it\Delta} f_{\lambda_1} ) \|_{L^2_{t,x}}
                                &\lesa \sum_{\substack{q, \tilde{q} \in \mc{Q}_{\lambda_0} \\ \dist(q, \tilde{q}) \lesa {\lambda_0}}} \| \dot{P}_\mu P_q v e^{it\Delta} P_{\tilde q} f_{\lambda_1} \|_{L^2_{t,x}} \\
      &\lesa \lambda_1^{-\frac{1}{2}} \big(\min\{\mu, \lambda_0\}\big)^{\frac{3}{2}} \sum_{\substack{q, \tilde{q} \in \mc{Q}_{\lambda_0} \\ \dist(q, \tilde{q}) \lesa {\lambda_0}}} \|\dot{P}_{\mu} P_q g\|_{L^2_x} \| P_{\tilde q} f_{\lambda_1} \|_{L^2_x}\\
                                &\lesa \Big( \frac{\min\{\mu, \lambda_0\}}{\lambda_1}\Big)^{\frac{1}{2}} \min\{\mu, \lambda_0\} \| \dot{P}_\mu v \|_{L^\infty_t L^2_x} \| f_{\lambda_1} \|_{L^2_{x}}
    \end{align*}
where the $L^2_{t,x}$ bound follows from a short computation using Plancherel (see, for instance, \cite[Lemma 2.6]{Candy2018} or \cite[Theorem 5.2]{Candy2017}).

In view of the estimate \eqref{eqn:bil wave-schro I free case}, a somewhat standard transference type argument implies that it suffices to prove that
        $$ \|  P_{\les \lambda_0}( \dot{P}_\mu v e^{it\Delta} u_{\lambda_1} ) \|_{L^2_{t,x}} \lesa  \lambda_1^{-\frac{1}{2}+\epsilon} \big(\min\{\mu, \lambda_0\}\big)^{\frac{3}{2}-\epsilon}  \|\dot{P}_\mu v \|_{L^\infty_t L^2_x} \big( \| u_{\lambda_1} \|_{L^\infty_t L^2_x} + \| (i\p_t + \Delta) u_{\lambda_1} \|_{L^2_t L^{\frac{4}{3}}_x}\big).$$
Again applying the estimate for free solutions, the Duhamel formula and the (dual) endpoint Strichartz estimate implies that the first estimate in Theorem \ref{thm:bilinear wave-schrodinger} would then follow from the inhomogeneous estimate
        \begin{equation}\label{eq:bilinear wave-schro inhom}
        \Big\| P_{\les \lambda_0} \Big[ \dot{P}_\mu v(t) \int_{-\infty}^t e^{i(t-s)\Delta} F_{\lambda_1}(s) ds \,\Big]\Big\|_{L^2_{t,x}}
                \lesa \lambda_1^{-\frac{1}{2}+\epsilon} \big(\min\{\mu, \lambda_0\}\big)^{\frac{3}{2}-\epsilon} \| v \|_{L^\infty_t L^2_x} \| F_{\lambda_1} \|_{L^2_t L^\frac{4}{3}_x}.
        \end{equation}
To prove \eqref{eq:bilinear wave-schro inhom}, we start by observing that for any intervals $I, J \subset \RR$ the estimate for free solutions \eqref{eqn:bil wave-schro I free case} together with the endpoint Strichartz estimate immediately implies that
        \begin{align}
          \Big\| P_{\les \lambda_0} \Big[ \dot{P}_\mu v(t) \ind_I(t) \int_J e^{i(t-s)\Delta} F_{\lambda_1}(s) ds \Big] \Big\|_{L^2_{t,x}}
                            &\lesa \lambda_1^{-\frac{1}{2}} \big(\min\{\mu, \lambda_0\}\big)^{\frac{3}{2}}   \| \dot{P}_\mu v \|_{L^\infty_t L^2_x} \Big\| \int_J e^{- i s \Delta} F_{\lambda_1}(s) ds \Big\|_{L^2_x} \notag \\
                            &\lesa \lambda_1^{-\frac{1}{2}} \big(\min\{\mu, \lambda_0\}\big)^{\frac{3}{2}}  \| \dot{P}_\mu  v \|_{L^\infty_t L^2_x} \| F_{\lambda_1}\|_{L^2_t L^\frac{4}{3}_x} \label{eq:temp wave-schro bilinear inhom}.
        \end{align}
If we instead put $F_{\lambda_1} \in L^{a'}_t L^{b'}_x$ for some non-endpoint Strichartz admissible pair $(a,b)$, then as $a>2$, Theorem \ref{thm:bilinear wave-schrodinger} would simply follow from an application of the Christ-Kiselev Lemma. This argument was used by Visan \cite[Lemma 2.5]{Visan2007} to prove a bilinear $L^2_{t, x}$ estimate for the Schr\"odinger equation. In the endpoint case $a=2$ the Christ-Kiselev Lemma does not apply, and we instead need combine the above argument with a Whitney decomposition and an estimate of Keel-Tao \cite{Keel1998} used in the proof of the endpoint Strichartz estimate.

\begin{proof}[Proof of Theorem \ref{thm:bilinear wave-schrodinger}] For ease of notation, we let $\sigma = \min\{\mu, \lambda_0\}$. A direct application of the estimate for free solutions, \eqref{eqn:bil wave-schro I free case}, implies that
  \begin{align}
        \Big\| P_{\les \lambda_0} \Big[ \dot{P}_\mu v(t)\int_{-\infty}^t e^{i(t-s)\Delta} F_{\lambda_1}(s) ds \Big]\Big\|_{L^2_{t,x}}
                &\les  \int_{\RR} \Big\| P_{\les \lambda_0} \Big[ \dot{P}_\mu v(t) e^{it \Delta} \big( e^{-is\Delta} F_{\lambda_1}(s)\big)\Big]\Big\|_{L^2_{t,x}} ds \notag \\
                &\lesa \lambda_1^{-\frac{1}{2}} \sigma^{\frac{3}{2}}  \|  \dot{P}_\mu v \|_{L^\infty_t L^2_x} \| F_{\lambda_1} \|_{L^1_t L^2_x}.
                \label{eqn:thm bilinear wave-schro:trivial L1 est}
  \end{align}
Arguing as above, another application of \eqref{eqn:bil wave-schro I free case}, together with the Duhamel formula shows that it suffices to prove \eqref{eq:bilinear wave-schro inhom}. Let $\D_j$ denote a decomposition of $\RR$ into left closed and right open intervals of length $2^j$. For intervals $I=[a_1, a_2)$ and $J=[b_1, b_2)$ we write $I\g J$ if $a_1\g b_2$. An application of Bernstein's inequality together with \cite[Lemma 4.1]{Keel1998} (and duality) implies that for all $r$ in a neighbourhood of $4$ and any $I, J \in \D_j$ such that $\dist(I, J) \approx 2^j$ we have
    \begin{align*}
       \Big\| P_{\les \lambda_0} \Big[ \dot{P}_\mu v(t) \ind_I(t) \int_J e^{i(t-s)\Delta} F_{\lambda_1}(s) ds \Big]\Big\|_{L^2_{t,x}}
                            &\lesa \sigma^{\frac{4}{r}} \| \dot{P}_\mu v \|_{L^\infty_t L^2_x} \Big\| \ind_I(t) \int_J e^{i(t-s)\Delta} F_{\lambda_1}(s) ds \Big\|_{L^2_t L^r_x} \\
                            &\lesa \sigma^{\frac{4}{r}} 2^{ - \frac{j}{2} (1 - \frac{4}{r})  } \| \dot{P}_\mu v \|_{L^\infty_t L^2_x} \| F_{\lambda_1} \|_{L^2_t L^\frac{4}{3}_x}.
    \end{align*}
Hence, taking $\theta \in [0, 1]$ and interpolating with \eqref{eq:temp wave-schro bilinear inhom} shows that for any $I, J \in \D_j$ with $\dist(I, J) \approx 2^j$ and any sufficiently small $\epsilon>0$, by choosing $r$ close to $4$ appropriately, we have
    \begin{align*}
        \Big\| P_{\les \lambda_0} \Big[ \dot{P}_\mu v(t) \ind_I(t) \int_J e^{i(t-s)\Delta} F_{\lambda_1}(s) ds \Big] \Big\|_{L^2_{t,x}}
                                    &\lesa \Big(\lambda_1^{-\frac{1}{2}} \sigma^{\frac{3}{2}}\Big)^{1-\theta} \Big( \sigma^{\frac{4}{r}} 2^{ - 2j (\frac{1}{4} - \frac{1}{r})  }\Big)^\theta \|\dot{P}_\mu v \|_{L^\infty_t L^2_x} \| F_{\lambda_1}\|_{L^2_t L^\frac{4}{3}_x}  \\
                                    &\les  \Big( \frac{\sigma}{{\lambda_1}}\Big)^{\frac{1-\theta}{2}} \sigma 2^{ - \epsilon \theta | j - k|}  \|\dot{P}_\mu v \|_{L^\infty_t L^2_x} \|F_{\lambda_1} \|_{L^2_t L^\frac{4}{3}_x}
    \end{align*}
where we choose $k \in \ZZ$ such that $2^{-\frac{1}{2} k } \approx \sigma$. To conclude the proof of \eqref{eq:bilinear wave-schro inhom}, we use the Whitney decomposition
         \begin{equation}\label{eq:whitney decomp} \ind_{\{t>s\}}(t,s) = \sum_{j \in \ZZ} \sum_{ \substack{ I, J \in \D_j, I\g J\\ \dist(I, J)\approx 2^j}} \ind_I(t) \ind_J(s) \qquad \qquad \text{ for a.e. $(t,s) \in \RR^2$} \end{equation}
and observe that for any $0<\theta<1$ we have
        \begin{align*}
            \Big\|P_{\les \lambda_0} \Big[ \dot{P}_\mu v(t) \int_{-\infty}^t e^{i(t-s)\Delta} F_{\lambda_1}(s) ds \Big] \Big\|_{L^2_{t,x}}
                    &\lesa \sum_{j \in \ZZ} \Big( \sum_{ \substack{ I, J \in \D_j, I\g J\\ \dist(I, J)\approx 2^j}} \Big\| P_{\les \lambda_0} \Big[ \dot{P}_\mu v(t)  \ind_I(t) \int_J e^{i(t-s)\Delta} F_{\lambda_1}(s) ds \Big] \Big\|_{L^2_{t,x}}^2\Big)^\frac{1}{2} \\
                    &\lesa \Big( \frac{\sigma}{{\lambda_1}}\Big)^{\frac{1-\theta}{2}}\sigma \| \dot{P}_\mu v \|_{L^\infty_t L^2_x} \sum_{j \in \ZZ}  2^{ - \epsilon \theta | j - k|} \Big( \sum_{ \substack{ I, J \in \D_j, I\g J\\ \dist(I, J)\approx 2^j}} \| \ind_J F_{\lambda_1}\|_{L^2_t L^\frac{4}{3}_x}^2 \Big)^\frac{1}{2}  \\
                    &\lesa  \Big( \frac{\sigma}{{\lambda_1}}\Big)^{\frac{1-\theta}{2}}\sigma \| \dot{P}_\mu v \|_{L^\infty_t L^2_x} \| F_{\lambda_1}\|_{L^2_t L^\frac{4}{3}_x}.
        \end{align*}
\end{proof}

The proof of the remaining bilinear estimate, Theorem
\ref{thm:bilinear schrodinger}, is more involved, as we are trying to
put the product into $L^1_t L^r_x$. In particular, unlike the proof of
Theorem \ref{thm:bilinear wave-schrodinger}, we cannot gain an
$\ell^2$ sum over the intervals $I \in \D_j$ before using the
corresponding bilinear restriction estimate for free
solutions. Instead, the key new ingredient is an \emph{atomic}
bilinear restriction estimate from \cite{Candy2017}. To state this result precisely, we need some additional notation. A function $\phi \in L^\infty_t L^2_x$ is an $\emph{atom}$ if we can write $ \phi(t) = \sum_{I} \ind_I(t) e^{it\Delta} f_I$, with the intervals $I \subset \RR$ forming a partition of $\RR$, and the $f_I:\RR^d \to \CC$ satisfy
        $$\Big( \sum_I \| f_I \|_{L^2_x}^2\Big)^\frac{1}{2} \les 1.$$
We then take
        $$U^2_\Delta = \{ \sum_j c_j \phi_j \mid \text{ $\phi_j$ an atom and } (c_j) \in \ell^1 \} $$
with the induced norm
        $$ \| u \|_{U^2_\Delta} = \inf_{ u = \sum_j c_j \phi_j} \sum_j |c_j|$$
where the infimum is over all representations of $u$ in terms of atoms. These spaces were introduced in unpublished work of Tataru, and studied in detail in \cite{Koch2005, Hadac2009}. The atomic bilinear restriction estimate we require is the following.

\begin{theorem}\label{thm:atomic bilinear restriction}
Let $r>\frac{5}{3}$. Then, for any $\mu \in 2^\ZZ$ and $w, u \in U^2_\Delta$ we have
        \begin{equation}\label{eq:prop atomic bilinear:high-high}
        			\| \dot{P}_\mu( \bar{w} u ) \|_{L^1_t L^r_x(\RR^{1+4})} \lesa \mu^{2 - \frac{4}{r}} \| u \|_{U^2_\Delta} \| v\|_{U^2_\Delta}.
        	\end{equation}
\end{theorem}
\begin{proof}
This is an application of \cite[Corollary 1.6]{Candy2017} together with an additional orthogonality argument. An application of Bernstein's inequality shows that it suffices to consider the case $\frac{5}{3}< r< 2$. Let $\mc{Q}_\mu$ denote a decomposition of $\RR^4$ into cubes $q$ of diameter $\frac{\mu}{100}$, and let $P_q$ be the corresponding Fourier localisation operators such that
    $$ f = \sum_{q \in \mc{Q}_\mu} P_q f, \qquad \supp \widehat{P_q f} \subset q.$$
A short computation using \cite[Corollary 1.6]{Candy2017} shows that for any $q, \tilde{q} \in \mc{Q}_\mu$ such that $\dist(q, \tilde{q}) \approx \mu$ we have
        $$ \| \overline{P_q w} P_{\tilde{q}} u \|_{L^1_t L^r_x} \lesa \mu^{2-\frac{4}{r}} \| P_q w \|_{U^2_\Delta} \| P_{\tilde{q}} u \|_{U^2_\Delta}. $$
Therefore, noting that $\dot{P}_\mu ( \overline{P_q w} P_{\tilde{q}} u ) = 0$ unless $\dist(q, \tilde{q})\approx \mu$, we conclude that
    \begin{align*}
        \| \dot{P}_\mu( \bar{w} u ) \|_{L^1_t L^r_x}
                        &\lesa \sum_{\substack{ q, \tilde{q} \in \mc{Q}_\mu \\ \dist(q, \tilde{q}) \approx \mu}} \| \overline{P_q w} P_{\tilde{q}} u \|_{L^1_t L^r_x } \\
                        &\lesa \mu^{2 - \frac{4}{r}} \Big( \sum_{q \in \mc{Q}_\mu} \| P_q w\|_{U^2_\Delta}^2 \Big)^\frac{1}{2} \Big( \sum_{\tilde{q} \in \mc{Q}_\mu} \| P_{\tilde{q}} u \|_{U^2_\Delta}^2 \Big)^\frac{1}{2} \\
                        &\lesa \mu^{2 - \frac{4}{r}} \| w \|_{U^2_\Delta} \| u \|_{U^2_\Delta}
    \end{align*}
where the last line follows from the fact that square sums of almost orthogonal Fourier multipliers is bounded in $U^2_\Delta$, see for instance  \cite[Proposition 4.3]{Candy2018a}.
\end{proof}

We now turn to the proof of Theorem \ref{thm:bilinear schrodinger}.

\begin{proof}[Proof of Theorem \ref{thm:bilinear schrodinger}]
We adopt the notation used in the proof of Theorem \ref{thm:bilinear wave-schrodinger}, thus $\D_j$ denotes a set of intervals of length $2^j$ which form a partition of $\RR$. The first step in the proof of  is to show that
        \begin{equation}\label{eq:proof of bil schro:mixed}
            \| \dot{P}_\mu[\bar{w} u] \|_{L^1_t L^r_x} \lesa \mu^{2-\frac{4}{r}} \| w \|_{U^2_\Delta} \Big( \| u \|_{L^\infty_t L^2_x} + \| (i \p_t + \Delta) u \|_{L^2_t L^\frac{4}{3}_x} \Big),
        \end{equation}
thus we can replace one of the $U^2_\Delta$ norms with the inhomogeneous Strichartz type norm. Similar to the proof of Theorem \ref{thm:bilinear wave-schrodinger}, in view of Theorem \ref{thm:atomic bilinear restriction}, it suffices to prove that
        $$
            \Big\| \dot{P}_\mu\Big[\bar{w} \int_{-\infty}^t e^{i(t-s)\Delta} F(s) ds \Big] \Big\|_{L^1_t L^r_x} \lesa \mu^{2-\frac{4}{r}} \| w \|_{U^2_\Delta} \| F \|_{L^2_t L^\frac{4}{3}_x}.
        $$
We start by observing that since $\sum_I \ind_I(t) e^{it\Delta} \int_\RR e^{-is \Delta} F_I(s) ds$ is a rescaled atom, an application of the endpoint Strichartz estimate gives the bound
   \begin{align}
     \Big\| \sum_{\substack{I, J \in \D_j, I \g J \\ \dist(I, J) \approx  2^j}} \ind_I(t) \int_\RR e^{i(t-s) \Delta} \ind_J(s) F(s) ds \Big\|_{U^2_\Delta}
                &\lesa  \Big( \sum_{J \in \D_j} \Big\| \int_\RR e^{-i s \Delta} \ind_J(s)  F(s) ds  \Big\|_{L^\infty_t L^2_x}^2 \Big)^\frac{1}{2} \notag \\
                &\lesa \Big( \sum_{J \in \D_j} \|  \ind_J  F  \|_{L^2_t L^\frac{4}{3}_x}^2 \Big)^\frac{1}{2} = \| F \|_{L^2_t L^\frac{4}{3}_x}. \label{eq:proof of bil schro:U2 bound}
   \end{align}
Consequently, an application of Theorem \ref{thm:atomic bilinear restriction} implies that
    \begin{align}
        \Big\| \dot{P}_{\mu}\Big[ \overline{w} \sum_{\substack{I, J \in \D_j, I \g J \\ \dist(I, J) \approx  2^j}} \ind_I(t) &\int_\RR e^{i(t-s) \Delta} \ind_J(s) F(s) ds \Big]\Big\|_{L^1_t L^r_x} \notag \\
        &\lesa \mu^{2 - \frac{4}{r}}  \| w \|_{U^2_{\Delta}}  \Big\| \sum_{\substack{I, J \in \D_j,  I \g J  \\ \dist(I, J) \approx  2^j}} \ind_I(t) \int_\RR e^{i(t-s) \Delta} \ind_J(s) F(s) ds \Big\|_{U^2_\Delta}  \notag \\
        &\lesa \mu^{2-\frac{4}{r}}  \| w\|_{U^2_\Delta} \| F \|_{L^2_t L^\frac{4}{3}_x}. \label{eq:thm inhom schro:bilinear temp}
    \end{align}
On the other hand, as in the proof of Theorem \ref{thm:bilinear wave-schrodinger}, to gain decay in $j$, we again use an application of \cite[Lemma 4.1]{Keel1998} and observe that for every $p$ in a neighbourhood of $4$
    \begin{align*}
        \Big\|\sum_{\substack{I, J \in \D_j, I \g  J \\ \dist(I, J) \approx  2^j}} \ind_I(t) \int_\RR e^{i(t-s) \Delta} \ind_J(s) F(s) ds \Big\|_{L^2_t L^p_x}
                    &\lesa \Big( \sum_{\substack{I, J \in \D_j, I \g  J \\ \dist(I, J) \approx  2^j}}\Big\| \ind_I(t) \int_\RR e^{i(t-s) \Delta} \ind_J(s) F(s) ds \Big\|_{L^2_t L^p_x}^2 \Big)^\frac{1}{2}\\
                    &\lesa 2^{- \frac{j}{2}( 1- \frac{4}{p})} \| F \|_{L^2_t L^\frac{4}{3}_x}
    \end{align*}
and hence for every $r$ in a neighbourhood of $2$ we have
    \begin{align}
        \Big\| \dot{P}_\mu\Big[ \bar{w} \sum_{\substack{I, J \in \D_j, I \g  J \\ \dist(I, J) \approx  2^j}} &\ind_I(t) \int_\RR e^{i(t-s) \Delta} \ind_J(s) F(s) ds \Big] \Big\|_{L^1_t L^r_x}\notag \\
            &\lesa  \| w\|_{L^2_t L^4_x} \Big\|\sum_{\substack{I, J \in \D_j, I \g  J \\ \dist(I, J) \approx  2^j}} \ind_I(t) \int_\RR e^{i(t-s) \Delta} \ind_J(s) F(s) ds \Big\|_{L^2_t L^{\frac{4r}{4-r}}_x}  \notag \\
            &\lesa 2^{j(\frac{2}{r}-1)} \| w \|_{U^2_\Delta}  \| F \|_{L^2_t L^\frac{4}{3}_x}. \label{eq:thm inhom schro:strichartz temp}
    \end{align}
Let $\frac{5}{3}<r_1 \les r \les r_2 $ with $r_2$ in a neighbourhood of $2$, and take $0<\theta<1$ such that $\frac{1}{r} = \frac{\theta}{r_1} + \frac{1-\theta}{r_2}$. Interpolating between \eqref{eq:thm inhom schro:bilinear temp} and \eqref{eq:thm inhom schro:strichartz temp}, and choosing $r_2$ appropriately,  we obtain for every sufficiently small $\epsilon>0$
    \begin{align*}
        \Big\| \dot{P}_\mu\Big[ \bar{w}  \sum_{\substack{I, J \in \D_j, I \g  J \\ \dist(I, J) \approx  2^j}} \ind_I(t) \int_\RR e^{i(t-s) \Delta} \ind_J(s) F(s) ds \Big] \Big\|_{L^1_t L^r_x} &\lesa \mu^{\theta(2 - \frac{4}{r_1})}  2^{ j( \frac{2}{r_2} -1)(1-\theta)} \| w \|_{U^2_\Delta} \| F \|_{L^2_t L^\frac{4}{3}_x}  \\
        &\approx \mu^{2-\frac{4}{r}} 2^{ - \epsilon |j-k| }  \| w \|_{U^2_\Delta} \| F \|_{L^2_t L^\frac{4}{3}_x}
    \end{align*}
where we take $k \in \ZZ$ such that $2^{-\frac{k}{2}} \approx \mu$. Therefore, applying the Whitney decomposition \eqref{eq:whitney decomp}, we conclude that
    \begin{align*}
      \Big\| \dot{P}_\mu\Big[ \bar{w} \int_{-\infty}^t  e^{i(t-s) \Delta} \ind_J(s) F(s) ds \Big] \Big\|_{L^1_t L^r_x}
                &\les \sum_{j \in \ZZ} \Big\| \dot{P}_\mu\Big[ \bar{w} \sum_{\substack{I, J \in \D_j, I \g  J \\ \dist(I, J) \approx  2^j}} \ind_I(t) \int_\RR e^{i(t-s) \Delta} \ind_J(s) F(s) ds \Big] \Big\|_{L^1_t L^r_x}  \\
                &\lesa \mu^{2-\frac{4}{r}} \sum_{j \in \ZZ}  2^{ -\epsilon |j-k|}  \| w \|_{U^2_\Delta} \| F \|_{L^2_t L^\frac{4}{3}_x} \lesa \mu^{2-\frac{4}{r}} \|w \|_{U^2_\Delta} \| F \|_{L^2_t L^\frac{4}{3}_x}
    \end{align*}
and hence \eqref{eq:proof of bil schro:mixed} follows.

The next step is to replace the $U^2_\Delta$ norm of $w$ with the required inhomogeneous norm. This follows by essentially repeating the above argument, but using \eqref{eq:proof of bil schro:mixed} in place of Theorem \ref{thm:atomic bilinear restriction}. More precisely, an application of \eqref{eq:proof of bil schro:mixed} and the $U^2_\Delta$ bound \eref{eq:proof of bil schro:U2 bound} gives
    \begin{align*}
        \Big\| \dot{P}_\mu \Big[ \sum_{\substack{I, J \in \D_j, I \g J \\ \dist(I, J) \approx  2^j}} \ind_I(t) \int_J \overline{e^{i(t-s) \Delta}  G(s)} ds \, u(t)\Big] \Big\|_{L^1_t L^r_x} \lesa \mu^{2 - \frac{4}{r}} \|G\|_{L^2_t L^\frac{4}{3}_x} \Big( \| u \|_{L^\infty_t L^2_x} + \| (i \p_t + \Delta) u \|_{L^2_t L^\frac{4}{3}_x} \Big).
    \end{align*}
On the other hand, again applying \cite[Lemma 4.1]{Keel1998} we have for every $r$ in a neighbourhood of $2$
    \begin{align*}
       \Big\| \dot{P}_\mu \Big[ \sum_{\substack{I, J \in \D_j, I \g J \\ \dist(I, J) \approx  2^j}} &\ind_I(t) \int_J \overline{e^{i(t-s) \Delta}  G(s)} ds \, u(t)\Big] \Big\|_{L^1_t L^r_x}\\
            &\lesa \Big\|  \sum_{\substack{I, J \in \D_j, I \g J \\ \dist(I, J) \approx  2^j}} \ind_I(t) \int_\RR e^{i(t-s) \Delta} \ind_J(s) G(s) ds  \Big\|_{L^2_t L^\frac{4r}{4-r}_x} \| u \|_{L^2_t L^4_x} \\
            &\lesa  2^{j(\frac{2}{r}-1)} \| G \|_{L^2_t L^\frac{4}{3}_x} \Big( \| u \|_{L^\infty_t L^2_x} + \| ( i \p_t + \Delta ) u \|_{L^2_t L^\frac{4}{2}_x} \Big).
    \end{align*}
Therefore, as in the proof of \eref{eq:proof of bil schro:mixed}, the required bound now follows by interpolation, together with the Whitney decomposition \eref{eq:whitney decomp}. Finally, to replace the norms $\| u \|_{L^\infty_t L^2_x} + \| (i\p_t + \Delta) u \|_{L^2_t L^{\frac43}_x}$ with the norm $\| u \|_Z$ follows from the trivial $L^1_t L^2_x$ transference type argument outlined in \eqref{eqn:thm bilinear wave-schro:trivial L1 est}.
\end{proof}

\section{Refined bilinear estimates}\label{sec:proof of res gain}
In this section we prove two estimates. The first is a version of Theorem \ref{thm:main schro est nonres} with a Besov refinement, which is used to control the error terms in the profile decomposition. The second estimate in this section is a version of the inhomogeneous endpoint Strichartz estimate with two spatially diverging potentials. Again, this estimate plays a key role in the proof of the uniform Strichartz estimate.

\subsection{Refinement of Theorem \ref{thm:main schro est nonres}}\label{subsec:res-int}
The following estimate is the main goal of this section, and it is also one of the core ingredients of this paper.

\begin{theorem} \label{thm:besov gain}
Let $0\les s < 1$ and $d=4$. There exists $0< \theta <1$ and $r>2$ such that for any $\lambda_0, \lambda_1, \lambda_2 \in 2^\NN $, any free solution to the wave equation $v = e^{\pm it|\nabla|} f \in L^\infty_t L^2_x$, and any $\chi(t) \in L^\infty\cap C^\infty(\RR)$ satisfying
    $$ \supp \widetilde{(\chi v)} \subset \{ |\tau| \lesa \lr{\xi}\}$$
we have
        \begin{equation}\label{eqn:thm besov gain:high-low gain} \| P_{\lambda_0}(\chi v  u_{\lambda_1}) \|_{N^s_{\lambda_0}} \lesa  \Big( \frac{\min\{\lambda_0, \lambda_1\}}{\max\{\lambda_0, \lambda_1\}}\Big)^\theta  \|\chi \|_{L^\infty_t} \| v\|_{L^\infty_t L^2_x} \|u_{\lambda_1} \|_{\underline{S}^s_{\lambda_1}} \end{equation}
and
         \begin{equation}\label{eqn:thm besov gain:besov gain} \| \mc{I}_0[P_{\lambda_0}(\chi v_{\lambda_2} u_{\lambda_1})] \|_{S^s_{\lambda_0}} \lesa \Big( \frac{\lambda_{\min}}{\lambda_{\max}}\Big)^\theta \| \chi v_{\lambda_2} \|_{L^\infty_t B^{\frac{4}{r}-2}_{r, \infty}}^\theta \Big(\|\chi \|_{L^\infty_t} \| v_{\lambda_2} \|_{L^\infty_t L^2_x}\Big)^{1-\theta} \| u_{\lambda_1} \|_{\underline{S}^s_{\lambda_1}} \end{equation}
where $\lambda_{\min} = \min\{\lambda_0, \lambda_1, \lambda_2\}$ and $\lambda_{\max} = \max\{\lambda_0, \lambda_1, \lambda_2\}$.
\end{theorem}
\begin{proof}
The case $\lambda_0 \gg \lambda_1$ of both \eqref{eqn:thm besov gain:high-low gain} and \eqref{eqn:thm besov gain:besov gain} follows from \eqref{eqn:thm multilinear est gen:main high-low} and Lemma \ref{lem:energy ineq}. On the other hand, if $\lambda_0 \ll \lambda_1$, we have to work a little harder as the estimates in Theorem \ref{thm:main schro est nonres} only suffice when $s>0$ in this case. We first observe that in view of the non-resonant identity $P_{\lambda_0}(\chi v \PN{\lambda_1}u ) = \PF{\lambda_0}( \chi v_{\approx \lambda_1} \PN{\lambda_1} u_{\lambda_1})$, together with the definition of $N_{\lambda_0}$ and $S_{\lambda_1}$, Theorem \ref{thm:bilinear wave-schrodinger} implies that for any $0\les \alpha < \frac{1}{2}$
       $$
            \| P_{\lambda_0}(\chi v_{\approx \lambda_1} \PN{\lambda_1}u ) \|_{N^s_{\lambda_0}} \les \lambda_0^{s-1} \|\chi\|_{L^\infty} \| P_{\lambda_0}(v_{\approx \lambda_1} \PN{\lambda_1}u ) \|_{L^2_{t, x}} \lesa \Big( \frac{\lambda_0}{\lambda_1} \Big)^{s+\alpha} \|\chi\|_{L^\infty}\| v_{\approx \lambda_1} \|_{L^\infty_t L^2_x} \| u_{\lambda_1} \|_{\underline{S}^s_{\lambda_1}}.
       $$
On the other hand, an application of H\"older's inequality gives
       $$  \| P_{\lambda_0}(\chi v_{\approx \lambda_1} \PN{\lambda_1}u ) \|_{N^s_{\lambda_0}} \les \lambda_0^{s-1} \| P_{\lambda_0}(\chi v_{\approx \lambda_1} \PN{\lambda_1}u ) \|_{L^2_{t, x}}
                            \lesa \Big( \frac{\lambda_0}{\lambda_1}\Big)^{s+\frac{4}{r}-2} \| \chi v_{\approx \lambda_1} \|_{L^\infty_t B^{\frac{4}{r}-2}_{r, \infty}} \| u_{\lambda_1} \|_{\underline{S}^s_{\lambda_1}}.$$
Combining these bounds with \eqref{eqn:thm multilinear est gen:main low-high with gain}, both \eqref{eqn:thm besov gain:high-low gain} and \eqref{eqn:thm besov gain:besov gain} follow in the case $\lambda_0 \ll \lambda_1$.

It remains to consider the case $\lambda_0 \approx \lambda_1$. The estimate \eqref{eqn:thm besov gain:high-low gain} follows immediately from \eqref{eqn:thm multilinear est gen:main high-high res} and \eqref{eqn:thm multilinear est gen:main high-high nonres}. On the other hand, in view of the non-resonant bound  \eqref{eqn:thm multilinear est gen:main high-high nonres}, to prove \eqref{eqn:thm besov gain:besov gain} it suffices to show that for $ \lambda_2 \lesa \lambda_0 \approx \lambda_1$ we have
        $$ \lambda_0^s \| \PN{\lambda_0} \mc{I}_0[ \chi v_{{\lambda_2}} \PN{\lambda_1} u_{\lambda_1}] \|_{L^2_t L^4_x} \lesa \Big( \frac{{\lambda_2}}{\lambda_{1}}\Big)^\theta \| v_{{\lambda_2}} \|_{L^\infty_t L^2_x}^{1-\theta} \|\chi\|_{L^\infty}^{1-\theta} \| \chi v_{{\lambda_2}} \|_{L^\infty_t B^{\frac{d}{r}-2}_{r, \infty}}^\theta \| u_{\lambda_1} \|_{\underline{S}^s_{\lambda_1}}.$$
An application of the energy estimate in Lemma \ref{lem:energy ineq} together with the definition of the norm $\| \cdot \|_{\underline{S}^s_{\lambda_0}}$ shows that it is enough to prove the dual formulation
    \begin{equation}\label{eqn:thm res:main est}
        \begin{split}
        \Big| \int_{\RR^{1+4}} \bar{w}_{\lambda_0} \chi v_{\lambda_2} u_{\lambda_1} dx dt \Big|
                    &\lesa \Big( \frac{{\lambda_2}}{\lambda_0} \Big)^\theta  \| \chi v_{\lambda_2} \|_{L^\infty_t B^{\frac{4}{r}-2}_{r, \infty}}^\theta \|\chi\|_{L^\infty}^{1-\theta}\| v_{\lambda_2} \|_{L^\infty_t L^2_x}^{1-\theta} \| w_{\lambda_0} \|_Z \| u_{\lambda_1} \|_Z
        \end{split}
    \end{equation}
By multiplying $u$ and $w$ with a constant, we may assume that
      $$  \| u_{\lambda_1}\|_{Z} = \| w_{\lambda_0}\|_{Z} = 1.$$
Let $0<p<1$, $ \frac{5}{3}<q<2$, $r>2$, and $0<\theta<1$ such that
    $$ 1 = \frac{\theta}{p} + (1-\theta) \frac{3}{4}, \qquad \frac{1}{p} = \frac{1}{q} + \frac{1}{r}, \qquad 4\Big( \frac{1}{r} - \frac{1}{2}\Big) \theta + \frac{1-\theta}{4} > 0,$$
which is easily obtained by taking $r>2$ close enough to $2$ for any fixed $q\in(5/3,2)$, with $p$ and $\theta$ determined by the above equations (one potential choice is $\frac{1}{p} = 1 + \frac{1}{24}$, $\frac{1}{q} = \frac{1}{2} + \frac{1}{24} + \frac{1}{100}$, $\frac{1}{r} = \frac{1}{2} -\frac{1}{100}$, and $\theta = \frac{6}{7}$). The convexity of $L^p_x$ spaces implies that (recall that $v_1$ contains all frequencies less than $1$)
    \begin{align*}
         \Big| \int_{\RR^{1+4}} \bar{w}_{\lambda_0} \chi v_{\lambda_2} u_{\lambda_1} dx dt \Big|
            &\les \sum_{\substack{\mu \in 2^\ZZ \\ \lambda_2 -1 \lesa \mu \lesa \lambda_2}} \Big| \int_{\RR^{1+4}}  \dot{P}_{\mu} ( \overline{w}_{\lambda_0} u_{\lambda_1})  \chi \dot{P}_{\approx \mu} v_{\lambda_2} dx dt \Big|\\
                &\les \sum_{\substack{\mu \in 2^\ZZ \\ \lambda_2 -1 \lesa \mu \lesa \lambda_2}}  \|  \dot{P}_{\mu} ( \overline{w}_{\lambda_0} u_{\lambda_1})  \chi \dot{P}_{\approx \mu} v_{\lambda_2} \|_{L^1_t L^p_x}^\theta \|\chi\|_{L^\infty}^{1-\theta}\| \dot{P}_{\mu} ( \overline{w}_{\lambda_0} u_{\lambda_1}) \dot{P}_{\approx \mu}v_{\lambda_2}\|_{L^1_t L^{\frac{4}{3}}_x}^{1-\theta}.
    \end{align*}
Applying Theorem \ref{thm:bilinear schrodinger} we see that
    \begin{align*}
       \|  \dot{P}_{\mu} ( \overline{w}_{\lambda_0} u_{\lambda_1})  \chi  \dot{P}_{\approx \mu} v_{\lambda_2} \|_{L^1_t L^p_x}
        &\les  \|  \dot{P}_{ \mu} ( \overline{w}_{\lambda_0} u_{\lambda_1}) \|_{L^1_t L^q_x} \|  \chi v_{\lambda_2} \|_{L^\infty_t L^r_x} \\
        &\lesa {\mu}^{2-\frac{4}{q}} \|  \chi v_{{\lambda_2}} \|_{L^\infty_t L^r_x} \les {\mu}^{ 2-\frac{4}{q}} \lambda_2^{2 - \frac{4}{r}} \| \chi v_{\lambda_2} \|_{L^\infty_t B^{\frac{4}{r}-2}_{r,\infty}}.
    \end{align*}
Let $K_{\mu}(y)$ denote the kernel of the Fourier multiplier $\dot{P}_{\mu}$. Note that $\| K_{\mu} \|_{L^1(\RR^4)} \lesa 1$. Hence
Theorem \ref{thm:bilinear wave-schrodinger} together with the translation invariance of $L^p$ spaces gives
    \begin{align*}
      \|  \dot{P}_{\mu} ( \overline{w}_{\lambda_0} u_{\lambda_1})  \dot{P}_{\approx \mu} v_{\lambda_2} \|_{L^1_t L^\frac{4}{3}_x} &= \Big\| \int_{\RR^4} K_{\mu}(y) \overline{w}_{\lambda_0}(t, x-y) u_{\lambda_1}(t, x-y) dy  \dot{P}_{\approx \mu} v_{\lambda_2}(t, x) \Big\|_{L^1_t L^\frac{4}{3}_x} \\
            &\les \int_{\RR^4} |K_{\mu}(y)| \| u_{\lambda_1}(t, x-y) \dot{P}_{\approx \mu} v_{\lambda_2}(t,x) \|_{L^2_{t,x}} dy \| w_{\lambda_0} \|_{L^2_t L^4_x} \\
            &\lesa {\mu} \Big( \frac{{\mu}}{\lambda_1} \Big)^{\frac{1}{4}} \| v_{\lambda_2} \|_{L^\infty_t L^2_x}.
    \end{align*}
Since $4(\frac{1}{r}-\frac{1}{2})\theta + \frac{1-\theta}{4}>0$, we conclude that
    \begin{align*}
       \Big| \int_{\RR^{1+4}} \bar{w}_{\lambda_0}  \chi v_{\lambda_2} u_{\lambda_1} dx dt \Big|
                    &\lesa \sum_{\substack{\mu \in 2^\ZZ \\ \lambda_2 -1 \lesa \mu \lesa \lambda_2}}  \Big[ {\mu}^{ 2-\frac{4}{q}} \lambda_2^{2 - \frac{4}{r}} \|  \chi v_{\lambda_2} \|_{L^\infty_t \dot{B}^{\frac{4}{r}-2}_{r,\infty}} \Big]^\theta \Big[ \mu \Big( \frac{{\mu}}{\lambda_1} \Big)^{\frac{1}{4}}  \| v_{\lambda_2} \|_{L^\infty_t L^2_x} \Big]^{1-\theta}\\
                    &= \sum_{\substack{\mu \in 2^\ZZ \\ \lambda_2 -1 \lesa \mu \lesa \lambda_2}}  \mu^{4\theta(\frac{1}{r}-\frac{1}{2}) + \frac{1-\theta}{4}} \lambda_2^{4\theta(\frac{1}{2}-\frac{1}{r})} \lambda_1^{-\frac{1-\theta}{4}} \|  \chi v_{\lambda_2} \|_{L^\infty_t \dot{B}^{\frac{4}{r}-2}_{r,\infty}} ^\theta  \| v_{\lambda_2} \|_{L^\infty_t L^2_x}^{1-\theta}\\
                   &\lesa \Big( \frac{{\lambda_2}}{\lambda_1} \Big)^{\frac{1-\theta}{4}} \| v_{\lambda_2} \|_{L^\infty_t B^{\frac{4}{r}-2}_{r,\infty}}^\theta \| v_{\lambda_2} \|_{L^\infty_t L^2_x}^{1-\theta}
    \end{align*}
and therefore \eqref{eqn:thm res:main est} follows.
\end{proof}

As an easy corollary, we obtain the following.

\begin{corollary}\label{cor:summed up besov gain}
Let $d=4$ and $0\les s < 1$. There exist $0<\theta<1$ and $r>2$ such that for any $|\ell|< \theta$ and any free wave $v=e^{\pm it |\nabla|} f$ we have
    $$ \| \mc{I}_0( v u) \|_{S^s} \lesa \| v \|_{L^\infty_t B^{\frac{4}{r}-2+\ell}_{r, \infty}}^\theta \|  v \|_{L^\infty_t H^\ell}^{1-\theta} \| u \|_{\underline{S}^{s-\ell}}. $$
\end{corollary}

\begin{remark}
Although the statement of Corollary \ref{cor:summed up besov gain} gives the crucial Besov gain required in later sections, it has the unfortunate problem that it only gives control over the weaker space $S^s$, but requires that we have $u \in \underline{S}^s$ in the stronger space. There are a number ways to resolve this difficulty. One option is to essentially iterate the equation twice, since Theorem \ref{thm:main schro est nonres} roughly shows that $\mc{I}_0$ maps $S^s$ to $\underline{S}^s$.

An alternative is approach is to exploit complex interpolation. More precisely, note that we can write Theorem \ref{thm:main schro est nonres} and Corollary \ref{cor:summed up besov gain} in the form
        $$ \| \mc{I}_0 v \|_{B(S^s \to \underline{S}^s)} \lesa \| v \|_{L^\infty_t L^2_x}, \qquad \| \mc{I}_0 v \|_{B(\underline{S}^s \to S^s)} \lesa \| v \|_{L^\infty_t B^{\frac{4}{r}-2}_{r, \infty}}^\theta \|  v \|_{L^\infty_t L^2_x}^{1-\theta}$$
for any free wave $v$. As observed above, we only have the crucial factor $\| v \|_{L^\infty_t B^{\frac{4}{r}-2}_{r, \infty}}$ when we map the strong space $\underline{S}^s$ to the weak space $S^s$. This deficiency can be resolved with a small twist in the function spaces. More precisely, define the complex interpolation space
\EQ{
 \pt S^s_{1/2}:=[S^s,\underline{S}^s]_{1/2}.}
Then the above bounds immediately imply that we have
    $$ \| \mc{I}_0 v \|_{B( S^s_{1/2} \to S^s_{1/2})} \lesa \| v \|_{L^\infty_t B^{\frac{4}{r}-2}_{r, \infty}}^{\frac{1}{2} \theta} \| v \|_{L^\infty_t L^2_x}^{1-\frac{1}{2} \theta}. $$
This has the important advantage that we map $S^s_{1/2}$ to $S^s_{1/2}$ but retain a power of the Besov norm.
\end{remark}

\subsection{Decay by spatial separation}\label{subsec:spatial-sep}
For the uniform Strichartz estimate in the non-radial case, we need an
extra decay for potentials separating in space. We use the notation
$f_a(x)=f(x-a)$ for translates of $f$ (only within this subsection).
\begin{lemma} \label{lem:x sep}
Let $d\ge 3$, $f,g\in L^{d/2}(\R^d)$ and $\e>0$. Then there exists $D>0$ such that for any $a,b\in\R^d$ satisfying $|a-b|\ge D$, and any $F\in L^2_tL^{2^*}_x$, we have
\EQ{
 \|f_a\mc{I}_0[g_bF]\|_{L^2_tL^{2_*}_x} \le \e \|F\|_{L^2_tL^{2^*}_x}.}
\end{lemma}
\begin{proof}
First, the double endpoint Strichartz with H\"older implies
\EQ{
 \|f_a\mc{I}_0[g_bF]\|_{L^2_tL^{2_*}} \pt\le \|f\|_{L^{d/2}}\|\mc{I}_0[g_bF]\|_{L^2_tL^{2^*}_x}
 \pr\lec \|f\|_{L^{d/2}}\|g_bF\|_{L^2_tL^{2_*}_x}
 \lec \|f\|_{L^{d/2}}\|g\|_{L^{d/2}}\|F\|_{L^2_tL^{2^*}_x},}
which allows us to approximate $f$ and $g$ in $L^{d/2}$ by Schwartz functions.

Second, in order to dispose of the high frequency of $F$, we use the local smoothing estimate in a global form:
\EQ{
 \|\na\mc{I}_0 f\|_{\ell^\I_\alpha L^2_{t, Q_\alpha}} \lec \|f\|_{\ell^1_\alpha L^2_{t,Q_\alpha}},}
where $\{Q_\alpha\}_{\alpha\in\Z^d}$ is the decomposition of $\R^d$
into  unit cubes.
By approximating $f,g$ in $L^{d/2}$, we may assume that $\hat f,\hat g\in C^\I(\R^d)$ with $\supp\hat f,\supp\hat g\subset\{|\x|<R\}$ for some $R\in(1,\I)$. Then for $\forall\la\gg R$, we have
\EQ{
 \supp\mc{F} f_a\mc{I}_0 [g_bF_{>\la}] \subset\{|\x|>\la/2\},}
so that we can gain as follows, with $u:=\mc{I}_0[g_b F_{>\la}]$,
\EQ{
 \pt\|f_a u\|_{L^2_tL^{2_*}_x}
 \lec \la^{-1}\|\na (f_a u)\|_{L^2_tL^{2_*}_x}
 \pr\lec \la^{-1}\|\na f\|_{L^{d/2}}\|u\|_{L^2_tL^{2^*}_x} +
 \la^{-1}\|f_a\|_{\ell^2_\alpha L^d(Q_\alpha)}\|\na u\|_{\ell^\I_\alpha L^2_{t,Q_\alpha}}
 \pr\lec \la^{-1}\|\na f\|_{L^{d/2}}\|g\|_{L^{d/2}}\|F\|_{L^2_tL^{2^*}_x}
 \pn+ \la^{-1}\|f\|_{\ell^2_\alpha L^d(Q_\alpha)}\|\na u\|_{\ell^\I_\alpha L^2_{t,Q_\alpha}}, }
where the last norm is bounded by using the above smoothing estimate
\EQ{
 \|\na u\|_{\ell^\I_\alpha L^2_{t,Q_\alpha}}
 \lec \|g_b F_{>\la}\|_{\ell^1_\alpha L^2_{t,Q_\alpha}} \pt\lec \|g\|_{\ell^{2_*}_\alpha L^d_x(Q_\alpha) }\|F_{>\la}\|_{\ell^{2^*}_\alpha L^2_tL^{2^*}_x(Q_\alpha)}
 \pr\lec  \|g\|_{\ell^{2_*}_\alpha L^d_x(Q_\alpha) }\|F\|_{L^2_tL^{2^*}_x}. }
Thus we can dispose of the high frequency contribution from $F_{>\la}$ for some large $\la$.

Third, in order to dispose of long time interactions, we exploit the dispersive decay estimate (in $L^\I_x$).
Decomposing the Duhamel integral by
\EQ{
 \mc{I}_0f \pt= \int_{t-L}^t \mc{U}_0(t-s)f(s)ds + \int_0^{t-L} \mc{U}_0(t-s)f(s)ds
 =: \mc{I}_{<L}f+ \mc{I}_{>L}f,}
we have, using the dispersive estimate, as well as Young and H\"older,
\EQ{
 \pt\|f_a\mc{I}_{>L}g_b F\|_{L^2_tL^{2_*}_x}
 \pr\lec \|f\|_{L^{2_*}} \|\int_{|t-s|>L} |t-s|^{-d/2}\|g_b F(s)\|_{L^1_x} ds\|_{L^2_t}
 \pr\lec \|f\|_{L^{2_*}}  \||t|^{-d/2}\|_{L^1(|t|>L)} \|g\|_{L^{2_*}} \|F\|_{L^2_tL^{2^*}_x}
 \pn\lec L^{-d/2+1}\|f\|_{L^{2_*}}\|g\|_{L^{2_*}} \|F\|_{L^2_tL^{2^*}_x}, }
so that we can dispose of $\mc{I}_{>L}$ for some large $L>1$.

Thus the problem is reduced to the decay of
\EQ{
 f_a\mc{I}_{<L}[g_bF_{<\la}].}
Now that both the traveling time and the frequency (group velocity) are bounded, we can exploit the spatial separation $|a-b|\to\I$. An easy way is to use
\EQ{
  x\mc{U}_0(t) = \mc{U}_0(t)(x-2it\na).}
By another approximation of $f,g\in L^{d/2}$, we may now assume that $\supp f,\supp g\subset\{|x|<S\}$ for some $S\in(0,\I)$. If $|a-b|\ge D\gg S$, then we have, with $u:=\mc{I}_{<L}[g_bF_{<\la}]$,
\EQ{
 \|f_au\|_{L^2_tL^{2_*}} \pt\lec D^{-1}\|(x-b)f(x-a)u\|_{L^2_tL^{2_*}}
 \pr=D^{-1}\|f(x-a)\mc{I}_{<L}[(x-b-2i(t-s)\na)g(x-b)F_{<\la}]\|_{L^2_tL^{2_*}}
 \pr\lec D^{-1}\|f\|_{L^{d/2}}\|xg\|_{L^{d/2}}\|F_{<\la}\|_{L^2_tL^{2^*}_x}
 +D^{-1}\|f\|_{L^{d/2}}\|\mc{I}_{<L}[(t-s)\na (g_bF_{<\la})]\|_{L^2_tL^{2^*}_x}, }
where the last norm is bounded by using Sobolev, Young and H\"older
\EQ{
 \|\int_{|t-s|<L}\|\De (g_bF_{<\la}(s))\|_{L^2_x} ds \|_{L^2_t}
 \pt\le \|1\|_{L^1(|t|<L)} \|g\|_{W^{2,d}_x} \|F_{<\la}\|_{L^2_tW^{2,2^*}_x}
 \pr\lec L\la^2 \|g\|_{W^{2,d}_x}\|F\|_{L^2_tL^{2^*}_x}, }
so that we can make its contribution as small as we like by $D\to\I$.
\end{proof}

\section{Uniform Strichartz estimates}\label{sec:uniform-str}
In this section, we prove a uniform Strichartz estimate for wave potential in $L^2(\R^4)$ below the ground state threshold $\|W^2\|_{L^2(\R^4)}$. $W$ is the ground state solution to the nonlinear Schr\"odinger equation on $\R^4$:
\EQ{
 W(x):=(1+|x|^2/8)^{-1}, \pq -\De W= W^3,}
or the Aubin-Talenti function, the unique maximizer of the Sobolev inequality $\|\fy\|_{L^4(\R^4)}\le C\|\nabla\fy\|_{L^2(\R^4)}$. It gives rise to the family of static solutions $(u,v)=(W_\la,-W_\la^2)$, where $W_\la(x):=\la W(\la x)$ is the invariant scaling in $\dot H^1(\R^4)$ and in $L^4(\R^4)$.

\begin{theorem} \label{thm:unif Stz}
Let $d=4$, $0\les s<1$, and $0<B<\|W^2\|_{L^2(\R^4)}$. There is a constant $C\in(0,\I)$, such that for any $g\in L^2(\R^4)$ with $\|g\|_{L^2}\le B$, we have the Strichartz estimate with the potential $V:=e^{it|\nabla|}g$
\EQ{ \label{unif Stz}
 \|u\|_{\underline{S}^s} \le C\BR{\inf_{t\in\R}\|u(t)\|_{H^s} + \|(i\p_t+\De-\Re V)u\|_{N^s}}. }
\end{theorem}

Roughly, the proof of Theorem \ref{thm:unif Stz} proceeds as follows. Let  $C(g)$ be the optimal constant in \eqref{unif Stz} with potential $V = e^{it|\nabla|} g$ and define the quantities
    $$ M(B) := \sup\{ \,C(g) \mid \| g\|_{L^2} \les B \}, \qquad B^* : = \sup \{ B>0 \mid M(B)<\infty \}.$$
Our goal is to show that $B^* = \| W^2 \|_{L^2(\RR^4)}$. It is easy enough to show that $B^* > 0$, this is simply a restatement of the small data theory. Suppose for the sake of a contradiction that $0< B^* < \| W^2 \|_{L^2(\RR^4)}$. Then there exists a sequence of potentials $V_n$ such that $ \| V_n \|_{L^\infty_t L^2_x} \ne  B^* < \| W^2 \|_{L^2(\RR^4)}$ such that the corresponding constant in \eqref{unif Stz} satisfies $C_n = C(V_n) \to \infty$. We now run a profile decomposition as in \cite{Bahouri1999} on the free waves $V_n$ in $L^\infty_t L^2_x$ with error going to zero in $\dot{B}^{-1}_{4, \infty}$. The estimates from the previous sections together with the orthogonality of the profiles reduces the problem to considering a single profile. It is in this reduction where the improved bilinear estimates in Theorem \ref{thm:main schro est nonres} play a crucial role. Finally, to deal with the single profile case, we can essentially proceed as in \cite{Guo2018} and show that it suffices to prove a Strichartz estimate for a static potential below the ground state. But this is a consequence of the general double endpoint Strichartz estimates contained in \cite{Mizutani2016}. \\

In the following, we make the proof sketched above precise. In Subsection \ref{ss:Duh exp}, we iterate the Duhamel formula to obtain a number of key identities that are exploited later. The arguments here are essentially algebraic in nature, but they have useful analytic consequences. In Subsection \ref{ss:prof decomp} we recall the profile decomposition of Bahouri-G\'{e}rard \cite{Bahouri1999}. In Subsection \ref{ss:proof of unif stri} we reduce the proof of Theorem \ref{thm:unif Stz} to proving three key properties: (i) orthogonal profiles only interact weakly, (ii) small $L^\infty_t \dot{B}^{-1}_{4, \infty}$ profiles can be discarded, (iii) the single profile case holds. Finally, in the remainder of this section, we give the proof of the properties (i), (ii), and (iii).

\subsection{Duhamel formula expansion} \label{ss:Duh exp}
The proof of the uniform Strichartz estimate is based on the profile decomposition applied to a sequence of wave potentials. Here we expand the Duhamel formula with respect to the potential, which allows us to reduce the uniform estimate to the case of single profile or remainder in the decomposition. The argument for expansion is simple and algebraic.

Fix $d\ge 3$ and let $t_0\in\R$ and $t_0 \in I \subset \RR$ be an interval. Let
    \begin{equation}\label{eqn:definition X0 Y0}
        Y_0:=L^2(I;L^{2_*}(\R^d)), \pq X_0:=C(I;L^2(\R^d))\cap L^2(I;L^{2^*}(\R^d)).
    \end{equation}
For any space-time functions $V\in L^\infty(I;L^{d/2}(\R^d))$ and $f\in L^2(I;L^{2_*}(\R^d))$, we define $\mc{I}_Vf$ be the unique solution $u\in C(I;L^2(\R^d))\cap L^2(I;L^{2^*}(\R^d))$ to
\EQ{
 u(t_0)=0, \pq (i\p_t+\De-V)u=f \pq(t\in I)}
and define $\mc{U}_V(t;s)$ to be the corresponding homogeneous solution operator with data at $t=s$. Under reasonable assumptions on the potential $V$, the operator $\mc{I}_V : Y_0 \to X_0$ is a bounded linear operator.

 \begin{lemma}[$\mc{I}_V$ well-defined]\label{lem: IV well-defined}
 Let $d\g 3$. There exists an $\epsilon>0$ such that for any $t_0 \in I\subset \RR$ and $V \in L^\infty_t L^{d/2}_x(I\times \RR^d)$, if  we can write $I = \cup_{j=1}^N I_j$ with $N<\infty$ and
        $$ \sup_{j=1, \dots N} \| V \|_{(L^\infty_t L^{d/2}_x + L^2_t L^{d}_x)(I_j \times \RR^d)} < \epsilon, $$
 then the linear operators $\mc{I}_V \in B(Y_0(I) \to X_0(I))$ and $\mc{U}_V(\cdot;t_0) \in B( L^2_x \to X_0(I))$ exist and are well-defined.
 \end{lemma}
 \begin{proof}
 The first step is to use a perturbative argument via the endpoint Strichartz estimate to construct a solution on each interval $I_j$. Since there are only $N$ intervals, we obtain a solution on the whole interval $I$ and moreover the required bound holds.
 \end{proof}

 It is easy enough to check that the hypothesis in Lemma \ref{lem: IV well-defined} is always satisfied if $\| V \|_{L^\infty_t L^{d/2}_x(\RR^{1+d})}<\epsilon$. On the other hand, it also suffices to simply assume that $I \subset \RR$ is compact and $V \in C_t(I;L^{d/2}_x(\RR^d))$, since we can simply approximate the potential $V$ with a smooth, compactly supported (in $x$) potential.  Alternative assumptions on $V \in L^\infty_t L^{d/2}_x$ to ensure the existence of $\mc{I}_V \in B(Y_0 \to X_0)$ are possible, but the previous lemma suffices for our purposes.

 We now turn to the algebraic component of the argument. The key point is to compare the operators $\mc{I}_V$ and $\mc{I}_{V'}$. This is extremely useful both when $V' =0$, and $V'\not = 0$. Identifying the potential $V$ with the multiplication operator,
\EQ{
 V: X_0 \to Y_0}
we obtain a bounded linear operator via H\"older's inequality, in other words we obtain an operator $V \in B(X_0 \to Y_0)$. Since $u=\mc{I}_Vf$ solves for any other $V'\in C(I;L^{d/2}(\R^d))$,
\EQ{
 u(t_0)=0, \pq (i\p_t+\De-V')u=f+(V-V')u,}
we obtain the relation between $\mc{I}_V$ and $\mc{I}_{V'}$:
\EQ{\label{exp IV by IV' 1}
 \mc{I}_V = \mc{I}_{V'}\BR{1+(V-V')\mc{I}_V},}
which can be written also as
\EQ{ \label{exp IV by IV'}
 \mc{I}_{V'} = \BR{1+\mc{I}_{V'}(V'-V)}\mc{I}_V.}
In particular, choosing $V'=0$ or $V=0$ yields
\EQ{ \label{exp IV by I0}
 \pt \mc{I}_V = \mc{I}_0\BR{1+V\mc{I}_V} = \BR{1+\mc{I}_VV}\mc{I}_0,
 \pq \mc{I}_0 = \mc{I}_V\BR{1-V\mc{I}_0} = \BR{1-\mc{I}_0V}\mc{I}_V.}
There are two immediate consequences of these identities that we wish to highlight. \\

\begin{enumerate}
\item[ (I) (Small perturbations). ] If $V$ is a small perturbation of a potential $V'$, and $I_{V'}$ is a bounded linear operator, then so is $\mc{I}_V$. More precisely, by expanding either $[1+(V'-V)\mc{I}_{V'}]^{-1}$ or $[1 + \mc{I}_{V'}(V'-V)]^{-1}$ in a Neumann series and applying the identities \eqref{exp IV by IV' 1} or \eqref{exp IV by IV'} yields for any Banach spaces $X\subset X_0$ and $Y\subset Y_0$ the implication
\EQ{ \label{small pert IV}
 \pt\de:=\min(\|\mc{I}_{V'}(V-V')\|_{B(X\to X)},\|(V-V')\mc{I}_{V'}\|_{B(Y\to Y)})<1
 \pr\pq\pq\implies \|\mc{I}_V\|_{B(Y\to X)} \le (1-\de)^{-1}\|\mc{I}_{V'}\|_{B(Y\to X)},}
which is useful both with $V'=0$ and with $V'\not=0$.\\

\item[ (II) (Automatic upgrading).]  Let $X_2\subset X_1\subset X_0$
  and $Y_1\subset Y_0\subset Y_2$ be Banach spaces with continuous
  embeddings, and $Y_0\subset Y_2$ dense. Then under suitable assumptions, we can automatically upgrade the operator $\mc{I}_V: Y_1 \to X_1$ to a map $Y_2 \to X_2$. More precisely, if
        $$\mc{I}_V:Y_1\to X_1, \qquad \mc{I}_0:Y_2\to X_2, \qquad V:X_2\to Y_1, \qquad V:X_1\to Y_2$$
    are bounded, then from \eqref{exp IV by I0} we obtain the
    extension $\mc{I}_V:Y_2\to X_2$ with the stronger estimate
\EQ{ \label{auto up}
 \|\mc{I}_V\|_{B(Y_2\to X_2)} \pt\les \|\mc{I}_0\|_{B(Y_2\to X_2)}\BR{1+\|V\|_{B(X_1\to Y_2)}\|\mc{I}_V\|_{B(Y_2\to X_1)}}
 \pr\les  \|\mc{I}_0\|_{B(Y_2\to X_2)}\BR{1+\|V\|_{B(X_1\to Y_2)}\|\mc{I}_0\|_{B(Y_2\to X_2)}(1+\|V\|_{B(X_2\to Y_1)}\|\mc{I}_V\|_{B(Y_1\to X_1)}) }.
}\\
\end{enumerate}

We now consider potentials of the form $V=V_1 + V_2$. By iterating the above, we arrive at the following key lemma.

\begin{lemma} \label{lem:Duhamel decompose}
Let $d\ge 3$, $t_0\in\R$, $t_0\in I$ be an interval, and $V=V_1+V_2$ with $V_j\in L^\infty(I;L^{d/2}(\R^d))$. Suppose that for some Banach spaces $\underline{X} \subset X\subset X_0$ and $Y\subset Y_0$, we have finite numbers $M,\e$ such that
\EQ{
 \|\mc{I}_0\|_{B(Y \to \underline{X})} + \|\mc{I}_{V_j}\|_{B(Y \to \underline{X})} + \|V_j\|_{B(X \to Y)} \le M,
\pq \|\mc{I}_0V_2\mc{I}_0V_1\|_{B(\underline{X}\to X)}\le\e,}
for $j=1, 2$, where $X_0,Y_0$ and $\mc{I}_V$ are as defined above. Let $C:=\|I\|_{B(\underline{X}\to X)}$ and $\ti M:=(1+CM^2)^6$.
If $\ti M\e<1$, then we have
\EQ{
 \|\mc{I}_V\|_{B(Y \to \underline{X})} \le (1-(\ti M\e)^\frac{1}{2})^{-1}M \ti M.}
\end{lemma}
\begin{proof}
By iterating the Duhamel expansions \eqref{exp IV by IV' 1} and \eqref{exp IV by IV'}, we obtain the identities
\begin{align*}
 \mc{I}_V &= \mc{I}_{V_1}(1+V_2\mc{I}_V) = \mc{I}_{V_2}(1+V_1\mc{I}_V) \\
 &= \mc{I}_{V_1}\BR{1+V_2 \mc{I}_{V_2}(1+V_1\mc{I}_V)} = \mc{I}_{V_1} + \mc{I}_{V_1}V_2\mc{I}_{V_2} + \mc{I}_{V_1}V_2 \mc{I}_{V_2}V_1 \mc{I}_V,
\end{align*}
which we rewrite as
$$
 \BR{1-\mc{I}_{V_1}V_2 \mc{I}_{V_2}V_1}\mc{I}_V = \mc{I}_{V_1} + \mc{I}_{V_1}V_2\mc{I}_{V_2}.
$$
In order to invert the bracket on the left, we further expand $\mc{I}_{V_2}$ twice which gives
$$
  \mc{I}_{V_2}  = (1+\mc{I}_{V_2}V_2)\mc{I}_0 =  \mc{I}_0  + ( 1 + \mc{I}_{V_2} V_2) \mc{I}_0 V_2 \mc{I}_0
$$
and hence
  \begin{align*}
    \mc{I}_{V_1} V_2 \mc{I}_{V_2} V_1 &= \mc{I}_{V_1} V_2 \mc{I}_0 V_1 + \mc{I}_{V_1} V_2 ( 1 + \mc{I}_{V_2} V_2) \mc{I}_0 V_2 \mc{I}_0 V_1 \\
        &= (1 + \mc{I}_{V_1} V_1) \mc{I}_0 V_2 \mc{I}_0 V_1 + \mc{I}_{V_1} V_2 ( 1 + \mc{I}_{V_2} V_2) \mc{I}_0 V_2 \mc{I}_0 V_1 = A \mc{I}_0 V_2 \mc{I}_0 V_1
  \end{align*}
with
    $$ A:=  1 + \mc{I}_{V_1} V_1 + \mc{I}_{V_1} V_2 ( 1 + \mc{I}_{V_2} V_2). $$
Therefore we conclude that
    \begin{equation}\label{eqn:IV ident with V1 V2}
 \BR{1- A \mc{I}_{0} V_2 \mc{I}_{0} V_1}\mc{I}_V = \mc{I}_{V_1} + \mc{I}_{V_1}V_2\mc{I}_{V_2}.
    \end{equation}
Hence if both $\mc{I}_{V_j}$ are bounded and $\mc{I}_0V_2\mc{I}_0V_1$ is small enough, then $\mc{I}_V$ is a small perturbation of $\mc{I}_{V_1} ( 1 + V_2 \mc{I}_{V_2})$ and so bounded as well. More precisely, an application of the assumed operator bounds together with composition of mappings, gives the bounds
    $$
        \| \mc{I}_{V_1} + \mc{I}_{V_1}V_2\mc{I}_{V_2} \|_{B(Y \to \underline{X})} \les \| \mc{I}_{V_1}  \|_{B(Y \to \underline{X})}  + \| \mc{I}_{V_1}  \|_{B(Y \to \underline{X})} \|V_2\|_{B(X \to Y)} \| \mc{I}_{V_2} \|_{B(Y \to X)} \les M ( 1 + CM^2)
    $$
and
    \begin{align*}
         \| A \|_{B(X\to X)} &\les 1 + \| \mc{I}_{V_1} \|_{B(Y \to X)} \| V_1 \|_{B(X\to Y)} + \| \mc{I}_{V_1} \|_{B(Y \to X)} \| V_2 \|_{B(X \to Y)} \big( 1 + \| \mc{I}_{V_2}\|_{B(Y \to X)} \| V_2 \|_{B(X\to Y)}\big) \\
                        &\les 1 + CM^2 + CM^2 ( 1 + CM^2) = ( 1 + CM^2)^2
    \end{align*}
and therefore
        $$ \| A \mc{I}_0 V_2 \mc{I}_0 V_1 \|_{B(\underline{X} \to X)} \les \| A \|_{B(X \to X)} \| \mc{I}_0 V_2 \mc{I}_0 V_1 \|_{B(\underline{X} \to X)} \les (1 + CM^2)^2 \epsilon. $$
On the other hand, a short computation also gives
        $$ \| A \mc{I}_0 V_2 \mc{I}_0 V_1 \|_{B(X \to \underline{X})}  \les  ( 1 + CM^2)^4$$
and consequently for every $n\ge 0 $
        $$ \| ( A \mc{I}_0 V_2 \mc{I}_0 V_1)^n \|_{B(\underline{X} \to X)} \les C\big( ( 1 + CM^2)^4 ( 1 + CM^2)^2 \epsilon \big)^{\frac{n}{2}}  = C\big( ( 1 + CM^2)^6 \epsilon \big)^{\frac{n}{2}}. $$
Therefore, since $( 1 + CM^2)^6 \epsilon < 1$ by assumption, the series $\sum_{n=0}^\infty ( A \mc{I}_0 V_2 \mc{I}_0 V_1)^n$ converges absolutely in $B(\underline{X}\to X)$. Consequently, by the Neumann series, we can invert the bracket on the left hand side of \eqref{eqn:IV ident with V1 V2} and define
        $$ \mc{I}_V = [ 1- A \mc{I}_0 V_2 \mc{I}_0 V_1]^{-1} \big( \mc{I}_{V_1}  + \mc{I}_{V_1}V_2\mc{I}_{V_2}\big): Y \to X$$
with the operator bound
        \begin{align*}
            \| \mc{I}_V \|_{B( Y \to X)} &\les \sum_{n=0}^\infty \| (A \mc{I}_0 V_2 \mc{I}_0 V_1)^n \|_{B(\underline{X} \to X)} \| \mc{I}_{V_1} + \mc{I}_{V_1}V_2\mc{I}_{V_2}\|_{B(Y \to \underline{X})} \\
                                         &\les \sum_{n=0}^\infty  C\big( ( 1 + CM^2)^6 \epsilon \big)^{\frac{n}{2}} M ( 1 + CM^2)  =  \big( 1 - (\ti{M} \epsilon)^\frac{1}{2} \big)^{-1} CM ( 1 + CM^2).
        \end{align*}
The only remaining step is to upgrade $\mc{I}_V$ to a map into the smaller space $\underline{X} \subset X$. This follows by repeating the argument in (II) above. Namely, the identity $ \mc{I}_V = \mc{I}_{V_1} (1 + V_2 \mc{I}_V)$
implies that
    \begin{align*}
        \| \mc{I}_V \|_{B(Y \to \underline{X})} \les \| \mc{I}_{V_1} \|_{B(Y\to \underline{X})} \big( 1 + \| V_2 \|_{B(X \to Y)} \| \mc{I}_{V} \|_{B(Y \to X)} \big) &\les M  + CM^2 \big( 1 - (\ti{M} \epsilon)^\frac{1}{2} \big)^{-1} M ( 1 + CM^2) \\
        &\les \big( 1 - (\ti{M} \epsilon)^\frac{1}{2} \big)^{-1} M \ti{M}
    \end{align*}
\end{proof}

\subsection{The profile decomposition}\label{ss:prof decomp}
In the following we state the profile decomposition due to Bahouri-G\'erard \cite{Bahouri1999}. The version we give below is slightly adjusted to our setting, but the proof is the same. Let $V_n$ be a sequence of free waves with bounded $L^2(\R^4)$ norm, and let
\EQ{ \label{prof decop}
 V_n = \sum_{j=1}^J V_n^j + \Ga_n^J}
be its profile decomposition in $L^2(\R^4)$, such that (along some subsequence)
\EQ{ \label{dec Gamma}
 \lim_{J\to\I} \limsup_{n\to\I} \|\Ga^J_n\|_{L^\I_t \dot B^{-1}_{4, \infty}}=0}
and
\begin{equation}\label{eqn:defn of profiles}
 V_n^j(t,x) = \fg_n^j V^j :=(\sigma_n^j)^2 V^j(\sigma_n^j t - t_n^j, \sigma_n^j x - x_n^j),
\end{equation}
where the profiles $V^j = e^{it |\nabla|} \varphi^j$ are free waves which are independent of $n$, and satisfy
    \begin{equation}\label{decoupling in L2}
        \lim_{n\to \infty} \Big| \| V_n \|_{L^\infty_t L^2_x} -  \Big( \| \Gamma^J_n \|_{L^\infty_t L^2_x}^2 + \sum_{j=1}^J \| V^j \|_{L^\infty_t L^2_x}^2 \Big)^\frac{1}{2}\Big| = 0.
    \end{equation}
The group elements $\mathfrak{g}^j_n = \mathfrak{g}[\sigma^j_n, x^j_n, t^j_n]$ are asymptotically orthogonal, in the sense that for each pair $j\not=k$, one of the following holds as $n\to\I$:
\begin{enumerate}
\item $|\log(\sigma^j_n/\sigma^k_n)|\to\I$ (scale separation).
\item $\sigma^j_n\equiv \sigma^k_n$ and $|t^j_n-t^k_n|\to\I$ (time separation).
\item $\sigma^j_n\equiv \sigma^k_n$, $t^j_n\equiv t^k_n$ and $|x^j_n-x^k_n|\to \I$ (space separation).
\end{enumerate}
We may normalize the parameters such that
\begin{enumerate}
\item $\sigma^j_n\equiv 1$, $\sigma^j_n\to+\I$ or $\sigma^j_n\to+0$.
\item $t^j_n\equiv 0$ or $t^j_n\to\pm\I$.
\item $x^j_n\equiv 0\in\R^d$ or $|x^j_n|\to\I$.
\end{enumerate}

\subsection{Reductions and the proof of Theorem \ref{thm:unif Stz}}\label{ss:proof of unif stri}

The first step in the proof \eqref{unif Stz} is to observe that the claimed Strichartz bound always holds, but with a constant $C$ that depends on the potential $V$. This is a consequence of the more general result \cite[Theorem 7.1]{Candy2023}. However, for completeness, we include the special case $d=4$ in the low regularity regime with a short proof here.

\begin{theorem} \label{thm:non-unif Stz}
Let $d=4$ and $0\les s < 1$. If $V = e^{it|\nabla|} g \in L^\infty_t L^2_x$, then there exists $C(V)>0$ such that for any $f \in H^s$, $F \in N^s$, there exists a unique solution $u \in C(\RR, H^s)\cap L^2_t L^{4}_x$ to
    $$ (i\p_t + \Delta - \Re(V) ) u = F, \qquad u(0) = f$$
and moreover, we have the bound
    $$ \| u \|_{\underline{S}^s} \les C(V) \big( \| f \|_{H^s} + \| F \|_{N^s}\big). $$
\end{theorem}
\begin{proof}
 An application of Theorem \ref{thm:main schro est nonres} gives for any interval $I\subset \RR$
    $$ \| \Re(V) u \|_{N^s(I)} \lesa \| V \|_{L^\infty_t L^2_x} \| u \|_{S^s(I)}. $$
On the other hand, the standard product inequality for Besov spaces gives
    \begin{align} \| \Re(V) u \|_{N^s(I)} &\lesa \| \Re(V) u \|_{L^2_t B^s_{4/3, 2}(I\times \RR^4)} \notag \\
                                &\lesa \| V \|_{L^2_t B^s_{4, 2}(I\times \RR^n)} \| u \|_{L^\infty_t H^s(I\times \RR^4)} \les \| V \|_{L^2_t B^s_{4, 2}(I\times \RR^4)} \| u \|_{S^s(I)}. \label{eqn:thm non-unif Stz:potential in L2time}
    \end{align}
In particular, if we have an interval $I\subset \RR$ such that  $V=V_1 + V_2$ with
        $$ \| V_1 \|_{L^\infty_t L^2_x} + \| V_2 \|_{L^2_t B^s_{4, 2}(I\times \RR^4)} < \epsilon, $$
then provided $\epsilon>0$ is sufficiently small, the iterates
        $$ (i\p_t + \Delta) u_j = \Re(V) u_{j-1} + F, \qquad u_j(t_0) = f$$
converge to a solution $u \in \underline{S}^s(I)$. To extend this result to $\RR$, we note that by \cite[Lemma 4.1]{Bejenaru2015} there exists a finite partition $\RR = \cup_{j=1}^N I_j$ into intervals, and a decomposition $V = V_0 + \sum_{j=1}^N \ind_{I_j}(t) V_1$ such that $V_0$ and $V_1$ are free waves, and we have the bounds
        $$ \| V_0 \|_{L^\infty_t L^2_x} + \sup_{j=1, \dots, N} \| V_1 \|_{L^2_t B^s_{4, 2}(I_j \times \RR^4)} \les \epsilon. $$
Hence we can iteratively construct the solution on each interval $I_j$. Since there are only finitely many intervals, we obtain a global bound as claimed.
\end{proof}

We now state the basic mapping properties of the Duhamel operators $\mc{I}_V$ that have been obtained above.

\begin{corollary}\label{cor:bounded maps}
Let $0\les s < 1$ and $d=4$.
\begin{enumerate}
    \item  There exists $C>0$ and $0< \theta < 1$ such that for any (real-valued) free wave $v= \Re(e^{it|\nabla|} g)  \in C(\RR ; L^2)$ we have
   $$ \| \mc{I}_0 \|_{B(N^s \to \underline{S}^s)} \les C, \qquad \| v \|_{B(S^s \to N^s)} \les C \| v \|_{L^\infty_t L^2_x}, \qquad \| \mc{I}_0 v \|_{B(\underline{S}^s \to S^s)} \les C \| v \|_{L^\infty_t L^2_x}^\theta \| v \|_{L^\infty_t B^{-1}_{4, \infty}}^{1-\theta}. $$
 \item For any interval $I\subset \RR$ and any (real-valued) free wave $v= \Re(e^{it|\nabla|} g) \in C(I;L^2)$, the Duhamel operator $\mc{I}_{v}:N^s(I)\to \underline{S}^s(I)$ is bounded.
\end{enumerate}
\end{corollary}
\begin{proof}
The first claim follows from the energy inequality \eqref{eqn:lem energy ineq:energy} and the bilinear estimates contained in Theorem \ref{thm:main schro est nonres} and Theorem \ref{thm:besov gain} while the second is a direct consequence of Theorem \ref{thm:non-unif Stz}.
\end{proof}

To simplify the argument to follow, we also give a restatement of Lemma \ref{lem:Duhamel decompose} in the special case $X = S^s$, $\underline{X} = \underline{S}^s$, and $Y = N^s$.

\begin{lemma} \label{lem:Duhamel decompose special}
Let $d=4$, $t_0\in\R$, $t_0\in I$ be an interval, and $V=V_1+V_2$ with $V_j\in L^\infty(I;L^2(\R^4))$. Suppose that we have finite numbers $M,\e$ such that
\EQ{
 \| \mc{I}_0 \|_{B(N^s \to \underline{S}^s)} + \|\mc{I}_{V_j}\|_{B(N^s \to \underline{S}^s)} + \|V_j\|_{B(S^s \to N^s)} \le M,
\pq \|\mc{I}_0V_1\mc{I}_0V_2\|_{B(\underline{S}^s\to S^s)}\le\e,}
for $j=1, 2$. Let $C:=\|I\|_{B(\underline{S}^s\to S^s)}$ and $\ti M:=(1+CM^2)^2$. If $\ti M\e<1$, then we have
\EQ{
 \|\mc{I}_V \|_{B(N^s \to \underline{S}^s)} \le (1-\ti M\e)^{-1}M \ti M.}
\end{lemma}

In the remainder of this subsection, we give the proof of Theorem \ref{thm:unif Stz} assuming the following key properties:
\begin{enumerate}[\bfseries ({A}1)]
\item  If $V_n^1$ and $V_n^2$ are two asymptotically orthogonal profiles as in \eqref{eqn:defn of profiles},  we have
\EQ{ \label{dich in Duh}
 \|\mc{I}_0v_n^1\mc{I}_0 v_n^2\|_{B(\underline{S}^s \to S^s)} \to 0}
 where $v_n^j = \Re( V_n^j)$.

\item  If a sequence of free waves $\Ga_n$ is bounded in $L^2$ and vanishing in $L^\I_t B^{-1}_{4, \infty}$, then letting $w_n = \Re(\Gamma_n)$ we have
\EQ{ \label{Bes dec}
 \|\mc{I}_0w_n\|_{B(\underline{S}^s \to S^s)}\to 0.}

\item  If $B^*<\|W^2\|_2$, then the desired estimate holds in the case of single profile. More precisely, if $V^1_n$ is a profile \eqref{eqn:defn of profiles}, and $v^1_n = \Re(V^1_n)$, then
\EQ{ \label{single prof in Hs}
 \|V^1_n\|_{L^\I_tL^2}=\|V^1\|_{L^\I_tL^2}<\|W^2\|_2 \implies \sup_{n}\|\mc{I}_{v_n^1}\|_{B(N^s\to \underline{S}^s)}<\I.}
\end{enumerate}

\begin{proof}[Proof of Theorem \ref{thm:unif Stz}] Let $v:=\Re(V)$ and let $\mc{U}_v(t;t_0) f$ denote the solution to the homogeneous equation
            $$ (i\p_t + \Delta) u  =  v u, \qquad u(t_0) = f. $$
In view of the identity
            $$ \mc{U}_{v} (t;t_0) f = \mc{I}_{v}\big( v e^{it\Delta} f \big) - e^{it\Delta} f $$
an application of Theorem \ref{thm:main schro est nonres} implies that for any free wave $V = e^{it|\nabla|} g$ we have
        $$ \| \mc{U}_{v} (t;t_0) f\|_{\underline{S}^s} \les  \| \mc{I}_{v} \|_{B(N^s \to \underline{S}^s)} \| V e^{it\Delta} f \|_{N^s}  + \| e^{it\Delta} f \|_{\underline{S}^s} \lesa \big( 1 + \| \mc{I}_{v} \|_{B(N^s \to \underline{S}^s)} \| V \|_{L^\infty_t L^2_x}\big) \| f \|_{H^s}$$
where the implied constant is independent of $V$. Consequently, to prove Theorem \ref{thm:unif Stz} it suffices to prove that we have a uniform bound for the Duhamel operators $\mc{I}_{v}$. An application of Corollary \ref{cor:bounded maps} implies that $\| \mc{I}_{v} \|_{B(N^s \to \underline{S}^s)}<\infty$ but with a bound potentially depending on $V$. Suppose for contradiction that the uniform estimate claimed in Theorem \ref{thm:unif Stz} fails at some $s\in[0,1)$ and $B\in(0,\|W^2\|_2)$. Then we have the threshold mass $B^*\in(0,\|W^2\|_{L^2})$ defined by
\EQ{
 \pt M(B):=\sup\{\|\mc{I}_v\|_{B(N^s\to \underline{S}^s)} \mid \text{$v = \Re(V)$, $V$ free wave with $\|V\|_{L^\I_tL^2_x}\le B$}\},
 \pr B^*:=\sup\{B>0 \mid M(B)<\I\},}
and a sequence $V_n$ of free waves satisfying
\begin{equation}\label{eqn:thm unif str:assump on V_n}
 \|V_n\|_{L^\I_tL^2_x} \nearrow B^*, \pq \|\mc{I}_{v_n}\|_{B(N^s\to \underline{S}^s)} \to \I
\end{equation}
with $v_n = \Re(V_n)$. Note that $B^*>0$ is ensured by the small perturbation of the free case using \eqref{small pert IV} and the bounds (i) in Corollary \ref{cor:bounded maps}.

Let \eqref{prof decop} be the profile decomposition of $V_n$, after passing to a subsequence if necessary, and define $v_n = \Re(V_n)$, $v^j_n = \Re(V^j_n)$, and $w_n^J = \Re(\Gamma^J_n)$. We begin by considering the case where all the $L^\infty_t L^2_x$ mass concentrates in one profile, that is when $V_n = V^1_n + \Gamma^1_n$ with
        $$ \lim_{n \to \infty} \| V_n \|_{L^\infty_t L^2_x} = \| V^1\|_{L^\infty_t L^2_x}, \qquad \Longleftrightarrow  \qquad  \lim_{n\to\I}\|\Ga^1_n\|_{L^\I_t L^2_x}=0.$$
As we have the uniform bound \eqref{single prof in Hs} for a single profile, we have $ \sup_{n} \| \mc{I}_{v^1_n} \|_{B(N^s \to \underline{S}^s)} \les M <  \infty$. Noting the identity $\mc{I}_{v_n} = \mc{I}_{v^1_n} + \mc{I}_{v^1_n} w^1_n \mc{I}_{v_n}$ we obtain
    $$ \| \mc{I}_{v_n} \|_{B(N^s \to \underline{S}^s)} \les \| \mc{I}_{v_n^1} \|_{B(N^s \to \underline{S}^s)} ( 1 + \| w^1_n \|_{B(\underline{S}^s \to N^s)} \| \mc{I}_{v_n} \|_{B(N^s \to \underline{S}^s)})$$
and hence Corollary \ref{cor:bounded maps} together with the fact that the error $w^1_n$ vanishes in $L^2_x$, implies that for all sufficiently large $n$ we have
    $$ \| \mc{I}_{v_n} \|_{B(N^s \to \underline{S}^s)} \les M + \frac{1}{2}  \| \mc{I}_{v_n} \|_{B(N^s \to \underline{S}^s)} \qquad \limsup_{n \to \infty} \| \mc{I}_{v_n} \|_{B(N^s \to \underline{S}^s)} \les 2 M.$$
In other words we have a contradiction to the choice of the sequence $V_n$. Thus the single profile case holds.

We now dispose of the case with no profile, in other words when $V_n = \Gamma^1_n$, hence $\| V_n \|_{L^\infty_t L^2_x} \nearrow B^*$ and $ \| V_n \|_{L^\infty_t \dot{B}^{-1}_{4, \infty}}  \to 0$. We argue as in the automatic upgrading property (II). Two applications of the identity \eqref{exp IV by I0} imply that
            $$ \mc{I}_{v_n} = \mc{I}_0 (  1 + v_n \mc{I}_{v_n}) = \mc{I}_0 + \mc{I}_0 v_n \mc{I}_0 + \mc{I}_0 v_n \mc{I}_0 v_n \mc{I}_{v_n}$$
and hence an application of the bounds in Corollary \ref{cor:bounded maps} implies that
        \begin{align*}
            \| \mc{I}_{v_n} \|_{B(N^s \to \underline{S}^s)} &\les \| \mc{I}_{0} \|_{B(N^s \to \underline{S}^s)} + \| \mc{I}_{0} \|_{B(N^s \to \underline{S}^s)}^2 \| v_n \|_{B(S^s \to N^s)}\\
            &\qquad + \| \mc{I}_{0} \|_{B(N^s \to \underline{S}^s)} \| v_n \|_{B(S^s \to N^s)}  \| \mc{I}_{0} v_n \|_{B(\underline{S}^s \to S^s)} \| \mc{I}_{v_n} \|_{B(N^s \to \underline{S}^s)} \\
            &\lesa 1 +  \| \mc{I}_{0} v_n \|_{B(\underline{S}^s \to S^s)} \| \mc{I}_{v_n} \|_{B(N^s \to \underline{S}^s)}
        \end{align*}
with the implied constant independent of $n$. Therefore, we conclude via  \textbf{(A2)} that $\limsup_n \| \mc{I}_{v_n} \|_{B(N^s \to \underline{S}^s)} < \infty$ which again contradicts the definition of the sequence $V_n$.

The remaining case is when the initial profile satisfies $0< \| V^1 \|_{L^\infty_t L^2_x} < B^*$. In this case, the asymptotic orthogonality of the profiles in $L^2$ implies that we have a uniform bound below the threshold $B^*$, namely we have some $B< B^*$ such that for any $J\g 1$
        \begin{equation}\label{eqn:profil belo threshold}
             \|V^1_n\|_{L^\I_tL^2_x}=\|V^1_1\|_{L^\I_tL^2_x} \les B,
 \qquad \limsup_{n\to\I}\|V^{2,J}_n\|_{L^\I_tL^2_x}\les B,  \qquad \limsup_{n\to \infty} \| \Gamma^J_n \|_{L^\infty_t L^2_x} \les B
        \end{equation}
where we introduce the notation
    $$V^{J_1,J_2}_n:=\sum_{j=J_1}^{J_2} V^j_n $$
with the convention that $V^{J_1, J_2}_n = 0$ if $J_1> J_2$. The definition of the threshold $B^*$ together with \eqref{eqn:profil belo threshold} implies that there exists $M>0$ such that for any $J\g1$ we have the uniform bounds
     \begin{equation}\label{eqn:unif bdd profiles}
        \limsup_{n \to \infty}  \| \mc{I}_{v^1_n} \|_{B(N^s \to \underline{S}^s)} + \limsup_{n \to \infty} \| \mc{I}_{v^{2, J}_n} \|_{B(N^s \to \underline{S}^s)} + \limsup_{n \to \infty} \| \mc{I}_{w^J_n} \|_{B(N^s \to \underline{S}^s)} \les M < \infty.
     \end{equation}
An application of Lemma \ref{lem:Duhamel decompose special} together with \textbf{(A1)} implies that for any $J\g 1$ we have the bound
            $$\limsup_{n \to \infty} \| \mc{I}_{v^{1, J}_n} \|_{B(N^s \to \underline{S}^s)} =  \limsup_{n \to \infty} \| \mc{I}_{v^1_n + v^{2, J}_n} \|_{B(N^s \to \underline{S}^s)} \les 2 M ( 1 + M^2)^4 < \infty.$$
Moreover, via Corollary \ref{cor:bounded maps}, we obtain
        $$ \| \mc{I}_0  v^{1, J}_n  \mc{I}_0 w^J_n \|_{B( \underline{S}^s \to S^s)} \les \| \mc{I}_0 \|_{B(N^s \to S^s)} \| v^{1, J}_n \|_{B( S^s \to N^s)} \| \mc{I}_0 w^J_n \|_{B(\underline{S}^s \to S^s)} \lesa \| \mc{I}_0 w^J_n \|_{B(\underline{S}^s \to S^s)}$$
where the implied constant is independent of $J$. In particular, choosing $J$ sufficiently large depending on $M$, the property \textbf{(A2)} together with the uniform bounds \eqref{eqn:unif bdd profiles} and another application of Lemma \ref{lem:Duhamel decompose special} implies that
      $$ \limsup_{n\to \infty} \| \mc{I}_{v^{1, J}_n + w^J_n} \|_{B(N^s \to \underline{S}^s)} < \infty. $$
Since $V_n = V^{1, J}_n + \Gamma^J_n$, this again contradicts the choice of the sequence $V_n$. Therefore Theorem \ref{thm:unif Stz} follows.
\end{proof}

Thus we are left with \eqref{dich in Duh}--\eqref{single prof in Hs} to be proven.

\subsection{Proof of (A2): Decay for the remainder}\label{subsec:a2}
Corollary \ref{cor:summed up besov gain} gives $0< \theta < 1$ such that
        $$ \| \mc{I}_0 w_n \|_{B(\underline{S}^s \to S^s)} \lesa \| \Gamma_n \|_{L^\infty_t L^2_x}^{1-\theta} \| \Gamma_n \|_{L^\infty_t B^{-1}_{4, \infty}}^\theta.$$
Hence letting $n\to \infty$, \eqref{Bes dec} follows. Note that the $X^{s-1,1}$ component of $S^s$ is treated easily by the product estimate
\EQ{ \label{neg prod}
 \|Vu\|_{H^{s-1}} \lec \|V\|_{B^{-\de}_{2(-\de),\I}} \|u\|_{B^s_{4,2}}, \pq 0<\de<\min(s,1-s), \pq \frac{1}{2(-\de)}=\frac12+\frac{-\de}{4},}
where the Besov norm of $V$ is controlled by interpolation between $L^2$ and $\dot H^{-1}_4$. On the other hand, the Strichartz component of $S^s$ is treated by Theorem \ref{thm:besov gain}, which is indeed the hardest part in the entire proof of the uniform Strichartz estimate.

\subsection{Proof of (A1): Decay for orthogonal profiles}\label{subsec:a1}
Here we prove \eqref{dich in Duh}.
Thanks to the uniform bounds on the operators, namely that Theorem \ref{thm:main schro est nonres} gives the uniform bound
        \begin{equation}\label{eqn:unif bound orthog}
            \| \mc{I}_0 v^j_n \|_{B(S^s \to \underline{S}^s)} \lesa \| V^j_n \|_{L^\infty_t L^2_x} = \| V \|_{L^\infty_t L^2_x},
        \end{equation}
we have
        \begin{equation}\label{eqn:unif bound iterated duhamel}\| \mc{I}_0 v^1_n \mc{I}_0 v^2_n \|_{B(\underline{S}^s \to S^s)} \lesa \| V^1_n \|_{L^\infty_t L^2_x} \|V^2_n \|_{L^\infty_t L^2_x} = \| V^1 \|_{L^\infty_t L^2_x} \|V^2 \|_{L^\infty_t L^2_x},  \end{equation}
 and hence we may always assume that $V^j = e^{it|\nabla|} \q^j$ with the spatial profile $\fy^j\in L^2(\R^4)$ belonging to any dense subset of nice functions, e.g., smooth and compactly supported in $x\in\R^4$ or in the Fourier space.

If the two profiles are separated by scaling, then we exploit the high-low and low-high gains. If they are separated in space-time, then we can reduce to the case $s=0$ by the complex interpolation, and use Lemma 8.6 for space separation and the dispersive decay for time separation.

\subsubsection{Low frequency decay}\label{subsubsec:low}
Let us first dispose of expanding profiles, namely the case of $\sigma^j_n\to+0$. Since
    \EQ{
 \|\sigma^2 V(\sigma t, \sigma x)\|_{L^2_tB^s_{4, 2}} \lesa  \sigma^{1/2} \lr{\sigma}^s \|V\|_{L^2_t B^s_{4, 2}},}
an application of \eqref{eqn:thm non-unif Stz:potential in L2time} gives
    $$ \| \mc{I}_0 v^j_n \|_{B(S^s \to \underline{S}^s)} \lesa  \sigma_n^{1/2} \lr{\sigma_n}^s \|V^j\|_{L^2_t B^s_{4, 2}}. $$
In particular, in view of the uniform bound \eqref{eqn:unif bound orthog}, we  conclude that
        \EQ{ \label{eqn:discard sigma to 0}
 \sigma_n^j\to+0 \implies \forall s\in [0,1):\; \ \|\mc{I}_0v_n^j\|_{B(S^s \to \underline{S}^s)}\to 0 \pq(n\to\I)}
and hence we may discard all the profiles $V_n^j$ with $\sigma_n^j\to 0$.
In other words, we may assume that either $\sigma^j_n\equiv 1$ or $\sigma^j_n\to\I$ for each $j$.

\subsubsection{Scale separation}\label{subsubsec:scale}
Next we consider the case of scale separation for temporary profiles, namely $\sigma^2_n/\sigma^1_n\to\I$ or $\sigma^1_n/\sigma^2_n\to\I$. Note that $\sigma^j_n\ge 1$ by the previous section. An application of Theorem \ref{thm:main schro est nonres} together with Corollary \ref{cor:summed up besov gain} implies that for every $0\les s < 1$ there exists $\epsilon>0$ such that
        $$ \| \mc{I}_0 v^1_n \mc{I}_0 v^2_n \|_{B(\underline{S}^s \to S^s)} \lesa \min\Big\{ \| V^1_n \|_{L^\infty_t H^\epsilon_x} \| V^2_n \|_{L^\infty_t H^{-\epsilon}_x}, \| V^1_n \|_{L^\infty_t  H^{-\epsilon}_x} \| V^2_n \|_{L^\infty_t H^{\epsilon}_x}\Big\}. $$
Since
   $$ \min\Big\{ \| V^1_n \|_{L^\infty_t H^\epsilon_x} \| V^2_n \|_{L^\infty_t H^{-\epsilon}_x}, \| V^1_n \|_{L^\infty_t  H^{-\epsilon}_x} \| V^2_n \|_{L^\infty_t H^{\epsilon}_x}\Big\} \lesa \Big(\min\Big\{ \frac{\sigma^1_n}{\sigma_n^2},  \frac{\sigma^2_n}{\sigma_n^1} \Big\}\Big)^\epsilon $$
we conclude that
    $$ |\log( \sigma_n^1/\sigma_n^2)| \to \infty \qquad \Longrightarrow \qquad \| \mc{I}_0 V^1_n \mc{I}_0 V^2_n \|_{B(\underline{S}^s \to S^s)} \to 0. $$
Thus we finish the case of scale separation.

\subsubsection{Space-time separation}\label{subsubsec:st-sep}
It remains to consider the case $\sigma^1_n=\sigma^2_n=:\sigma_n$.
Here we approximate the wave profile by decomposition in time: For any $\fy\in L^2$ and $\e\in(0,1)$, there is $\chi\in\cS(\R)$ such that
\EQ{ \label{eqn:cutoff prop}
 \supp\mathcal{F}\chi\subset[-\e,\e],\pq |\chi|\le 1,\pq \|(1-\chi(t))e^{it|\nabla|}\fy\|_{L^\I_t\dot H^{-1}_4}<\e.}
Since $\e<1$, we have
\EQ{
 \supp\mathcal{F}_{t,x}[\chi(t)e^{\pm it|\nabla|}\fy]  = \supp[\hat\chi(\ta)*\de(\ta\pm|\x|)\hat\fy(\x)]
 \subset\{|\ta|\le|\x|+1\}}
and the same inclusion for $\supp\mathcal{F}_{t,x}[(1-\chi(t))e^{\pm it|\nabla|}\fy]$, so that we can keep using the same non-resonance property for the free waves after the cut-off by $\chi(t)$.
In particular, for any free wave $v = \Re(V) = \Re( e^{it|\nabla|} \varphi)$, an application of \eqref{eqn:thm main schro est nonres:goal} implies that
\EQ{ \label{unif bd cutin t}
 \sup_n \sup_{\epsilon < 1} \|\fg_n \chi(t) v\|_{B(S^s \to N^s)} \lesa \| V \|_{L^\infty_t L^2_x}   }
while Theorem \ref{thm:besov gain} gives $\theta>0$ such that
\EQ{
 \sup_n\|\mc{I}_0\fg_n(1-\chi(t))v\|_{B(\underline{S}^s \to S^s)} \lesa \| (1-\chi(t) )  V \|_{L^\infty_t B^{-1}_{4, \infty}}^\theta \| V \|_{L^\infty_t L^2_x}^{1-\theta}. }
Consequently, decomposing the profiles $V^j=\chi V^j+(1-\chi)V^j$ using the properties \eqref{eqn:cutoff prop} imply that it suffices to show that
    $$ \|\mc{I}_0 [\fg^1_n\chi v^1]\mc{I}_0[\fg^2_n\chi v^2]\|_{B(\underline{S}^s \to S^s)} \to 0 \pq(n\to\I)$$
where as above, $v^j = \Re(V^j)$. This can be reduced further after observing that another application of Theorem \ref{thm:besov gain} implies that we have the high-low gain
    $$ \big\|P_{\lambda_0} \mc{I}_0\big( [\fg^1_n\chi v^1] \mc{I}_0([\fg^2_n\chi v^2] u_{\lambda_1})\big) \big\|_{S^s}\lesa \Big( \frac{\min\{\lambda_0, \lambda_1\}}{\max\{\lambda_0, \lambda_1\}}\Big)^\theta \| \chi \|_{L^\infty_t} \| V^1 \|_{L^\infty_t L^2_x} \| V^2 \|_{L^\infty_t L^2_x} \| u_{\lambda_1} \|_{\underline{S}^s_{\lambda_1}}$$
and hence it is enough to prove that
\EQ{ \label{eqn:reduced goal}
 \|[\fg^1_n\chi v^1]\mc{I}_0[\fg^2_n\chi v^2]\|_{B(S^0 \to N^0)} \to 0 \pq(n\to\I).}
Since $S^0\subset X_0$ and $Y_0\subset N^0$, we may replace the operator norm with $B(X_0\to Y_0)$. Then by the scaling invariance for $u(t,x)\mapsto \la^{d/2}u(\la^2 t,\la x)$, it suffices to show that
\EQ{
 \|[\ti\fg^1_n\chi v^1]\mc{I}_0[\ti\fg^2_n\chi v^2]\|_{B(X_0\to Y_0)} \to 0,}
where $\ti\fg^j_n$ is defined by
\EQ{
 \ti\fg^j_n V(t,x) := V(t/\sigma^j_n-t^j_n,x-x^j_n).}

Now we use approximation by step functions in time:
\EQ{ \label{time decop}
 \chi(t)e^{it|\nabla|}\fy(x) = \sum_{k=1}^K I_k(t)\fy_k(x)  + R(t,x),}
such that $I_1\etc I_K$ are mutually disjoint bounded intervals, $\fy_k\in L^2(\R^4)$ and $\|R\|_{L^\I_tL^2}<\e$.
Since $\mc{I}_0:Y_0\to X_0$ and $\|F\|_{B(X_0\to Y_0)}\lec\|F\|_{L^\I_tL^2}$, the decomposition \eqref{time decop} yields
\EQ{
 \pt\|[\ti\fg^1_n\chi v^1]\mc{I}_0[\ti\fg^2_n\chi v^2]\|_{B(X_0\to Y_0)}
 \le \sum_{k,l} \|[\ti\fg^1_n I_k(t)\fy_k(x)]\mc{I}_0[\ti\fg^2_n I_l(t)\fy_l(x)\|_{B(X_0\to Y_0)}+O(\e),}
so that we can further reduce to
\EQ{
 T^j_n(t,x):=\ti\fg^j_n I_j(t)\fy^j(x) \implies \|T^1_n\mc{I}_0 T^2_n\|_{B(X_0\to Y_0)}\to 0 \pq(n\to\I),}
for any bounded intervals $I_1,I_2$ and any $\fy^1,\fy^2\in L^2$.

In the case of time separation $|t^1_n-t^2_n|\to\I$, we have $|t^1_n-t^2_n|>2(\sup I_1 + \sup I_2)$ for large $n$, then the dispersive decay estimate yields
\EQ{
 \|T^1_n(t)\mc{I}_0T^2_nu\|_{2_*}
 \pt\lec \|\fy^2\|_{2_*}\int |t-s|^{-d/2}I_1(t/\sigma_n-t^1_n)I_2(s/\sigma_n-t^2_n)\|\fy^1\|_2 \|u(s)\|_2 ds
 \pr\lec \|\fy^2\|_{2_*}\|\fy^1\|_1\|u\|_{L^\I_tL^2}(\sigma_n|t^1_n-t^2_n|)^{1-d/2} I_1(t/\sigma_n-t^1_n).}
Hence, using $d\ge 3$ and $\sigma_n\ge 1$,
\EQ{
 \|T^1_n\mc{I}_0T^2_nu\|_{B(L^\I_tL^2\to L^2_tL^{2_*})}
 \lec \|\fy^2\|_{L^{2_*}}\|\fy^1\|_1|t^1_n-t^2_n|^{1-d/2}\sigma_n^{(3-d)/2} \to 0. }

In the case of space separation $|x^1_n-x^2_n|\to\I$, Lemma \ref{lem:x sep} yields, after discarding the time restriction by intervals,
\EQ{
 \|T^1_n\mc{I}_0T^2_n\|_{B(L^2_tL^{2^*}\to L^2_tL^{2_*})}
 \le\|\fy^1(x-x^1_n)\mc{I}_0\fy^2(x-x^2_n)\|_{B(L^2_tL^{2^*}\to L^2_tL^{2_*})}\to 0.}
Therefore, we conclude that
\EQ{
 \limsup_{n\to\I}\|[\fg^1_n\chi v^1]\mc{I}_0[\fg^2_n\chi v^2]\|_{B(S^0\to N^0)} \lec \e.}
Sending $\e\to+0$, we obtain \eqref{eqn:reduced goal} as required. This completes the proof of \eqref{dich in Duh}.

\subsection{Proof of (A3): The uniform Strichartz in the single profile case}\label{subsec:a3}
The strategy to prove the single profile case is similar to that used in the proof of the radial case \cite{Guo2018}. We use a weight in the frequency for $H^s$ adapted to the profile, so that we can exploit the high-low gains while dealing with the main terms of the same frequency as in the $L^2$ ($s=0$) case, where we obtain the uniform estimate (with respect to concentration of the profile) by approximating the wave potential with a step function in time, thereby reducing the estimate to the static potential case.
Having only one profile makes the first step simpler.

Let $v_n = \Re(V_n)$ be a sequence in the case of a single profile, namely
\EQ{
 V_n = \fg_n V = (\si_n)^2 V(\si_n t, \si_nx), \pq \|V\|_{L^\I_t L^2}<\|W^2\|_2,}
where $V$ is a free wave (independent of $n$), and the translation parameters are removed, since the desired estimate is obviously translation invariant.

The goal of this subsection is to prove
\EQ{ \label{unifS singP}
 \limsup_{n\to\I}\|\mc{I}_{v_n}\|_{B(N^s\to \underline{S}^s)}<\I}
for $0\les s<1$. Since the uniform bound is trivial if $\si_n\equiv 1$, from \eqref{eqn:discard sigma to 0} we may also assume $\si_n\nearrow\I$.
Moreover, since the estimates are stable under small perturbations in $L^\infty_t L^2_x$, \eqref{small pert IV} together with a density argument implies we may assume that $V=e^{it|\na|}g$ with $\hat g\in C^\I(\R^4)$ and
\EQ{
 \x\in\supp\hat g \implies 1/\be<|\x|<\be}
for some $\be>1$. We fix this parameter $\be>1$ large enough so that the following arguments work.
Following \cite{Guo2018}, we introduce a sequence of weight functions $w_n:(0,\I)\to(0,\I)$ adapted to the frequencies $\si_n\to\I$ as follows. With a large parameter $\be>1$ to be fixed, let $\e_n:=(\log\be^2)/(\log(\si_n/\be^2))\to+0$ and
\EQ{
 w_n(r)=\CAS{1 &(0<r\le 1),\\ r^{1+\e_n} &(1\le r\le \si_n/\be^2),\\ \si_n &(\si_n/\be^2\le r \le \si_n\be^2),\\ r/\be^2 &(\si_n\be^2\le r<\I).}}
Then $w_n$ is continuous, increasing, and piecewise logarithmic-linear,
            \begin{equation}\label{eqn:w incres}
                 r/\be^2 \le w_n(r) \le \be^2 r, \pq \lim_{n\to\I}w_n(r)/r=1.
            \end{equation}
Moreover, by \cite[Lemma 5.1]{Guo2018},  for any $0<s<s'$, provided we choose $n\in \NN$ sufficiently large, we have for every $0<r<r'$ the key inequality
            \begin{equation}\label{eqn:w ineq key}
                w_n^s(r) \les w_n^s(r') \les w_n^s(r) \Big( \frac{r'}{r} \Big)^{s'}.
            \end{equation}

The weight $w_n(r)$ is needed to reduce to the case $s=0$.
\begin{lemma}
For $V_n$ as above, suppose that \eqref{unifS singP} holds for $s=0$. Then it holds also for $0\les s<1$.
\end{lemma}
\begin{proof}
In the following argument, we often omit writing explicitly the dependence on $n$. Fix any $0<s<1$ and any sequence $F_n\in N^s$ with $\|F_n\|_{N^s}\le 1$. Let $u=u_n=\mc{I}_{v_n}F_n$. In view of \eqref{eqn:w incres} it suffices to prove that
        \begin{equation}\label{eqn:lem s=0 implies gen:goal}
            \| w^s u \|_{\underline{S}^0} := \Big( \sum_{\lambda \in 2^\NN} w(\lambda)^{2s} \| u_\lambda \|_{\underline{S}^0_\lambda}^2 \Big)^\frac{1}{2} \lesa \Big( \sum_{\lambda \in 2^\NN} w(\lambda)^{2s} \| F_\lambda \|_{N^0_\lambda}^2 \Big)^\frac{1}{2} =: \| w^s F \|_{N^0}
        \end{equation}
where the implied constant is independent of $n$. We decompose $u$ into the frequencies
\EQ{
 u_n = u_{<\si_n/\be^2} + u_{\si_n/\be^2<\cdot<\si_n\be^2} + u_{>\si_n\be^2}
 =: u_L + u_M + u_H. }
Suppose for the moment that there exists $\delta>0$ such that for all sufficiently large $n\in \NN$ we have
        \begin{align}
            \| w^s (v_n u_L) \|_{N^0} + \| w^s (v_n u)_L \|_{N^0} &\lesa \beta^{-\delta} \| V_n \|_{L^\infty L^2_x} \|  w^s u \|_{\underline{S}^0} \label{eqn:lem s=0 implies gen:low1} \\
                \| (v_n u)_L \|_{N^0} &\lesa \beta^{-\delta} \| V_n \|_{L^\infty_t L^2_x} \| u \|_{\underline{S}^0} \label{eqn:lem s=0 implies gen:low2} \\
            \| w^s (v u_H) \|_{N^0} + \| w^s (v_n u)_H \|_{N^0}  &\lesa  \| V_n \|_{L^\infty_t L^2_x} \| w^s F_n \|_{N^0} + \beta^{-\delta} \| V_n \|_{L^\infty_t L^2_x}^2 \| w^s u \|_{\underline{S}^0}
            \label{eqn:lem s=0 implies gen:high}
        \end{align}
where the implied constant is independent of $n$. To bound the low frequency contribution $u_L$, we note that $u_L = \mc{I}_0[ ( F_n + v_n u)_L]$ and hence an application of Lemma \ref{lem:energy ineq} together with \eqref{eqn:lem s=0 implies gen:low1} implies that
        $$ \| w^s u_L \|_{\underline{S}^0} \lesa \| w^s ( F_n + v_n u)_L \|_{N^0} \lesa \| w^s F_n \|_{N^0} + \beta^{-\delta} \| V_n \|_{L^\infty_t L^2_x} \| w^s u \|_{\underline{S}^0}. $$
Similarly, to bound the high frequency contribution, we write $u_H = \mc{I}_0[ ( F_n + v_n u)_H ]$ and again apply Lemma \ref{lem:energy ineq} together with \eqref{eqn:lem s=0 implies gen:high}
        $$ \| w^s u_H \|_{\underline{S}^0} \lesa \| w^s( F_n + v_n u)_H \|_{N^0} \lesa ( 1+ \| V_n\|_{L^\infty_t L^2_x}) \| w^s F_n \|_{N^0}  + \beta^{-\delta} \| V_n\|_{L^\infty_t L^2_x}^2 \|w^s u \|_{\underline{S}^0}.$$
Finally, to bound the remaining medium frequency contribution $u_M$, we apply the assumed uniform Strichartz bound. Let $C_S>0$ be the best uniform Strichartz constant at $s=0$, namely
\EQ{
 C_S := \sup_{n}\|\mc{I}_{v_n}\|_{B(N^0\to S^0)}. }
Since $u_M$ satisfies equation
        $$ i\p_t u_M + \De u_M - v_n u_M = (v_nu_n+F)_M-v_nu_M = F_M + v_n u_H - (v_n u)_H + (v_n u_L)_M - (v_n u_M)_L$$
applying the uniform Strichartz estimate, the bounds \eqref{eqn:lem s=0 implies gen:low1}, \eqref{eqn:lem s=0 implies gen:low2}, and \eqref{eqn:lem s=0 implies gen:high}, the bound \eqref{eqn:w ineq key} and the Fourier support assumption on $V_n$, we conclude that
    \begin{align*}
      \| w^s u_M \|_{\underline{S}^0} &= \sigma_n^s \| u_M \|_{\underline{S}^0} \\
       &\les C_S \sigma_n^s \Big( \| F_M \|_{N^0} + \| v_n u_H \|_{ N^0} + \| (v_n u)_H \|_{N^0} + \| (v_n u_L)_M \|_{N^0} + \| (v_n u_M)_L \|_{N^0} \Big)\\
        &\lesa C_S \Big( \| w^s F \|_{N^0} + \| w^s (v_n u_H) \|_{ N^0} + \|  w^s (v_n u)_H \|_{N^0} + \| w^s (v_n u_L) \|_{N^0} + \| (v_n \sigma_n^s u_M)_L \|_{N^0}\Big) \\
        &\lesa C_S  ( 1 + \| V_n \|_{L^\infty_t L^2_x}) \| w^s F_n \|_{N^0} + \beta^{-\delta} C_S  \| V_n \|_{L^\infty_t L^2_x}^2 \| w^s u \|_{\underline{S}^0}
    \end{align*}
where again the implied constant is independent of $n$. Therefore, since $\sup_n \| V_n \|_{L^\infty_t L^2_x} < \| W^2 \|_{L^2_x}$, provided we choose $\beta>1$ sufficiently large so that $\beta^{-\delta} C_S \ll 1$,  the required bound \eqref{eqn:lem s=0 implies gen:goal} follows.

It only remains to prove the inequalities \eqref{eqn:lem s=0 implies gen:low1}, \eqref{eqn:lem s=0 implies gen:low2}, and \eqref{eqn:lem s=0 implies gen:high}. Note that all the interaction terms on the left have frequency gaps between the high and low frequencies of lower ratio bound $\be$: $v_n u_L$ is high-low to high, $(v_n u)_L$ is high-high to low, while $v_n u_H$ and $(v_n u)_H$ are low-high to high. This frequency gap is exploited via the high-low gain in Theorem \ref{thm:besov gain} to give the small factor $\beta^{-\delta}$. To prove \eqref{eqn:lem s=0 implies gen:low1}, we begin by noting that the Fourier support assumption on $V_n$, together with Theorem \ref{thm:besov gain} and \eqref{eqn:w ineq key} implies that
    \begin{align*}
      \|w^s (v_n u_L) \|_{ N^0} &\lesa \sum_{\frac{\sigma_n}{4\beta} \les \lambda\les 4 \beta \sigma_n} \sum_{\mu\les \frac{\sigma_n}{\beta^2} } w^{s}(\lambda) \| P_\lambda( v_n u_\mu) \|_{N^0} \\
                                &\lesa \sum_{\frac{\sigma_n}{4\beta} \les \lambda\les 4 \beta \sigma_n} \sum_{\mu\les \frac{\sigma_n}{\beta^2} } w^{s}(\lambda) \Big( \frac{\mu}{\lambda} \Big)^{s+2\delta} \|V_n \|_{L^\infty_t L^2_x} \| u_\mu \|_{\underline{S}^0} \\
                                &\lesa \|V_n \|_{L^\infty_t L^2_x} \| w^s  u\|_{\underline{S}^0} \sum_{\frac{\sigma_n}{4\beta} \les \lambda\les 4 \beta \sigma_n} \sum_{\mu\les \frac{\sigma_n}{\beta^2} } \Big( \frac{\mu}{\lambda} \Big)^{\delta} \approx \beta^{-\delta}  \|V_n \|_{L^\infty_t L^2_x} \| w^s  u\|_{\underline{S}^0} .
    \end{align*}
Similarly,
    \begin{align*}
      \|w^s (v_n u)_L \|_{ N^0} &\lesa \sum_{\lambda \les \frac{\sigma_n}{\beta^2} } \sum_{\frac{\sigma_n}{4\beta} \les \mu \les 4 \beta \sigma_n} w^{s}(\lambda) \| P_\lambda( v_n u_\mu) \|_{N^0} \\
                                &\lesa \sum_{\lambda \les \frac{\sigma_n}{\beta^2} }\sum_{\frac{\sigma_n}{4\beta} \les \mu \les 4 \beta \sigma_n} w^{s}(\lambda) \Big( \frac{\lambda}{\mu} \Big)^{2\delta} \|V_n \|_{L^\infty_t L^2_x} \| u_\mu \|_{\underline{S}^0} \\
                                &\lesa \|V_n \|_{L^\infty_t L^2_x} \| w^s  u\|_{\underline{S}^0} \sum_{\lambda \les \frac{\sigma_n}{\beta^2} } \sum_{\frac{\sigma_n}{4\beta} \les \mu \les 4 \beta \sigma_n}  \Big( \frac{\lambda}{\mu} \Big)^{\delta} \approx \beta^{-\delta}  \|V_n \|_{L^\infty_t L^2_x} \| w^s  u\|_{\underline{S}^0} .
    \end{align*}
Thus \eqref{eqn:lem s=0 implies gen:low1} follows. The proof of \eqref{eqn:lem s=0 implies gen:low2} is identical. Finally, the proof of the remaining bound \eqref{eqn:lem s=0 implies gen:high} requires a little more work to obtain the gain of $\beta^{-\delta}$ due to the fact we only have a low-high to high gain in Theorem \ref{thm:besov gain} when bounding the Duhamel operator $\mc{I}_0$ in $S^s$. To deal with this technical issue, one option is to simply plug in the equation once more, and exploit the fact that since the potential $V_n$ has very low frequencies compared to $u_H$, the Fourier support condition is essentially preserved. To make this precise, we take
        $$ u_{\ti{H}} = u_{>\frac{1}{4} \sigma_n \beta^2}.$$
Thus $u_{\ti{H}}$ is a slight widening of the Fourier support of $u_H$. An application of Theorem \ref{thm:main schro est nonres} together with the Fourier support condition on $V_n$ and \eqref{eqn:w ineq key} gives
    \begin{align*}
              \| w^s (v u_H) \|_{N^0} + \| w^s (v_n u)_H \|_{N^0} &\lesa \Big( \sum_{\mu > \frac{1}{4} \sigma_n \beta^2}  w^{2s}(\mu) \sum_{ \frac{1}{4} \mu \les \lambda \les 4\mu} \| P_{\lambda}(v_n u_\mu) \|_{N^0}^2\Big)^\frac{1}{2} \\
              &\lesa \| V_n \|_{L^\infty_t L^2_x} \Big( \sum_{\mu > \frac{1}{4} \sigma_n \beta^2}  w^{2s}(\mu)   \| u_\mu \|_{S^0}^2\Big)^\frac{1}{2} \\
              &\approx    \| V_n \|_{L^\infty_t L^2_x} \| w^s u_{\ti{H}} \|_{S^0}.
    \end{align*}
On the other hand, since $u_{\ti{H}} = \mc{I}_0[ (F + v_n u)_{\ti{H}}]$, again using the Fourier support assumption on $V_n$ together with Theorem \ref{thm:besov gain} and Lemma \ref{lem:energy ineq} (to deal with the forcing term $F_n$), we see that provided $\beta \gg 1$
    \begin{align*}
      \| w^s u_{\ti{H}} \|_{S^0}
      &\lesa \| w^s F_n \|_{N^0} + \sum_{\frac{1}{\beta} \sigma_n \les \mu \les \sigma_n \beta}   \sum_{\lambda > \frac{1}{8} \sigma_n \beta^2} w^{s}(\lambda) \sum_{\frac{1}{4}\lambda \les \lambda_0 \les 4 \lambda} \|\mc{I}_0[ P_{\lambda_0}(v_\mu u_\lambda)] \|_{S^0_{\lambda_0}} \\
      &\lesa \| w^s F_n \|_{N^0} + \| V_n \|_{L^\infty_t L^2_x} \sum_{\frac{1}{\beta} \sigma_n \les \mu \les \sigma_n \beta}  \sum_{\lambda > \frac{1}{8} \sigma_n \beta^2} w^{s}(\lambda) \Big(\frac{\mu}{\lambda}\Big)^{\delta} \| u_\lambda \|_{\underline{S}^0} \\
      &\lesa \| w^s F_n \|_{N^0} + \beta^{-\delta}  \| V_n \|_{L^\infty_t L^2_x}  \| w^s u \|_{\underline{S}^0}.
    \end{align*}
Therefore \eqref{eqn:lem s=0 implies gen:high} follows.
\end{proof}

Thus we are left with the proof of the case $s=0$. We start by proving a decomposability lemma.

\begin{lemma}\label{lem:decom for IV}
There exists a constant $C>0$ such that for any $t_0\in \R$ and
any partition $\R = \cup_{j=1}^N I_j$ and any free wave $v = \Re( e^{\pm it|\nabla|} g) \in L^\infty_t L^2_x$ we have
        $$ \| \mc{I}_v \|_{B(N^0 \to \underline{S}^0)} \les C (1+\|v\|_{L^\I_tL^2_x})^2 \sum_{j=1}^N \| \mc{I}_v \|_{B(Y_0(I_j) \to X_0(I_j))} $$
where we define
        $$ \| \mc{I}_v \|_{B(Y_0(I_j) \to X_0(I_j))} = \sup_{ \| F \|_{Y_0(I_j)} \les 1 } \| \mc{I}_v[ \ind_{I_j} F] \|_{X_0(I_j)}. $$
\end{lemma}
\begin{proof}
By the automatic upgrading \eqref{auto up}, together with the
boundedness $Y_0\subset N^0$, $\underline{S}^0\subset X_0$,
$\mc{I}_0:N^0\to\underline{S}^0$ and the boundedness of the
multiplication operator $v:X_0\to Y_0$, it suffices to prove
\EQ{
 \|\mc{I}_v\|_{B(Y_0\to X_0)} \les 3 \|C_j\|_{\ell^1}}
where we define  $C_j := \| \mc{I}_v \|_{B(Y_0(I_j) \to X_0(I_j))}$. By translation invariance, we may assume that $t_0=0$ and label the partition as $I_j=[T_{j-1},T_j)$. The bound $ \| \mc{I}_v \|_{B(Y_0(J) \to X_0(J))}\les \| \mc{I}_v \|_{B(Y_0(I) \to X_0(I))}$ for any intervals $J\subset I$ implies that we may freely add additional points to the partition, in particular, we may assume that $I_{N_0} = [0, T_{N_0})$ for some $1<N_0\les N$ (i.e. $t_0=0$ belongs to the partition). Before proceeding further, let us recall the $TT^*$ argument on the interval $I_j$. Let $(T\fy)(t):=\mc{U}_v(t;0)\fy$ on $t\in I_j$. Then using the unitary property of $\mc{U}_v$ we have (here we use the standard convention that $\int_b^a = - \int_a^b$ if $a< b$)
$$  T^*f=\int_{I_j} \mc{U}_v (0;s)f(s)ds, \qquad   TT^*f = \int_{I_j} \mc{U}_v(t;s)f(s)ds = \int_0^{T_j} \mc{U}_v(t;s) \ind_{I_j}(s) f(s) ds,$$
and
    \begin{align*}
        \|T^*f\|_{L^2_x}^2 &= \int_{I_j} \int_0^t \LR{\mc{U}_v(t;s) (\ind_{I_j} f)(s)|f(t)} dsdt + \int_{I_j} \int_0^s \LR{f(s)| \mc{U}_v(s;t)(\ind_{I_j} f)(t)} dtds  \\
                    &= \LR{\cI_{v} (\ind_{I_j} f) |\ind_{I_j} f}_{t,x} + \LR{ \ind_{I_j} f|\cI_{v}(\ind_{I_j} f)}_{t,x},
    \end{align*}
and consequently for any $j=1, \dots, N$
$$ \|T\|_{B(L^2\to Y_0'(I_j))}^2 = \|T^*\|_{B(Y_0(I_j)\to L^2)}^2 \le 2 \|\cI_{v}\|_{B(Y_0(I_j)\to Y_0'(I_j))} = 2 C_j, $$
where $Y_0'=L^2_t L^{2^*}_x$ is the dual space.
Therefore, unpacking the definition of the operator $T$, we see that
\EQ{ \label{TT*}
    \Big\| \int_{I_j} \mc{U}_v(0; s) f(s) ds \Big\|_{L^2_x} \les (2C_j)^\frac{1}{2} \| f \|_{Y_0(I_j)}, \qquad \| \mc{U}_v(t;0) \varphi \|_{Y_0'(I_j)} \les (2C_j)^\frac{1}{2} \| \varphi \|_{L^2_x}. }
Take any $f\in Y_0$ and let $u$ be the solution of
    $$(i\p_t+\De-v)u=f, \pq u(t_0)=0$$
and for $j=1, \dots, N$ let $u_j$ be the solution to
    $$(i\p_t+\De-v )u_j = \ind_{I_j} f, \pq u_j(t_0)=0.$$
Clearly we can write $u=\sum_{j=1}^N u_j$ and $u_j(t) = 0$ if either  $T_{j-1} \g 0$ and $t\les T_{j-1}$, or $T_j \les 0$ and $t \g T_j$. On the other hand, if $t\g T_j > 0$ or $t\les T_{j-1} < 0$ then we have the identity
        $$ u_j(t) = \mc{U}_v(t;0) \int_{T_{j-1}}^{T_j} \mc{U}_v(0; s) f(s) ds. $$
Therefore,  the fact that $\mc{U}_v$ is a unitary operator, together with an application of \eqref{TT*} gives the bounds
        $$ \| u_j\|_{L^\infty_t L^2_x} = \| u_j \|_{L^\infty_tL^2_x(I_j\times \RR^4)} \les C_j \| f \|_{Y_0(I_j)}$$
and
    $$  \|u_j\|_{Y_0'(I_k)} \le \CAS{   2 C_j^\frac{1}{2} C_k^\frac{1}{2} \| f \|_{Y_0(I_j)} &(k>j\g N_0 \text{ or } k<j< N_0), \\
                                         C_j\|f\|_{Y_0(I_j)} &(k=j),\\
                                         0 &\text{otherwise.}}$$
Hence, using the $L^2_t$ structure in $Y_0$ and $Y_0'$, we obtain
    \begin{align*}
       \| u_j \|_{X_0} \les  C_j \| f \|_{Y_0(I_j)} + 2 \Big( \sum_k C_j C_k \| f \|_{Y_0(I_j)}^2 \Big)^\frac{1}{2} \les  3 \| C_k \|_{\ell^1}^\frac{1}{2} C_j^\frac{1}{2} \| f \|_{Y_0(I_j)}
    \end{align*}
and thus summing up over $j$ we conclude that
    $$ \| u \|_{X_0} \les \sum_j \| u_j \|_{X_0} \les 3 \| C_j \|_{\ell^1} \Big( \sum_j \| f \|_{Y_0(I_j)}^2 \Big)^\frac{1}{2} = 3 \| C_j \|_{\ell^1} \| f\|_{Y_0}. $$
\end{proof}

\begin{proof}[Proof of \eqref{unifS singP} for $s=0$]

Fix the profile $V = e^{it|\nabla|} g$ with $\| g \|_{L^2_x} < \| W^2 \|_{L^2_x}$, and let $v=\Re(V)$. The argument proceeds in the following steps:
    \begin{enumerate}
       \item Choose a nice partition $ \RR = \cup_{j=0}^N I_j$ depending only on $V$.

       \item Let $\sigma_n \to \infty$ be a sequence of positive real numbers, and define $v_n = \sigma_n^2 v(\sigma_n t , \sigma_n x)$ and  $I^n_j = \{ \sigma_n t \in I_j \}$. Show that the choice of partition in the first step implies that
                    \begin{equation}\label{eqn:unif stri conclusion:loc bnds}
                         \limsup_{n\to \infty} \| \mc{I}_{v_n} \|_{B( Y_0(I^n_j) \to X_0(I^n_j))} < \infty.
                    \end{equation}

       \item Apply Lemma \ref{lem:decom for IV} to conclude that
                    $$ \limsup_{n\to \infty} \| \mc{I}_{v_n} \|_{B(N^0 \to \underline{S}^0)} < \infty. $$
    \end{enumerate}

\textbf{Step (i): Construction of the partition $I_j$.} Let $\epsilon>0$ be a small constant, eventually it will be taken to be smaller than various universal constants, and in particular, $\epsilon$ will be independent of $v$ (and $n$). We first choose $I_0 = (-\infty, -T)$ and $I_N = (T, \infty)$ where $T> \epsilon^{-1}$ satisfies
                        $$ \| v \|_{L^\infty_t \dot{B}^{-1}_{4, \infty}(\{2|t|>T\}\times \RR^4)} \les \epsilon.$$
This is always possible by breaking $v$ into a small part, and a $C^\infty_0$ part, and using the dispersive decay of the free wave. The assumption that we are below the ground state implies that for every $t^* \in [-T, T]$ we have by \cite{Mizutani2016}
            $$ \| \mc{I}_{v(t^*)} \|_{B(Y_0 \to X_0)} < \infty. $$
On the other hand, by \eqref{small pert IV} and the $L^2$ continuity of the free wave $v$, for every $t^* \in [-T, T]$ there exists $\delta>0$ such that
            $$ \sup_{|s^* - t^*|<\delta} \| \mc{I}_{v(s^*)} \|_{B(Y_0 \to X_0)} < \infty. $$
By compactness, we then conclude that
            $$ C_* = \sup_{ t^* \in [-T, T] } \| \mc{I}_{v(t^*)} \|_{B(Y_0 \to X_0)}  < \infty. $$
We now choose the (bounded) intervals $I_j = [t_j, t_{j+1} )$ for $j=1, \dots, N-1$ such that $-T = t_1 < \dots < t_{N} = T$ and
            $$  \sup_{t\in I_j} \| v(t) - v(t_j) \|_{L^2_x} \les \epsilon C_*^{-1}. $$

\textbf{Step (ii): The localised bounds. } There are two distinct cases to consider, the unbounded intervals $I_0$ and $I_N$, in which the operator $\mc{I}_{v}$ is a perturbation of the free case $\mc{I}_0$, and the bounded intervals $I_j$ for $j=1, \dots, N-1$ where $\mc{I}_v$ is a small perturbation of the stationary case $\mc{I}_{v(t_j)}$. We start with the more involved case of the unbounded intervals $I_0$ and $I_N$. Choose a smooth cutoffs $\rho, \chi \in C^\infty$ such that\footnote{Let $\eta_j \in C^\infty$ with $\eta_1(t) = 1$ on $|t|<1$ and $\eta_1(t) =0$ for $|t|>2$, and $\widehat{\eta}_2(0) = 1$, $\supp \widehat{\eta}_2 \subset \{ |\tau| \les 1\}$. Define $\rho(t) =  1 - \eta_1(2t/T)$ and $\chi(t) = \eta_2 * \rho(t)$. }
    $$ \rho = \begin{cases}
       1 &\text{ if }|t|>T \\
       0   & \text{ if }     |t|< T/2,
    \end{cases}, \qquad \supp \widehat{\chi} \subset \{|\tau| \les 1\}, \qquad \| \rho - \chi \|_{L^\infty_t(\RR)} \lesa T^{-1} \lesa \epsilon. $$
Expanding the Duhamel formula twice, gives the identity
    \begin{equation}\label{eqn:unif stri conclusion:ident for u}
        u:= \mc{I}_{ (\rho v)_n } F = \mc{I}_0 F + \mc{I}_0 [ (\rho v)_n u ] = \mc{I}_0 F + \mc{I}_0 [ (\rho v)_n \mc{I}_0 F] + \mc{I}_0[(\rho v)_n \mc{I}_0( (\rho v)_n u )]
    \end{equation}
where $(\rho v)_n(t,x) = \rho( \sigma_n t)  \sigma_n^2 v(\sigma_n t, \sigma_n x)$. An application of Theorem \ref{thm:besov gain} together with H\"older's inequality and the usual square function estimate gives
    \begin{align*}
     \| \mc{I}_0[ (\chi v)_n \mc{I}_0( (\rho v)_n u )] \|_{X_0} \lesa \| \mc{I}_0[ (\chi v)_n \mc{I}_0( (\rho v)_n u )] \|_{S^0}  &\lesa \| (\chi v)_n \|_{L^\infty_t \dot{B}^{-1}_{4,\infty}}^\theta \| \chi \|_{L^\infty_t}^{1-\theta} \| v_n \|_{L^\infty_t L^2_x}^{1-\theta} \| \mc{I}_0[ (\rho v)_n u ] \|_{\underline{S}^0} \\
     &\lesa \| (\chi v)_n \|_{L^\infty_t \dot{B}^{-1}_{4,\infty}}^\theta \| \chi \|_{L^\infty_t}^{1-\theta} \| v_n \|_{L^\infty_t L^2_x}^{1-\theta} \| (\rho v)_n u  \|_{Y_0} \\
     &\lesa \| \chi v \|_{L^\infty_t \dot{B}^{-1}_{4,\infty}}^\theta \| \chi \|_{L^\infty_t}^{1-\theta} \| v \|_{L^\infty_t L^2_x}^{2-\theta} \| \rho \|_{L^\infty_t L^2_x} \| u \|_{X_0}
    \end{align*}
where the last line followed by rescaling. On the other hand simply applying H\"older's inequality and the free Strichartz estimate for $\mc{I}_0$ implies that
    \begin{align*}
      \| \mc{I}_0[ ((\rho-\chi) v)_n \mc{I}_0( (\rho v)_n u )] \|_{X_0} &\lesa \| ( ( \rho - \chi) v )_n \|_{L^\infty_t L^2_x} \| \mc{I}_0( (\rho v)_n u ) \|_{X_0} \\
                &\lesa \| \rho - \chi \|_{L^\infty_t } \| \rho \|_{L^\infty_t} \| v \|_{L^\infty_t L^2_x}^2 \| u \|_{X_0}.
    \end{align*}
In view of our assumptions on $T$, $v$, and the cutoffs $\chi$ and $\rho$, we conclude that
    $$ \| \chi v \|_{L^\infty_t \dot{B}^{-1}_{4,\infty}} \lesa \| \rho v \|_{L^\infty_t \dot{B}^{-1}_{4,\infty}} + \| (\chi - \rho) v \|_{L^\infty_t L^2_x}  \lesa \epsilon$$
and thus combining the above bounds we obtain
   \begin{align*}
        \| \mc{I}_0[ (\rho v)_n \mc{I}_0( (\rho v)_n u )] \|_{X_0}
            &\lesa  \| \mc{I}_0[ ((\rho-\chi) v)_n \mc{I}_0( (\rho v)_n u )] \|_{X_0} +  \| \mc{I}_0[ (\chi v)_n \mc{I}_0( (\rho v)_n u )] \|_{X_0} \\
            &\lesa  \epsilon^\theta  \| u \|_{X_0}
   \end{align*}
with the implied constant depending only on $\|v\|_{L^\I_tL^2_x}$, and in particular is independent of $v$, $n$, and the choice of cutoffs $\rho$ and $\eta$. The identity \eqref{eqn:unif stri conclusion:ident for u} then gives
    \begin{align*}
       \| u \|_{X_0} &\lesa \| F \|_{Y_0} + \| (\rho v)_n \mc{I}_0 F \|_{Y_0} + \| \mc{I}_0[ (\rho v)_n \mc{I}_0 ( (\rho v)_n u )] \|_{X_0} \\
                     &\lesa \| F \|_{Y_0} + \| \rho \|_{L^\infty} \| v \|_{L^\infty_tL^2_x} \| F \|_{Y_0} + \epsilon^\theta \| u \|_{X_0} \\
                     &\lesa \| F \|_{Y_0} + \epsilon^\theta \| u \|_{X_0}
    \end{align*}
and therefore, provided that $\epsilon>0$ is sufficiently small (this choice is independent $T$, $n$, the cutoffs $\rho$ and $\eta$, and $v$), we conclude that
        $$ \limsup_{n\to \infty} \| \mc{I}_{(\rho v)_n} F \|_{X_0} \lesa \| F \|_{Y_0}.$$
Since
        $$ \| \mc{I}_{v_n} \|_{B( Y_0(I_0^n) \to X_0(I^n_0))} + \| \mc{I}_{v_n} \|_{B( Y_0(I_N^n) \to X_0(I_N^n))} \les 2 \| \mc{I}_{(\rho v)_n} \|_{B(Y_0 \to X_0)}$$
we finally see that
        $$ \limsup_{n\to \infty}   \| \mc{I}_{v_n} \|_{B( Y_0(I_0^n) \to X_0(I^n_0))} +   \limsup_{n\to \infty} \| \mc{I}_{v_n} \|_{B( Y_0(I_N^n) \to X_0(I_N^n))} < \infty. $$
It remains to deal with the bounded intervals $I^n_j  = [\sigma_n^{-1} t_j, \sigma_n^{-1} t_{j+1})$ with $j=1, \dots, N-1$. But this a direct application of a (localised version) of the small perturbation observation \eqref{small pert IV} after noting that rescaling gives
        $$ \| \mc{I}_{v_n(\sigma_n^{-1} t_j)} \|_{B(Y_0 \to X_0)} = \| \mc{I}_{v(t_j)} \|_{B(Y_0 \to X_0)} \les \sup_{t\in [-T, T]} \| \mc{I}_{v(t)} \|_{B(Y_0 \to X_0)} = C_*$$
and hence
        \begin{align*}
            \| \mc{I}_{v_n(\sigma_n^{-1} t_j)} ( v_n- v_n(\sigma_n^{-1} t_j)) \|_{B( X(I_j^n) \to X(I_j^n))} &\les \| \mc{I}_{v_n(\sigma_n^{-1} t_j)} \|_{B(Y_0 \to X_0)} \| v_n - v_n(\sigma_n^{-1} t_j) \|_{L^\infty_t L^2_x(I^n_j \times \RR^4)} \\
            &\les C_* \| v - v(t_j) \|_{L^\infty_t L^2_x(I_j \times \RR^4)} \les \epsilon
        \end{align*}
where the last line used the choice of the intervals $I_j$. Therefore
    $$ \limsup_{n\to \infty} \| \mc{I}_{v_n} \|_{B(Y_0(I^n_j) \to X_0(I^n_j))} \les ( 1- \epsilon)^{-1} \limsup_{n\to \infty} \| \mc{I}_{v_n(\sigma_n^{-1} t_j)} \|_{B(Y_0 \to X_0)} \les \frac{C_*}{1-\epsilon} < \infty. $$
This completes the proof of the localised bounds \eqref{eqn:unif stri conclusion:loc bnds}.

\textbf{Step (iii): Conclusion of proof.} In view of the localised bounds \eqref{eqn:unif stri conclusion:loc bnds}, and the fact that the number of intervals in the partition $I_j$ is independent of $n$, an application of Lemma \ref{lem:decom for IV} implies \eqref{unifS singP}.
\end{proof}

Thus we have proven (A1)--(A3), thereby completing the proof of Theorem \ref{thm:unif Stz}.
It is worth noting that (A1)--(A3) are independent of the contradiction argument in \S \ref{ss:proof of unif stri}, so they hold true in general.

\section{Local and global well-posedness}\label{sec:gwp-proof}

In this section we give the proof of Theorems \ref{thm:gwp below W} and \ref{thm:lwp below W}. The key point is the refined local well-posedness result which shows that the time of existence is independent of the initial profile, and thus only depends on the size of the norm.

\subsection{Local well-posedness}\label{subsec:lwp} A short computation using Sobolev embedding and the Strichartz estimate gives for any $s>\frac{1}{2}$
        $$ \| e^{it\Delta} f \|_{L^2_t W^{\frac{1}{2}, 4}_x([0, T]\times \RR^4)} \lesa T^{\frac12\min\{1,( s- \frac{1}{2})\}} \| f \|_{H^s}. $$
Consequently Theorem \ref{thm:lwp below W} follows from the following slightly sharper local well-posedness result.

\begin{theorem}\label{thm:lwp sharp}
Let $d=4$ and take $(s, \ell)$ satisfying \eqref{eqn:cond on s l}. Let $0<B< \| W^2\|_{L^2(\RR^4)}$. Then there exists $\epsilon>0$ such that for any $T>0$ and data $(f,g) \in H^s \times H^\ell$ satisfying
            \begin{equation}\label{eqn:thm lwp sharp:cond on data}
                \| e^{it\Delta} f \|_{L^2_t W^{\frac{1}{2}, 4}_x([0, T]\times \RR^4)} \| f \|_{H^\frac{1}{2}(\RR^4)}^7 \les \epsilon, \qquad  \| g \|_{L^2(\RR^4)} \les B
            \end{equation}
there exists a unique solution $(u, V) \in C([0, T]; H^s(\RR^4)\times H^\ell(\RR^4))$ to \eqref{eqn:Zakharov first order} with $u\in L^2_tW^{\frac{1}{2}, 4}_x([0, T]\times \RR^4)$ and $(u, V)(0)=(f,g)$. Moreover the data to solution map is locally Lipschitz continuous.
\end{theorem}

Note that this is a large data result, since given \emph{any} $f\in
H^s$ with $s \g \frac{1}{2}$  there exists a $T>0$ such that $\|
e^{it\Delta} f \|_{L^2_t W^{\frac{1}{2}, 4}_x} < \epsilon$. The key
advantage of Theorem \ref{thm:lwp sharp} over the results in
\cite{Candy2023} is that we can take the time of existence $T$ to be
\emph{independent} of $V$.

\begin{proof}[Proof of Theorem \ref{thm:lwp sharp}]
By the persistence of regularity obtained in \cite[Theorem 8.1]{Candy2023}, it suffices to consider the case $(s, \ell) = (\frac{1}{2}, 0)$. Fix data $(f, g) \in H^\frac{1}{2}\times H^\ell$ and an interval $I=[0, T]$ satisfying \eqref{eqn:thm lwp sharp:cond on data} and define
    $$ V_0 = e^{it|\nabla|} g, \qquad \mc{J}_0[G] =  - i \int_0^t e^{i(t-s)|\nabla|} G(s) ds. $$
For ease of notation, we take
    $$ \| V \|_W = \| V \|_{L^\infty_t L^2_x} + \| (i\p_t + |\nabla|) V \|_{L^2_t H^{-1}_x}, \qquad \|u \|_D = \| u \|_{L^2_t W^{\frac{1}{2}, 4}_x}. $$
The norm $\| \cdot \|_W$ is used to control the wave evolution, while $\| \cdot \|_D$ is simply the endpoint Strichartz space. An application of \cite[Proposition 6.1 and Proposition 6.2]{Candy2023} gives the bounds\footnote{Note that the $S^\frac{1}{2} = S^{\frac{1}{2}, 0, 0}$, $N^{s} = N^{s, 0, 0}$, and $W = W^{0, 0, 0}$ where $S^{\frac{1}{2}, 0, 0}$, $N^{s, 0, 0}$, and $W^{0, 0, 0}$ are the spaces in \cite{Candy2023}.}
    \begin{align}
      \| v  \psi  \|_{N^\frac{1}{2}(I)} &\lesa \| v \|_{W(I)} \Big( \| \psi \|_{S^\frac{1}{2}(I)} \|\psi  \|_{D(I)}\Big)^\frac{1}{2} \label{eqn:thm lwp sharp:schro est}\\
      \| \mc{J}_0[\overline{\varphi} \psi] \|_{W(I)} &\lesa \Big( \|\varphi\|_{S^\frac{1}{2}(I)} \| \psi \|_{S^{\frac{1}{2}}(I)} \Big)^\frac{1}{2} \Big( \| \varphi\|_{D(I)} \| \psi \|_{D(I)} \Big)^\frac{1}{2} \label{eqn:thm lwp sharp:wave est}
    \end{align}
On the other hand, an application of the uniform Strichartz estimate, Theorem \ref{thm:unif Stz}, implies that if
    $$(i\p_t + \Delta - \Re(V_0)) \psi = F$$
with $\psi(0) = f$ then
    \begin{equation}\label{eqn:thm lwp sharp:energy}
        \| \psi \|_{S^{\frac{1}{2}}(I)} \lesa_B \| f \|_{H^{\frac{1}{2}}} + \| F \|_{N^{\frac{1}{2}}}.
    \end{equation}
Writing $\psi  = \mc{U}_{\Re(V_0)} f + \mc{I}_{\Re(V_0)}[F] = e^{it\Delta} f + \mc{I}_0[ \Re(V_0) e^{it\Delta} f] + \mc{I}_{\Re(V_0)}[F]$, an application of \eqref{eqn:thm lwp sharp:schro est}, \eqref{eqn:thm lwp sharp:energy}, and Lemma \ref{lem:energy ineq} gives
    \begin{align}
      \| \psi \|_{D(I)} &\les \| e^{it\Delta} f \|_{D(I)} + \| \mc{I}_0[\Re(V_0) e^{it\Delta}f] \|_{S^\frac{1}{2}(I)} + \| \mc{I}_{\Re(V_0)}[ F] \|_{S^\frac{1}{2}(I)} \notag \\
                        &\lesa_B \| e^{it\Delta} f \|_{D(I)} + \| g \|_{L^2_x}\Big( \| f \|_{H^\frac{1}{2}} \| e^{it\Delta} f \|_{D(I)}\Big)^\frac{1}{2} + \| F \|_{N^{\frac{1}{2}}(I)} \notag \\
                        &\lesa \Big( \| f \|_{H^\frac{1}{2}} \| e^{it\Delta} f \|_{D(I)}\Big)^\frac{1}{2} + \| F \|_{N^{\frac{1}{2}}(I)}. \label{eqn:thm lwp sharp:D energy}
    \end{align}
Note that the implied constants in the inequalities \eqref{eqn:thm lwp sharp:energy}--\eqref{eqn:thm lwp sharp:D energy} potentially depend on $B$ (in using Theorem \ref{thm:unif Stz}), but are otherwise independent of $(f,g) \in H^\frac{1}{2} \times L^2$. We now define a sequence $u_j$ recursively as
    $$ u_j = \mc{U}_{\Re(V_0)} f + \mc{I}_{\Re(V_0)} \big[ \mc{J}_0[ |\nabla| |u_{j-1}|^2 ] u_{j-1} \big]. $$
A standard application of the above bounds shows that provided $\epsilon>0$ is sufficiently small (depending only on the implied constants in \eqref{eqn:thm lwp sharp:schro est} - \eqref{eqn:thm lwp sharp:D energy}), $u_j$ is a Cauchy sequence, and hence converges to a solution $u\in S^\frac{1}{2}(I)$. The wave component is then defined as $V = e^{it|\nabla|} g + \mc{J}_0[ |\nabla||u|^2]$. Uniqueness of solutions satisfying $u \in L^2_{t,loc}(I_T; W^{\frac{1}{2},4}( \RR^4))$ follows from \cite[Theorem 7.7 and 8.1]{Candy2023}. Local Lipschitz continuity follows in the standard way from the estimates above.
\end{proof}

\subsection{Global well-posedness below the ground state}\label{subsec:gwp-proof} The proof of
Theorem \ref{thm:gwp below W} follows from the variational properties
of the ground state $W$ together with the refined local well-posedness
result in Theorem \ref{thm:lwp below W}. More precisely, we exploit
the following result from \cite{Guo2018}, slightly adjusted to our
sign convention.

\begin{lemma}[{\cite[Section 6]{Guo2018}}]\label{lem:en-bound}
Let $f \in \dot{H}^1(\RR^4)$ and $g \in L^2(\RR^4)$ with
        \begin{equation}\label{eqn:constraint}
              E_Z(f,g) < E_Z(W, -W^2) = \frac{1}{4} \| W^2 \|_{L^2(\RR^4)}^2, \qquad \| g \|_{L^2(\RR^4)} \les \| W^2 \|_{L^2(\RR^4)}.
        \end{equation}
Then we have the bounds
    $$ \| g \|^2_{L^2} \les 4 E_Z(f,g),  \qquad  \| \nabla f \|_{L^2}^2
    \les \frac{ \| W^2 \|_{L^2}}{\| W^2 \|_{L^2} - \| g \|_{L^2}}
    \big( 4 E_Z(f,g) - \| g \|_{L^2}^2\big)
    \les \|W^2\|_{L^2}^2$$
\end{lemma}
\begin{proof}
We begin by observing that for any $f \in \dot{H}^1(\RR^4)$ and $g \in L^2(\RR^4)$, the properties of the ground state $W$ imply that
     \begin{align}
           E_Z(f,g) &= \frac{1}{2} \| \nabla f \|_{L^2}^2 + \frac{1}{4} \| g \|_{L^2}^2 +\frac{1}{2} \Re \lr{g, |f|^2} \notag \\
                    &\g \frac{1}{2} \| \nabla f \|_{L^2}^2 + \frac{1}{4} \| g \|_{L^2}^2 - \frac{1}{2} \|g \|_{L^2} \| f \|_{L^4}^2 \notag \\
                    &\g \frac{1}{2} \|W^2\|_{L^2}^{-1} \big( \| W^2 \|_{L^2} - \| g \|_{L^2} \big) \| \nabla f \|_{L^2}^2 + \frac{1}{4} \| g \|_{L^2}^2.
                    \label{eqn:energy chain 1}
        \end{align}
If we assume that in addition $(f,g)$ satisfy the constraints in \eqref{eqn:constraint}, then \eqref{eqn:energy chain 1} immediately gives
    $$ \| g \|_{L^2}^2 \les 4 E_Z(f,g) $$
and similarly
    \begin{align*}
        2 \| \nabla f \|_{L^2}^2 &\les \frac{ \| W^2 \|_{L^2}}{\| W^2 \|_{L^2} - \| g \|_{L^2}} \big( 4 E_Z(f,g) - \| g \|_{L^2}^2\big)\\
            &\les \frac{ \| W^2 \|_{L^2}}{\| W^2 \|_{L^2} - \| g \|_{L^2}} \big( \|W^2 \|_{L^2}^2 - \| g \|_{L^2}^2\big) \les \| W^2 \|_{L^2}\big( \|W^2 \|_{L^2} + \| g \|_{L^2}\big),
    \end{align*}
    which completes the proof.
  \end{proof}

\begin{proof}[Proof of Theorem \ref{thm:gwp below W}]
Suppose now that $(s,\ell)\in\R^2$ satisfy \eqref{eqn:cond on s l} and
$s\ge 1$. Let $(u(0),V(0))\in H^s\times H^\ell$ be given initial data
satisfying
\[
E_Z(u(0),V(0))<\frac14 \|W^2\|_{L^2}^2-\e , \; \|V(0)\|_{L^2}\les \| W^2 \|_{L^2}.
\]
We may restrict to positive times by reversability.
An application of Lemma \ref{lem:en-bound} shows that we must have the
strict inequality $\|V(0)\|_{L^2}<\| W^2 \|_{L^2}$. Let $T^*$ be the
supremum of all $T\g 0$ such that there exists a solution $(u,V)\in
C([0,T],H^s\times H^\ell)\cap L^2([0,T],W^{\frac12,4})$
which conserves mass and energy. Then, we claim that for all $t\in [0,T]$,
\begin{equation}\label{eq:em}E_Z(u(t),V(t))<\frac14 \|W^2\|_{L^2}^2-\e, \;
  \|V(t)\|_{L^2}<\|W^2\|_{L^2}.
\end{equation}
Otherwise, let $t_*$ denotes the infimum of all $t \in [0,T]$ such
that $\|V(t)\|_{L^2}\g \|W^2\|_{L^2}$. By continuity and energy
conservation we must have $\|V(t_*)\|_{L^2}=\|W^2\|_{L^2}$ and Lemma
\ref{lem:en-bound} yields the contradiction
$\|V(t_*)\|_{L^2}<\|W^2\|_{L^2}$, so that \eqref{eq:em} is proved.

Suppose now, for
the sake of contradiction, that $T^*<\infty$. According to
\cite[Theorem 7.6]{Candy2023} there exists a sequence of $T_n>0$ with
$T_n \nearrow T^*$ and the solution $(u,V)\in
C([0,T_n],H^s\times H^\ell)\cap L^2([0,T_n],W^{\frac12,4})$
conserves mass and energy. Lemma \ref{lem:en-bound} implies
\[
\|V(T_n)\|_{L^2}^2\les 4 E_Z(u(T_n),V(T_n))= 4 E_Z(u(0),V(0))<\|W^2\|^2_{L^2}-4\e.
\]
Further,
\[
\|u(T_n)\|_{H^{1}}^2\les \|\nabla
u(T_n)\|_{L^2}^2+\|u(0)\|_{L^2}^2\les \|W^2\|^2_{L^2}+m^2.
\]
Both of these bounds are uniform in $n$.
Therefore, by Theorem \ref{thm:lwp below W} there exists a time
$\tau=\tau (\e,m)$, independent of $n$, such that we can extend the
solution to $(u,V)\in
C([0,T_n+\tau],H^{1}\times L^2)\cap
L^2([0,T_n+\tau],W^{\frac12,4})$. By \cite[Theorem 8.1]{Candy2023} we have $(u,V)\in
C([0,T_n+\tau],H^s\times H^\ell)\cap
L^2([0,T_n+\tau],W^{\frac12,4})$ and mass and energy are
conserved. For large enough $n$ we have $T_n+\tau>T^*$, which is in contradiction to the
definition of $T^*$, and we conclude that $T^*=\infty$. In addition,
\cite[Theorem 8.1]{Candy2023} implies that for any $T<\infty$ the
flow map $ (u(0),V(0))\mapsto (u,V)$ is Lipschitz continuous as a map
from a small ball of initial data in $H^s\times H^\ell$ below the
ground state
to  solutions in
$C([0,T],H^s\times H^\ell)\cap
L^2([0,T],W^{\frac12,4})$.
Also, the argument above and Lemma \ref{lem:en-bound} imply that \[\sup_{t \in[0,\infty)}\|\nabla u(t)\|_{L^2}^2\les
  2\e^{-1}\|W^2\|^2_{L^2}E_Z(u(0),V(0)),\; \sup_{t\in
    [0,\infty)}\|V(t)\|_{L^2}^2 \les 4E_Z(u(0),V(0)).\]
\end{proof}

\subsection*{Acknowledgements}
Financial support by the German Research Foundation (DFG) through the CRC 1283 ``Taming uncertainty and profiting from randomness and low regularity in analysis, stochastics and their
  applications'' is acknowledged. T. C. is supported by the Marsden Fund Council grant 19-UOO-142, managed by Royal Society Te Ap\={a}rangi.

\bibliographystyle{amsplain}
\bibliography{Zakharov}
\end{document}